\documentclass[11pt]{amsart}

\usepackage{amsfonts, amstext, amsmath, amsthm, amscd, amssymb}

\usepackage{graphicx, color}

\usepackage{microtype}

\usepackage[hidelinks,pagebackref,pdftex]{hyperref}

\usepackage[margin=3cm]{geometry}

\newtheorem{theorem}{Theorem}[section]
\newtheorem{proposition}[theorem]{Proposition}
\newtheorem{lemma}[theorem]{Lemma}
\newtheorem{corollary}[theorem]{Corollary}

\newtheorem*{namedtheorem}{\theoremname}
\newcommand{\theoremname}{testing}
\newenvironment{named}[1]{\renewcommand{\theoremname}{#1}\begin{namedtheorem}}{\end{namedtheorem}}

\theoremstyle{definition}
\newtheorem{definition}[theorem]{Definition}

\newtheorem{remark}[theorem]{Remark}
\newtheorem{question}[theorem]{Question}

\title[Lorenz links, Minimal Braids, Positive braids, and Geometric Types]{Lorenz links, T-links, Minimal Braids, Positive braids with a full twist, and Geometric Types}
\author{Thiago de Paiva}
\address{Beijing International Center for Mathematical Research, Peking University, Beijing 100871, China P.R.}
\email{thhiagodepaiva@gmail.com}
 
\begin{document}

\begin{abstract}
Lorenz links and T-links are equivalent families of links by work of
Birman--Kofman. Lorenz links arise as periodic orbits of the Lorenz system,
whereas T-links are closures of certain positive braids. Birman, Williams, and
Franks showed that every Lorenz link admits a positive braid representative
with at least one full twist that realizes the braid index. In this paper, we
study Lorenz links from the perspective of these minimal braid representatives.
We compute them explicitly from the parameters of the corresponding T-links and
show that they are precisely a family of positive braids with a full twist,
which we call V-link braids. This gives an explicit correspondence between
V-links and T-links.

We use this correspondence to study equivalences among T-link presentations.
As a first application, we obtain new criteria for distinguishing T-links
directly from their parameters. We then recover the Birman--Kofman equivalence
between T-link presentations from the V-link/T-link correspondence. We prove
that, under mild assumptions, this equivalence produces two
genuinely distinct T-link presentations of the same link. We also generalize
this phenomenon: under further natural conditions, we produce four distinct T-link presentations of the same T-link.
More generally, we prove that non-uniqueness can be arbitrarily large even when
the braid index is fixed: for a fixed braid index, the number of distinct
V-link and T-link presentations of the same link can be made arbitrarily large.

In contrast, we also initiate the study of uniqueness of T-link presentations.
We prove that the opposite phenomenon occurs: there are infinite families of
Lorenz links which admit a unique T-link presentation, after excluding the obvious destabilization case.
In terms of the Lorenz flow, this gives families of periodic orbits of the
Lorenz system whose associated Lorenz links admit only one T-link parameter
description in their isotopy class, up to these standard destabilizations.

Finally, we use the V-link form to place Lorenz links in the setting of
positive braids with a full twist, where several techniques for detecting geometric type apply. 
We prove general criteria for such braids to be
satellite or hyperbolic and apply them to V-links and T-links. As a consequence,
we obtain new families of satellite and hyperbolic T-links and generalize
several known results.
\end{abstract}

\maketitle
\section{Introduction}

The   meteorologist E. N. Lorenz found the following system of three ordinary differential equations \cite{lorenz1963deterministic} in $\mathbb{R}^{3}$ 
$$\frac{dx}{dt} = 10(y - x), \,\,\, \frac{dy}{dt} = 28x - y - xz, \,\,\, \frac{dz}{dt} = xy - \frac{8}{3}z  $$
when he was trying to discover the equations that govern weather patterns.
The flow of this system has many knotted closed periodic orbits, which are called Lorenz knots. Many of these  knotted  orbits are linked together, which are called Lorenz links. Lorenz links were the first interesting examples of links to appear in this unexpected setting, dynamical systems. They piqued the interest of mathematicians. Since then, many studies have aimed to understand their topology and geometry.

Guckenheimer and Williams~\cite{GuckenheimerWilliams}, and Tucker~\cite{Tucker} proved that we can study Lorenz links by studying links in the Lorenz template, which is an embedded branched surface in $\mathbb{R}^{3}$. In other words, they proved that for every Lorenz link there is an equivalent link in the Lorenz template, and every link in the template is a Lorenz link. Therefore, we can define Lorenz links as links that have a representative that can be embedded into the Lorenz template.

After that, topologists in knot theory started to study Lorenz links with diagrams induced by the Lorenz template. See, for example, Birman and Williams \cite{MR682059}, Williams \cite{MR758900}, El-Rifai \cite{Necessary}, Gomes, Franco and Silva \cite{GomesFrancoSilva:Partial, GomesFrancoSilva:Farey}. Then, Birman and Kofman started studying Lorenz links with different diagrams. More precisely, they proved that every Lorenz link can be described as the closure of a certain positive braid as follows \cite{newtwis}:
$$(\sigma_{1}\sigma_{2}\dots\sigma_{r_{1} - 1})^{s_{1}} (\sigma_{1}\sigma_{2}\dots\sigma_{r_{2} - 1})^{s_{2}}\dots (\sigma_{1}\sigma_{2}\dots\sigma_{r_{n} - 1})^{s_{n}},$$where $2 \leq r_{1}<\dots< r_{n}$ and $s_{i}>0$ for $i \in \lbrace 1, \dots, n \rbrace.$ They denoted the closure of the last braid by $T((r_{1}, s_{1}), \dots, (r_{n}, s_{n}))$ and called it a T-link, where $\sigma_1, \dots, \sigma_{r_n-1}$ are the standard generators of the braid group $B_{r_n}$. Then, many other people studied Lorenz links using this equivalence with T-links, for example, de Paiva and Purcell \cite{de2021satellites, dePaivaPurcell2024}, de Paiva \cite{Lorenzknots, de2022torus, de2022hyperbolic}.


From Thurston's perspective, all non-trivial knots in $S^3$ can be classified into one of three geometric types called torus, satellite, and hyperbolic. 
A torus knot, denoted by T$(p, q)$ with $p, q$ integers, is a knot that can be embedded in the surface $T$ of an unknotted torus, where $p, q$ denote the number of times that T$(p, q)$ wraps along the longitude, meridian, respectively, of $T$. 
Only the two integers $p$, $q$ are enough to classify the torus knot $T(p, q)$. 
A satellite knot $K$ is a knot that can be embedded in the tubular neighbourhood of another knot $C$. The boundary of the tubular neighbourhood of $C$ is a surface, which is called an essential torus, in the exterior of $K$ in $S^3$. The essential torus provides additional information that helps classify the satellite knot $K$. 
Finally, a hyperbolic knot is a knot whose complement in $S^3$ admits a complete hyperbolic metric. 

Mostow-Prasad rigidity says that when a knot is hyperbolic, all hyperbolic structures in its complement are isometric. 
Thus, properties defined based on the hyperbolic metric are invariants of the hyperbolic knot. They are called hyperbolic invariants.
Examples of hyperbolic invariants are volume and systole.
To classify a hyperbolic knot we study its hyperbolic invariants.
As a consequence, we also obtain a procedure for distinguishing knots.
But if we want to apply geometry to study a family of knots, the first question we need to answer is the following.

\begin{question}
Given a family of knots, which are torus, satellite, hyperbolic knots?
\end{question}

In this paper, we study Lorenz links from the perspective of their minimal
braid diagrams.  This viewpoint leads to an explicit correspondence between
T-links and minimal V-link representatives.  We use this correspondence to
obtain a parameter-counting obstruction to isotopy among T-links, to recover
and extend the Birman--Kofman parameter equivalence, and to study uniqueness
and non-uniqueness of T-link presentations.  Finally, the V-link viewpoint
places Lorenz links in the setting of positive braids with a full twist, where
we prove new geometric-type criteria and generalize several known results for
T-links and Lorenz links.

The starting point is a theorem of Birman and Williams, which states that every
Lorenz link admits a positive braid representative with at least one full twist
\cite[Theorem~5.1]{MR682059}; by a result of Franks and Williams, this
representative realizes the braid index \cite[Corollary~2.4]{Franks}. Using a
different approach, de Paiva proved that every T-link admits such a positive
full-twist representative, although without explicitly computing the associated
minimal braid from the T-link parameters \cite[Theorem~3.5]{Lorenzknots}. We
make this correspondence explicit. Starting from the parameters of a T-link, we
compute its minimal braid representative and show that these minimal
representatives are precisely the braids in a family that we call V-link braids.

V-link braids are positive braids with at least one full twist, defined as
follows: let \(p,q\) be positive integers with
\(2\le p\le q\).  Let
\[
2\le u_1<\cdots<u_m\le p
\qquad\text{and}\qquad
2\le r_1<\cdots<r_n<p
\]
be possibly empty sequences, and let
\[
v_1,\dots,v_m,s_1,\dots,s_n
\]
be positive integers.  The corresponding V-braid is the positive braid on
\(p\) strands
\begin{align*}
&(\sigma_{p-1}\sigma_{p-2}\dots\sigma_{{p-u_1+1}})^{v_1}\dots(\sigma_{p-1}\sigma_{p-2}\dots\sigma_{{p-u_m+1}})^{v_m}(\sigma_1\sigma_2\dots\sigma_{r_1-1})^{s_1}\dots(\sigma_{1}\sigma_{2}\dots\sigma_{r_n-1})^{s_n}\\
&(\sigma_{1}\sigma_{2}\dots\sigma_{{p-1}})^{q}.
\end{align*}
If one of the sequences is empty, the corresponding product is omitted.  The
closure of this braid is called the V-link
\[
V((u_1,\overline{v_1}),\dots,(u_m,\overline{v_m}),
(r_1,s_1),\dots,(r_n,s_n),(p,q)).
\]
Since \(q\ge p\), the final block contains at least one positive full twist on
all \(p\) strands.

Then we find an explicit relation for obtaining an element of one of these two families, V-links and T-links, from the other given by the following theorem.

\begin{named}{Theorem~\ref{equivalenceTV-links}}
Minimal braids of T-links are the same as V-link braids. 
In other words, a minimal braid of a T-link represents a V-link and a V-link represents a minimal braid of a T-link.
Furthermore, they are related as follows: the V-link
$$V((u_1,\overline{v_1}), \dots, (u_{m},\overline{v_{m}}), (r_1,s_1), \dots, (r_{n},s_{n}), (p, q))$$ 
is equivalent to the T-links
\begin{align*}
&T( (r_1,s_1), \dots, (r_{n},s_{n}), (p, q-u_m), (p+ v_m, u_m - u_{m-1}), \dots, (p+v_m+\dots + v_{2}, u_2-u_1),\\
&(p+ v_m+ \dots + v_{1}, u_1))\text{ and }T((u_1, v_1), \dots, (u_{m}, v_{m}), (q, p-r_n), (q+ s_n, r_n - r_{n-1}), \dots,\\
&(q+s_n+\dots + s_{2}, r_2-r_1), (q+ s_n+ \dots + s_{1}, r_1)).
\end{align*}
\end{named}

As a first application of Theorem~\ref{equivalenceTV-links}, we obtain a
parameter-counting obstruction to isotopy among T-links. Since T-links and
Lorenz links are equivalent families, this gives an explicit criterion for
distinguishing Lorenz links directly from their T-link parameters; see
Theorem~\ref{thm:parameter-counting-obstruction}. More precisely, if
\[
L=T((r_1,s_1),\dots,(r_n,s_n))
\qquad\text{and}\qquad
L'=T((r'_1,s'_1),\dots,(r'_m,s'_m))
\]
are T-links with $s_n, s'_m>1$ and
\[
n>2br(L')-2,
\]
where $br(L')$ is the braid index of $L'$, then \(L\) and \(L'\) are not isotopic. In particular, since \(br(L')\le r'_m\),
it is enough to require
\[
n>2r'_m-2;
\]
see Corollary~\ref{cor:parameter-counting-obstruction-rm}. This provides an
explicit and easily checkable way to distinguish many T-links, and hence many
Lorenz links, from their integer parameters. To the best of our knowledge, this is the first explicit parameter-counting
criterion for distinguishing T-links, and hence Lorenz links, directly from
their T-link parameters.

The correspondence also sheds light on the non-uniqueness of T-link
presentations.  We recover the Birman--Kofman parameter equivalence as the
T-link-level expression of a single V-link presentation: the two T-link
presentations associated to a V-link are exchanged by the Birman--Kofman
transformation; see Corollary~\ref{cor:BK}.  We then determine simple
conditions under which these two presentations are genuinely distinct after
the standard simplifications; see
Proposition~\ref{prop:when-two-T-presentations-coincide}.  Using the centrality
of the full twist, we further generalize this equivalence and show that, under
natural non-degeneracy assumptions, one can obtain four distinct T-link
presentations of the same Lorenz link; see
Theorem~\ref{thm:four-distinct-T-link-presentations}.  Finally, we prove that
this non-uniqueness is not caused merely by increasing the braid index: even
with the braid index fixed, the number of distinct V-link and T-link
presentations of the same link can be made arbitrarily large; see
Theorem~\ref{thm:fixed-braid-index-many-presentations}.

The correspondence also reveals a complementary uniqueness phenomenon.  We
prove that there are infinite families of Lorenz links with a unique T-link
presentation: after excluding the obvious destabilization case
\(s_n=1\), these links admit only one T-link presentation; see
Theorem~\ref{thm:unique-T-link-presentations}.

These results show that the V-link viewpoint is useful for understanding when
different parameter presentations define the same Lorenz link.  The usefulness
of Theorem~\ref{equivalenceTV-links} is not limited to parameter equivalences:
it also plays an important role in the study of geometric types.  Its main
advantage is that it places every T-link in a braid-index-realizing positive
braid position with an explicit full twist.

This position is especially well suited for studying the geometry of the link
complement.  Several techniques for detecting torus, satellite, and hyperbolic
knots from braid diagrams apply most effectively to positive braids containing
a full twist.  In this setting, the braid axis provides a natural way to control
essential surfaces in the complement.  Thus the V-link description transforms
the problem of determining geometric types of Lorenz links into a problem about
positive braids with a full twist, where essential tori and annuli can be
studied using the braid axis.

We next use the V-link viewpoint for a second purpose: the study of geometric
types of Lorenz knots.
There has been some progress in this direction, but previous work on the
geometric types of T-knots has mainly focused on two particular subfamilies. They are T-knots obtained by adding full twists on torus knots and T-knots whose standard T-link presentations already contain at least two full twists.

T-knots obtained from full twists on torus knots have the form
\[
K=T((r_1,r_1s_1),\dots,(r_n,r_ns_n),(p,q)),
\]
where \(\gcd(p,q)=1\). When \(q\ge r_n\), this knot is naturally represented
as a V-link of the form
\[
V((r_1,s_1),\dots,(r_n,s_n),(\min\{p,q\}, \max\{p,q\})).
\]
Otherwise, the corresponding minimal V-link representative has the form
\[
V((q,\overline{p-r_n}),(r_1,r_1s_1),\dots,(r_n,r_ns_n+q)),
\]by Theorem~\ref{equivalenceTV-links}.
Thus, T-knots obtained from full twists on torus knots form a rather special subfamily from
the V-link point of view. The classification of when such knots are torus
knots is almost complete; see~\cite{de2022torus}. Some sufficient conditions for them
to be satellite or hyperbolic have also been found \cite{de2021satellites, unexpected, dePaivaPurcell2024}. However,
the complete geometric classification of this family is still open. The criteria developed in this
paper are not primarily aimed at improving the known classification of this
subfamily; rather, they apply more naturally to the broader class of minimal
V-link representatives with a full twist.

The second subfamily consists of T-knots whose standard T-link presentation
contains at least two full twists. These have the form
\[
T((r_1,s_1),\dots,(r_{n-1},s_{n-1}),(r_n,s_n)),
\]
with $s_n\ge 2r_n.$ For this class, there are some results detecting hyperbolic \cite[Corollary 1.3]{MR4494619} and satellite knots \cite[Theorem 5.4]{dePaivaPurcell2024}.

In this paper, we investigate the geometric types of Lorenz links from the
minimal braid perspective. Instead of requiring the original T-link presentation
to contain full twists, we pass to the minimal V-link representative, which is a
positive braid with a full twist by Theorem~\ref{equivalenceTV-links}. This is
a larger and more natural setting for Lorenz links, since every T-link admits
such a minimal representative. In particular, the results in this paper extend the
known satellite and hyperbolicity criteria for T-links whose standard T-link
presentation contains at least two full twists. 


For the satellite case, we study in Section~\ref{section6} when V-links are
satellite links. We give two sets of sufficient conditions on the parameters of V-links that guarantee they are satellite links.
Applying the correspondence between V-links and T-links, we obtain new families
of satellite T-links. These families extend previously known satellite criteria
for T-links. For example, the following corollary of
Theorem~\ref{satellitecaseone} extends
\cite[Theorem~5.4]{dePaivaPurcell2024}.

\begin{named}{Corollary~\ref{satellitecase1}}
Suppose that there are positive integers $i, j,$ and $d>1$ such that 
$r_1, \dots, r_i,$ $u_1, \dots, u_j$ are less than or equal to $d$ and 
$r_{i+1}, \dots, r_n, s_{i+1}, \dots, s_n, u_{j+1}, \dots, u_m, v_{j+1}, \dots, v_m,$ $p, q$ are multiples of and greater than $d$. 
Then, the T-links 
\begin{align*}
&T((r_1,s_1), \dots, (r_{i},s_{i}), (r_{i+1},s_{i+1}), \dots, (r_{n},s_{n}), (p, q-u_m), (p+ v_m, u_m - u_{m-1}), \dots,\\
&(p+ v_m+\dots + v_{j+1}, u_{j+1} - u_{j}), (p+ v_m+\dots + v_{j}, u_{j}-u_{j-1}), \dots,\\
&(p+ v_m+\dots + v_{2}, u_2-u_1), (p+ v_m + \dots + v_{1}, u_1))\text{ and }
\end{align*}
\begin{align*}
&T((u_1, v_1), \dots, (u_{j}, v_{j}), (u_{j+1}, v_{j+1}), \dots, (u_{m}, v_{m}), (q, p-r_n), (q+ s_n, r_n - r_{n-1}), \dots,\\
&(q+s_n+\dots + s_{i+1}, r_{i+1}-r_{i}), (q+ s_n+\dots + s_{i}, r_{i}-r_{i-1}), \dots,\\
&(q+s_n+\dots + s_{2}, r_2-r_1), (q+ s_n+\dots + s_{1}, r_1))
\end{align*}
are satellite links.
\end{named}

A second satellite criterion is given in
Corollary~\ref{corollarysatellitecase2}, which generalizes
\cite[Theorem~6.1]{de2021satellites}.

For hyperbolicity, we first study positive braids with at least one full twist in Section~\ref{Hyperbolic}.
In Section~\ref{Hyperbolic}, we give several sufficient conditions under which the
closure of such a braid is hyperbolic.  The strategy is to rule out both
satellite and torus knots: we first prove that, under suitable arithmetic and
braid-theoretic hypotheses, the closure is atoroidal, and then prove that it
cannot be a torus knot.  One representative result is the following.

\begin{named}{Theorem~\ref{Hyperbolicity3}}  
Let $p, q, r, k$ be positive integers such that $p>r > q$, $p-r\geq q$, and $r$ doesn't divide $p$. Furthermore, suppose that $gcd(r, q) = 1$, $gcd(p, q) = 1$, or $gcd(p, r) = 1$.  
Let $\beta$ be a positive braid with $r$ strands and at least one full twist on $r$ strands with crossing number different from $(p-1)(r-q)$. Then, the knot $K$ given by the closure of the positive braid $$\beta(\sigma_{1}\dots \sigma_{p-1})^{kp+q}$$ is hyperbolic.
\end{named}

Further hyperbolicity criteria are given in Theorems~\ref{Hyperbolicity1} and~\ref{Hyperbolicity2}.

We then apply these hyperbolicity criteria to V-links and translate the results
to T-links using Theorem~\ref{equivalenceTV-links}. The key point is that the minimal V-link representative can display a full-twist structure that is not visible
in the original T-link presentation. Thus, our criteria apply to T-links beyond
the range of previous full-twist criteria, producing new families of hyperbolic
T-links and extending known hyperbolicity results for T-knots with full twists.
One representative consequence is the following.

\begin{named}{Corollary~\ref{hyperbolicT-links3}}
Let $r_1, \dots, r_n, s_1, \dots, s_n, u_1, \dots, u_m, v_1, \dots, v_m, p, q, k$ be positive integers such that $2\leq r_1< \dots < r_{n} < p$, $2\leq u_1< \dots < u_{m} < p$, $q<p$ with $gcd(p, q) = 1$. Denote $\mathcal{C} = (r_1-1)s_1 + \dots + (r_n-1)s_n+(u_1-1)v_1 + \dots + (u_m-1)v_m$.
\begin{itemize}
\item Suppose $s_n \geq r_{n}> q\geq u_{m}$, $p-r_{n}\geq q$, $r_{n}$ doesn't divide $p$, $gcd(r_{n}, q) = 1$ or $gcd(p, r_{n}) = 1$, and $\mathcal{C}$ is different from
$(p-1)(r_n-q)$, or
\item Suppose $v_m\geq u_{m}>q\geq r_{n}$, $p-u_{m}\geq q$, $u_{m}$ doesn't divide $p$, $gcd(u_{m}, q) = 1$ or $gcd(p, u_{m}) = 1$, and $\mathcal{C}$ is different from
$(p-1)(u_{m}-q)$.
\end{itemize}
Then, the T-knots
\begin{align*}
&T((r_1,s_1), \dots, (r_{n},s_{n}), (p, kp+q-u_m), (p+ v_m, u_m - u_{m-1}), \dots, (p+v_m+ \dots + v_{2}, u_2-u_1), \\
&(p+ v_m+\dots + v_{1}, u_1))\text{ and }T((u_1, v_1), \dots, (u_{m}, v_{m}), (kp+q, p-r_n), (kp+q+ s_n, r_n - r_{n-1}),\\
&\dots,(kp+q+s_n+ \dots + s_{2}, r_2-r_1), (kp+q+ s_n+ \dots + s_{1}, r_1))
\end{align*}
are hyperbolic.
\end{named}

Further hyperbolicity criteria for T-knots are given in
Corollaries~\ref{C1} and~\ref{C2}. In particular,
Corollary~\ref{C1} generalizes \cite[Corollary~1.3]{MR4494619}.

We conclude that minimal braid diagrams provide a useful framework for studying
Lorenz links.  On the one hand, V-link representatives give a rigid
braid-theoretic way to compare parameter presentations.  This allows us to
recover and extend the Birman--Kofman parameter equivalence, to obtain
obstructions to isotopy among T-links, and to prove new uniqueness and
non-uniqueness results for T-link presentations.  In particular, the
parameter-counting obstruction of
Theorem~\ref{thm:parameter-counting-obstruction} and
Corollary~\ref{cor:parameter-counting-obstruction-rm} gives an explicit way to
distinguish many T-links directly from their parameters.  We also show that the
number of distinct V-link and T-link presentations can be made arbitrarily
large even when the braid index is fixed, while some infinite families of
Lorenz links admit a unique T-link presentation after the standard
simplifications.

On the other hand, V-link representatives place Lorenz links in the setting of
positive braids with a full twist, where braid-axis techniques can be used to
detect essential surfaces in the complement.  This allows us to recover and
generalize known results, and to identify new families of satellite and
hyperbolic Lorenz links.

The results of this paper do not give a complete classification of geometric
types of Lorenz knots.  Rather, they provide criteria that apply to broad
families arising from minimal V-link representatives.  The remaining cases
include those whose minimal V-link representatives satisfy none of the torus,
satellite, or hyperbolicity criteria proved here.  They also include the
still-open cases, discussed above, for T-knots obtained by adding full twists
to torus knots.

\subsection{Organisation}

This paper is organized as follows. In Section~\ref{section2}, we prove some
isotopies that are used throughout the paper. In Section~\ref{section3}, we
prove Theorem~\ref{equivalenceTV-links}. The proof follows the main idea of
\cite[Theorem~3.5]{Lorenzknots}, but we simplify the argument using the
isotopies introduced in Section~\ref{section2}. In Section~\ref{section4}, we
use this equivalence to obtain a parameter-counting obstruction to isotopy
among T-links.

In Section~\ref{section5}, we study equivalent V-link and T-link presentations.
We recover the Birman--Kofman parameter equivalence from the V-link/T-link
correspondence, give conditions under which the corresponding T-link
presentations are genuinely distinct, and construct families with many
distinct presentations. We also exhibit infinite families of Lorenz links with
a unique T-link presentation, after excluding the obvious destabilization case.

In Section~\ref{section6}, we study conditions on the parameters of V-links
that guarantee satellite links, and then apply these conditions to T-links
using Theorem~\ref{equivalenceTV-links}. In Section~\ref{Hyperbolic}, we find
conditions under which positive braids with a full twist are hyperbolic. In the
final section, Section~\ref{section7}, we apply the criteria from
Section~\ref{Hyperbolic} to obtain new families of hyperbolic V-links and
T-links, and to generalize some known families.

\subsection*{Acknowledgments}

Part of this work was carried out while the author was a PhD student and later
a research assistant at Monash University, with support from the Australian
Research Council grant DP210103136.  The final version was prepared during the
author's postdoctoral appointment at Peking University.

\section{Isotopies}\label{section2}

In this section we establish some isotopies that we use later.

The first isotopy below was proved in \cite[Proposition 3.2]{Lorenzknots}. We add its proof here because it will be useful for understanding the other isotopies and their proofs.

\begin{proposition}\label{isopoty1}
Consider $\beta$ a braid with $r$ strands and $p, q$ integers such that $0<q\leq r <p$. Then, there is an isotopy that takes the closure of the braid $\beta(\sigma_{1}\dots \sigma_{p-1})^{q}$
to the braid
$$(\sigma_{r-1}\dots \sigma_{r-q+1})^{p-r}\beta(\sigma_{1}\dots \sigma_{r-1})^{q}$$
if $q>1$ and
$\beta(\sigma_{1}\dots \sigma_{r-1})$
if $q=1$.
Moreover, this isotopy happens in the complement of the braid axis of $\beta$.
\end{proposition} 

\begin{proof}
The case $q=1$ immediately follows by applying some destabilization moves on the right of $B = \beta(\sigma_{1}\dots \sigma_{p-1})^{q}$. So we consider that $q>1$.

We start with the $(r-q+1)$-st strand at the top of the braid $B$.
If we follow this strand anticlockwise, then it goes around the braid closure to be at the $(r-q+1)$-st strand at the bottom of the braid. Then, as $q\leq r$, it continues until passes through the sub-braid $(\sigma_1\dots \sigma_{p-1})^{q}$ and finally ends at the $(r+1)$-st strand.
Hence the $(r-q+1)$-st strand and the $(r+1)$-st strand are connected by an under strand that goes one time around the braid closure, as illustrated by the red strand in the second drawing of Figure~\ref{Isopoty1}.  
This means that we can remove one strand from the braid $B$ by shrinking this under strand until it becomes an under strand between the $(r-q+1)$-st strand and the $(r)$-st strand, as shown by the third drawing in Figure~\ref{Isopoty1}. After that, we obtain the braid
$$(\sigma_{r-1}\dots \sigma_{r-q+1})\beta(\sigma_1\dots \sigma_{p-2})^{q}.$$

\begin{figure}
\includegraphics[scale=0.25]{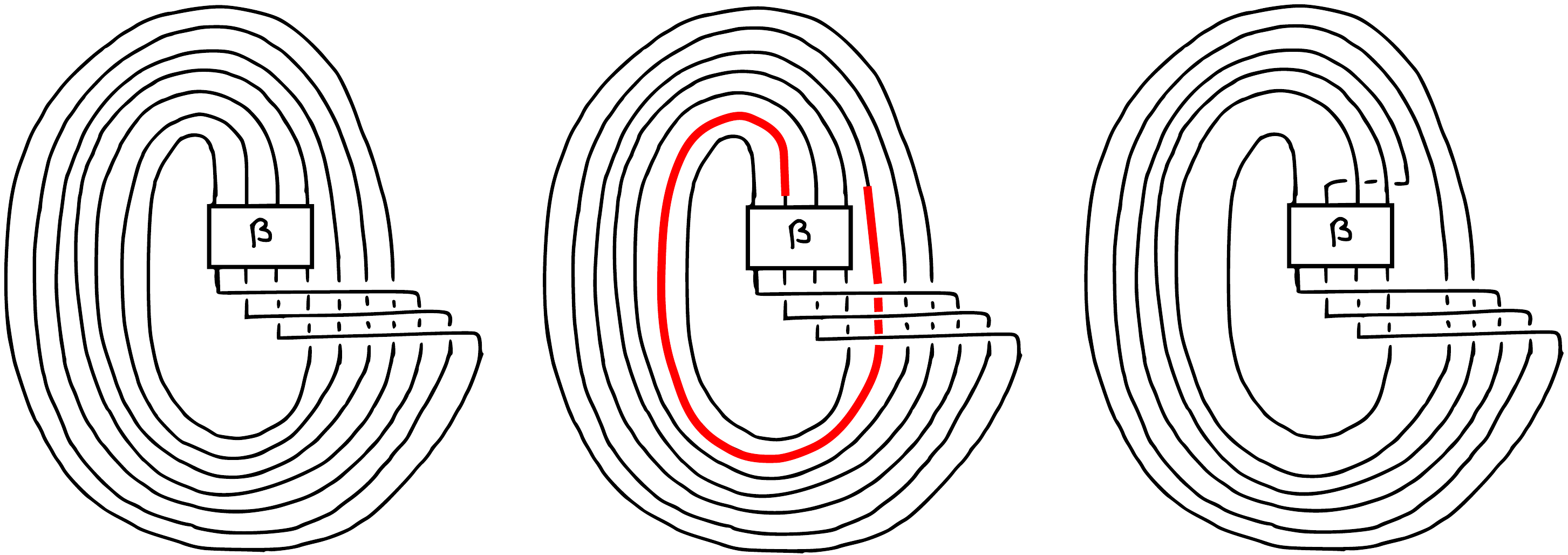} 
\caption{We push down and shrink the red strand in the second link to reduce one strand from the first link. This yields the third link.}
\label{Isopoty1}
\end{figure}

The sub-braid $(\sigma_{r-1}\dots \sigma_{r-q+1})\beta$ has $r$ strands. Then, we can apply the last isotopy again. This yields the braid
$$(\sigma_{r-1}\dots \sigma_{r-q+1})(\sigma_{r-1}\dots \sigma_{r-q+1})\beta(\sigma_1\dots \sigma_{p-3})^{q}.$$
The sub-braid $(\sigma_{r-1}\dots \sigma_{r-q+1})^2\beta$ still has $r$ strands. Then, it continues until we apply this isotopy $p-r$ times in total to obtain the braid
$$(\sigma_{r-1}\dots \sigma_{r-q+1})^{p-r}\beta(\sigma_{1}\dots \sigma_{r-1})^{q}. \qedhere$$
\end{proof}

\begin{proposition}\label{isopoty2}
Consider $\beta$ a braid with $r$ strands and $p, q$ integers such that $1< q\leq r< p$. Then, there is an isotopy that takes the closure of the braid $\beta(\sigma_{p-1}\dots \sigma_{1})^{q}$
to the braid
$$\beta(\sigma_{r-q+1}\dots \sigma_{r-1})^{p-r}(\sigma_{r-1}\dots \sigma_{1})^{q}.$$
\end{proposition} 

\begin{proof}
The $(r+1)$-st strand at the top of the braid anticlockwise follows around the braid closure to be at the $(r+1)$-st strand at the bottom of the braid. It then passes above the sub-braid $(\sigma_{p-1}\dots \sigma_{1})^{q}$ until it ends at the $(r+1-q)$-st strand of the sub-braid $(\sigma_{p-1}\dots \sigma_{1})^{q}$ because $q\leq r$.
Hence the $(r+1)$-st strand of the braid $\beta(\sigma_{p-1}\dots \sigma_{1})^{q}$ is connected to the $(r+1-q)$-st strand of the sub-braid $(\sigma_{p-1}\dots \sigma_{1})^{q}$ by an over strand that goes one time around the braid closure, as illustrated by the red strand in the second drawing of Figure~\ref{Isopoty2}.  
We push  up and shrink this over strand so that we reduce one strand of $\beta(\sigma_{p-1}\dots \sigma_{1})^{q}$, as shown in Figure~\ref{Isopoty2}. After that, this strand becomes an over strand between the $(r-q+1)$-st strand and the $(r)
$-st strand, as shown by the third drawing in Figure~\ref{Isopoty2}, producing the braid 
$$\beta(\sigma_{r-q+1}\dots \sigma_{r-1})(\sigma_{p-2}\dots \sigma_{1})^{q}.$$

The sub-braid $\beta(\sigma_{r-q+1}\dots \sigma_{r-1})$ of the last braid is still a braid on the leftmost $r$ strands, hence the $(r+1)$-st strand is still connected to the $(r+1-q)$-st strand by an over strand that goes one time around the braid closure.
So, we apply the last isotopy again to obtain the braid
$$\beta(\sigma_{r-q+1}\dots \sigma_{r-1})^2(\sigma_{p-3}\dots \sigma_{1})^{q}.$$
\begin{figure}
\includegraphics[scale=0.25]{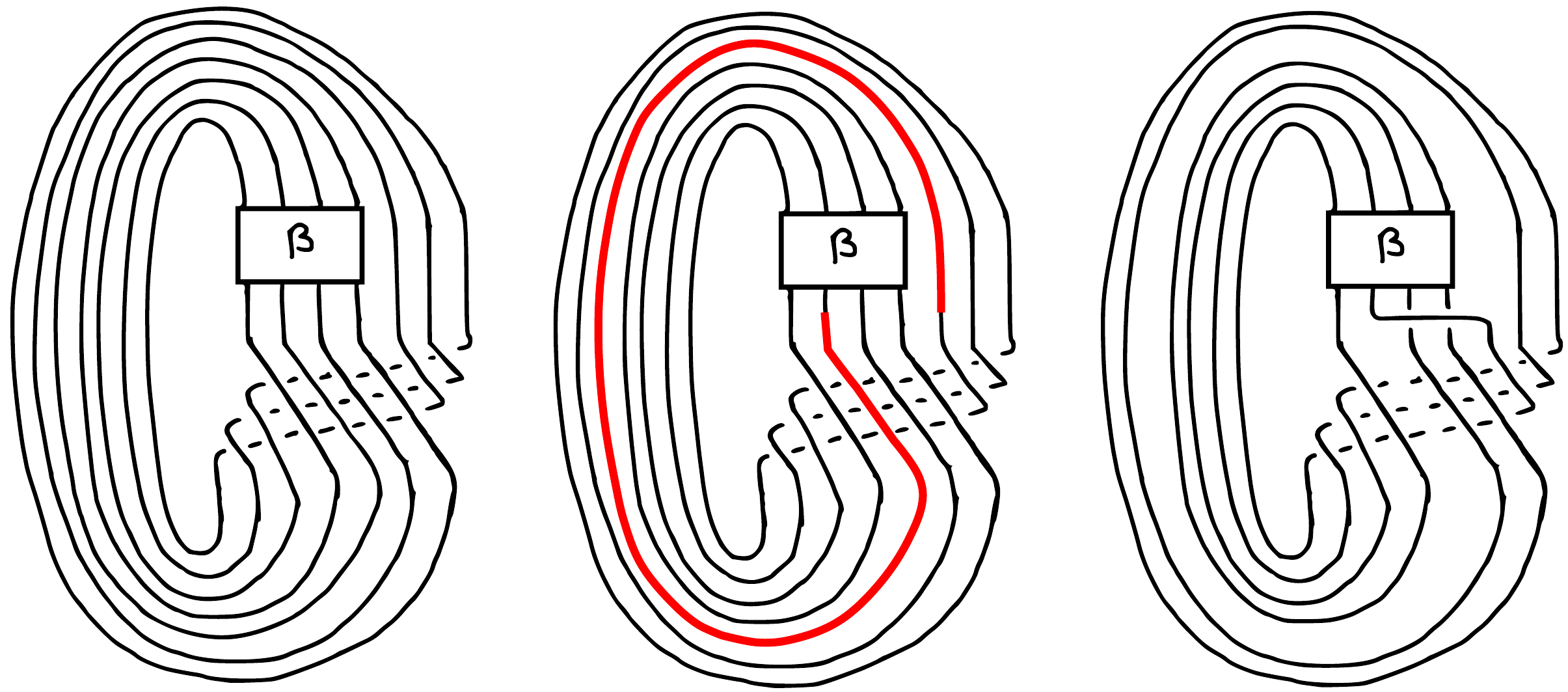} 
\caption{The third link, which is equivalent to the first link, is obtained by shrinking the red strand of the second link.}
\label{Isopoty2}
\end{figure}

We keep applying this isotopy for $p-r$ times in total until we obtain the braid
$$\beta(\sigma_{r-q+1}\dots \sigma_{r-1})^{p-r}(\sigma_{r-1}\dots \sigma_{1})^{q},$$
as we want.
\end{proof}

\begin{proposition}\label{isopoty3}
Consider $\beta$ a braid with $r$ strands and $p, q$ integers such that $1<r\leq q <p$. Then, there is an isotopy that takes the closure of the braid $\beta(\sigma_{1}\dots \sigma_{p-1})^{q}$
to the braid
$$(\sigma_{q-1}\dots \sigma_{1})^{p-q}\beta(\sigma_{1}\dots \sigma_{q-1})^{q}.$$
\end{proposition} 
\begin{proof}
We start with the first strand at the top of the braid.
This strand anticlockwise goes around the braid closure to pass through the sub-braid $(\sigma_{1}\dots \sigma_{p-1})^{q}$ and ends at the $(q+1)$-st strand at the top of the braid because $r\leq q$.
We conclude that the $(1)$-st strand and the $(q+1)$-st strand are connected by an under strand that goes once around the braid closure, as illustrated by the red strand in the second drawing of Figure~\ref{Isopoty3}.
This under strand can be pushed down and  shrunk to reduce one strand of the last braid, as shown in the third drawing of Figure~\ref{Isopoty3}, so that this strand becomes an under strand between the $(1)$-st strand and the $(q)$-st strand yielding the braid 
$$(\sigma_{q-1}\dots \sigma_{1})\beta(\sigma_{1}\dots \sigma_{p-2})^{q}.$$
\begin{figure}
\includegraphics[scale=0.25]{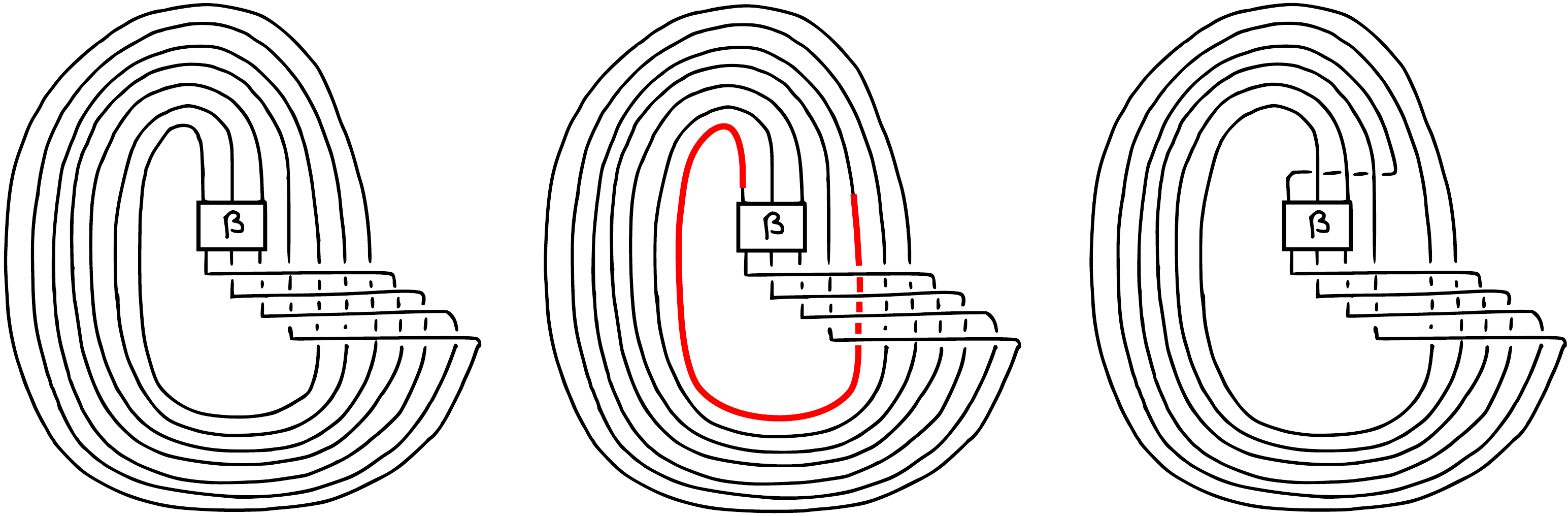} 
\caption{The third link is obtained from the first link by applying the isotopy of shrinking the red strand in the second link.}
\label{Isopoty3}
\end{figure}
The sub-braid $(\sigma_{q-1}\dots \sigma_{1})\beta$ of the last braid is a braid on the leftmost $q$ strands, we similarly conclude  that the $(1)$-st strand is connected to the $(q+1)$-st strand by an under strand that goes once around the braid closure.
We can apply the last isotopy again to obtain the braid
$$(\sigma_{q-1}\dots \sigma_{1})^2\beta(\sigma_{1}\dots \sigma_{p-3})^{q}.$$

We keep applying this isotopy until we apply it $p-q$ times in total to obtain the braid
$$(\sigma_{q-1}\dots \sigma_{1})^{p-q}\beta(\sigma_{1}\dots \sigma_{q-1})^{q}. \qedhere$$
\end{proof} 

\begin{proposition}\label{isopoty3'}
Consider $\beta$ a braid with $r$ strands and $p, q$ integers such that $1<r\leq q <p$. Then, there is an isotopy that takes the closure of the braid $\beta(\sigma_{1}\dots \sigma_{p-1})^{q}$
to the braid $$\beta(\sigma_{q-1}\dots \sigma_{1})^{p}.$$
\end{proposition} 
\begin{proof}
From Proposition~\ref{isopoty3}, there is an isotopy that takes the closure of the braid $\beta(\sigma_{1}\dots \sigma_{p-1})^{q}$ to the braid 
$(\sigma_{q-1}\dots \sigma_{1})^{p-q}\beta(\sigma_{1}\dots \sigma_{q-1})^{q}.$
We note that $(\sigma_{1}\dots \sigma_{q-1})^{q} = (\sigma_{q-1}\dots \sigma_{1})^{q}$. Finally we push the sub-braid $(\sigma_{q-1}\dots \sigma_{1})^{p-q}$ anticlockwise around the braid closure to obtain the braid $\beta(\sigma_{q-1}\dots \sigma_{1})^{p},$ as we want.
\end{proof} 

\begin{proposition}\label{isopoty4}
Consider $\beta$ a braid with $r$ strands and $p, q$ integers such that $1<r\leq q <p$. Then, there is an isotopy that takes the closure of the braid $\beta(\sigma_{p-1}\dots \sigma_{1})^{q}$
to the braid
$$\beta(\sigma_{1}\dots \sigma_{q-1})^{p-q}(\sigma_{q-1}\dots \sigma_{1})^{q}.$$
\end{proposition} 

\begin{proof}
Consider the $(q+1)$-st strand at the top of the braid $\beta(\sigma_{p-1}\dots \sigma_{1})^{q}$.
This strand anticlockwise goes around the braid closure to be at the $(q+1)$-st strand at the bottom of the braid
and then it passes through the sub-braid $(\sigma_{p-1}\dots \sigma_{1})^{q}$ to be at the $(1)$-st strand at the top of the sub-braid $(\sigma_{p-1}\dots \sigma_{1})^{q}$ as $r\leq q$.
We conclude that the $(q+1$)-st strand is connected to the $(1)$-st strand by an over strand that goes one time around the braid closure, as illustrated by the red strand in Figure~\ref{Isopoty4}.  
This over strand can be pushed up and shrunk to reduce one strand from the last braid, as shown in Figure~\ref{Isopoty4}, so that this over strand becomes an over strand between the $(1)$-st strand and the $(q)$-st strand producing  the braid 
$$\beta(\sigma_{1}\dots \sigma_{q-1})(\sigma_{p-2}\dots \sigma_{1})^{q}.$$

\begin{figure}
\includegraphics[scale=0.25]{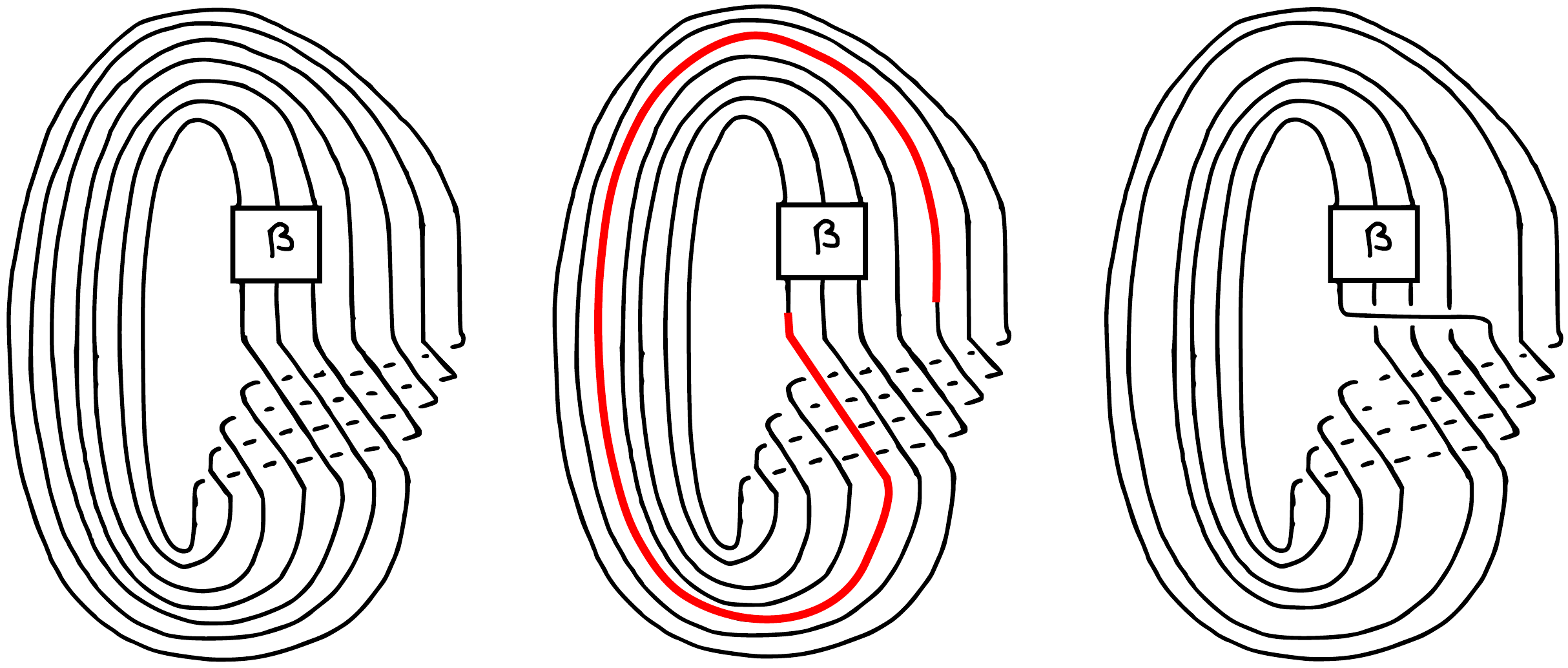} 
\caption{The red strand in the second link is pushed up and shrunk to yield the third link.}
\label{Isopoty4}
\end{figure}

Likewise, the $(q+1)$-st strand at the top of the last braid is still connected to the $(1)$-st strand at the bottom of the sub-braid $\beta(\sigma_{1}\dots \sigma_{q-1})$ by an over strand that goes one time around the braid closure because the sub-braid $\beta(\sigma_{1}\dots \sigma_{q-1})$ is still a braid on the leftmost $q$ strands. 
Hence we can apply the last isotopy again to obtain the braid
$$\beta(\sigma_{1}\dots \sigma_{q-1})^2(\sigma_{p-3}\dots \sigma_{1})^{q}.$$

We continue applying this isotopy until we apply it $p-q$ times in total to obtain the braid
$$\beta(\sigma_{1}\dots \sigma_{q-1})^{p-q}(\sigma_{q-1}\dots \sigma_{1})^{q}. \qedhere$$
\end{proof}

\section{V-links and Minimal braids of T-links}\label{section3}

In this section we prove that the minimal braids of the T-links are the same as a family of braids, called  V-braids. To prove this, we follow the idea of the proof of \cite[Theorem 3.5]{Lorenzknots}, but now we simplify the proof using some different isotopies from section 1. Moreover, if we follow the proof of \cite[Theorem 3.5]{Lorenzknots}, we end up with two families of braids, but we realize that they are the same family. This is illustrated in Lemma~\ref{otherbraid} below.

\begin{definition}
Let \(p,q\) be positive integers with \(2\le p\le q\). Let
\[
2\le u_1<\cdots<u_m\le p
\quad\text{and}\quad
2\le r_1<\cdots<r_n<p
\]
be possibly empty sequences, and let
\[
v_1,\dots,v_m,s_1,\dots,s_n
\]
be positive integers. The corresponding \(V\)-braid is the braid on \(p\)
strands given by
\begin{align*}
&(\sigma_{p-1}\sigma_{p-2}\dots\sigma_{{p-u_1+1}})^{v_1}\dots(\sigma_{p-1}\sigma_{p-2}\dots\sigma_{{p-u_m+1}})^{v_m}(\sigma_1\sigma_2\dots\sigma_{r_1-1})^{s_1}\dots(\sigma_{1}\sigma_{2}\dots\sigma_{r_n-1})^{s_n}\\
&(\sigma_{1}\sigma_{2}\dots\sigma_{{p-1}})^{q}.
\end{align*}
where the corresponding product is omitted if one of the sequences is empty.
Its closure is called the \(V\)-link
\[
V((u_1,\overline{v_1}),\dots,(u_m,\overline{v_m}),
(r_1,s_1),\dots,(r_n,s_n),(p,q)).
\]
\end{definition}

We call this family of braids V-links because to obtain them we start with the braid $(\sigma_{1}\sigma_{2}\dots\sigma_{{p-1}})^{q}$ that represents the torus link $T(p, q)$ and add sub-braids independently on both sides of this braid, left and right, in increasing order referring to the shape of the letter V.

\begin{lemma}\label{otherbraid}
Let $r_1, \dots, r_n, s_1, \dots, s_n, u_1, \dots, u_m, v_1, \dots, v_m, p, q$ be positive integers such that $2\leq r_1< \dots < r_{n} < p$, $2\leq u_1< \dots < u_{m} \leq p\leq q$. Then, the V-link  
$$V_1 = V((u_1,\overline{v_1}), \dots, (u_{m},\overline{v_{m}}), (r_1,s_1), \dots, (r_{n},s_{n}), (p, q))$$ 
is also equivalent to the closure of the braid
\begin{align*}
&B = (\sigma_1\sigma_2\dots\sigma_{u_1-1})^{v_1}\dots(\sigma_1\sigma_2\dots\sigma_{u_m-1})^{v_m}(\sigma_{p-1}\sigma_{p-2}\dots\sigma_{{p-r_1+1}})^{s_1}\dots\\
&(\sigma_{p-1}\sigma_{p-2}\dots\sigma_{{p-r_n+1}})^{s_n}(\sigma_{p-1}\sigma_{p-2}\dots\sigma_{{1}})^{q}.
\end{align*}
\end{lemma}

\begin{proof}
Consider the braid $B$ as above.
We rotate the projection plane of $B$ by $180^{\circ}$ degrees along the horizontal line as illustrated in Figure~\ref{V-link}.
This isotopy changes the position in which we see this braid.  After this isotopy, we see the back of $B$. Therefore, $B$ becomes the braid
\begin{align*}
&(\sigma_{p-1}\sigma_{p-2}\dots\sigma_{{p-u_1+1}})^{v_1}\dots(\sigma_{p-1}\sigma_{p-2}\dots\sigma_{{p-u_m+1}})^{v_m}(\sigma_1\sigma_2\dots\sigma_{r_1-1})^{s_1}\dots(\sigma_{1}\sigma_{2}\dots\sigma_{r_n-1})^{s_n}\\
&(\sigma_{1}\sigma_{2}\dots\sigma_{{p-1}})^{q},
\end{align*}
which represents the V-link $V_1$.
\end{proof}

Therefore, we conclude that the V-link $V((u_1,\overline{v_1}), \dots, (u_{m},\overline{v_{m}}), (r_1,s_1), \dots, (r_{n},s_{n}), (p, q))$ is represented by either the braid
$$(\sigma_{p-1}\dots\sigma_{{p-u_1+1}})^{v_1}\dots(\sigma_{p-1}\dots\sigma_{{p-u_m+1}})^{v_m}(\sigma_1\dots\sigma_{r_1-1})^{s_1}\dots(\sigma_{1}\dots\sigma_{r_n-1})^{s_n}(\sigma_{1}\dots\sigma_{{p-1}})^{q}$$
or the braid
$$(\sigma_1\dots\sigma_{u_1-1})^{v_1}\dots(\sigma_1\dots\sigma_{u_m-1})^{v_m}(\sigma_{p-1}\dots\sigma_{{p-r_1+1}})^{s_1}\dots(\sigma_{p-1}\dots\sigma_{{p-r_n+1}})^{s_n}(\sigma_{p-1}\dots\sigma_{{1}})^{q}.$$

Figure~\ref{V-link} illustrates an example of these two types of braids that represent the same V-link.

\begin{figure}
\includegraphics[scale=0.35]{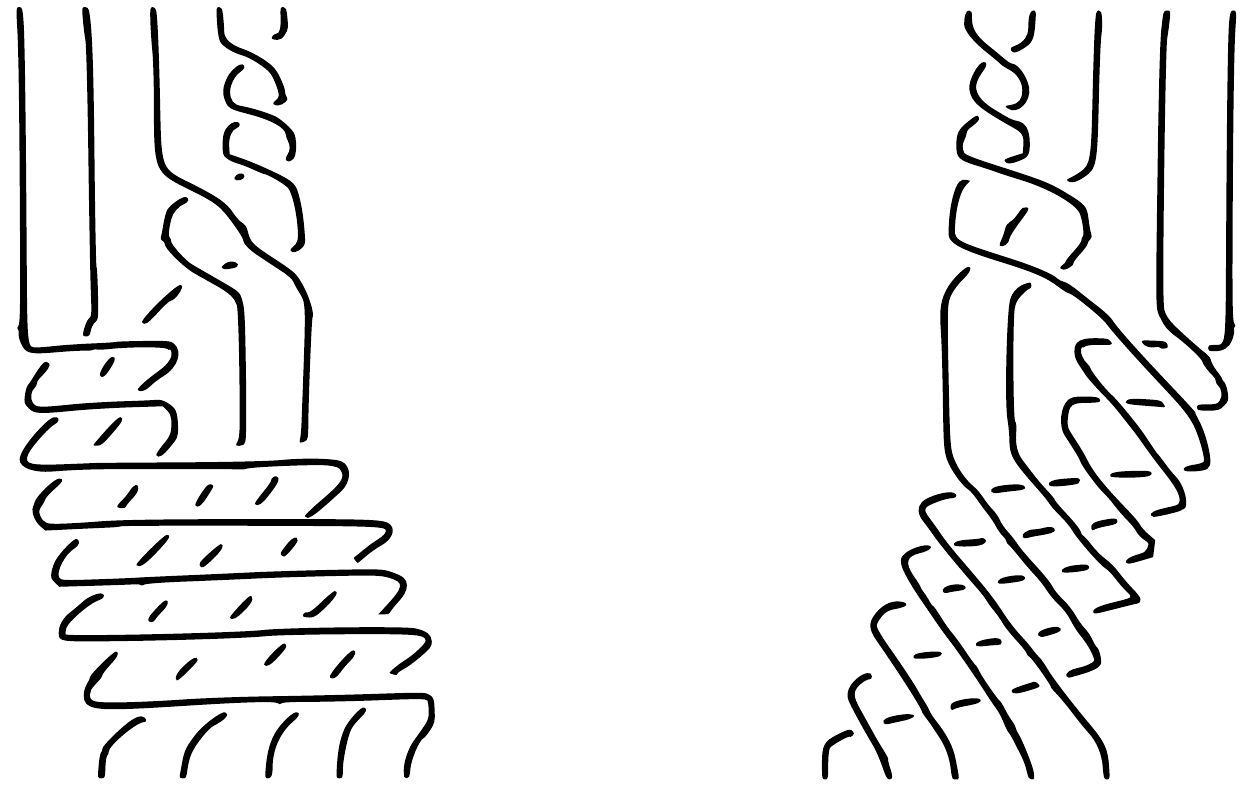} 
\caption{We have the V-link $V((2, \overline{2}), (3, \overline{2}), (3, 2), (5,5))$ in the first drawing. The second drawing illustrates the braid $(\sigma_1)^{2}(\sigma_{1}\sigma_{2})^{2}(\sigma_4\sigma_{3})^{2}(\sigma_{4}\sigma_{3}\sigma_{2}\sigma_{{1}})^{5}$ which also represents the last V-link. To see this, we apply an isotopy that rotates the projection plane of the first braid by $180 ^{\circ}$ degrees along the horizontal line. After that, we obtain the second braid.}
\label{V-link}
\end{figure}

\begin{lemma}\label{minimal}
V-link braids are minimal braids.
\end{lemma}

\begin{proof}
V-link braids are positive braids with at least one positive full twist.  By Franks and Williams \cite[Corollary 2.4]{Franks}, a positive braid with at least one positive full twist reaches the braid index. Therefore, V-link braids are minimal braids.
\end{proof}

We denote by $K$ the non-trivial T-link $$T((r_1,s_1), (r_2,s_2), \dots, (r_{n-1},s_{n-1}), (r_n,s_n)).$$
We can consider that $s_n>1$ otherwise this T-link braid can be destabilized on its right to yield the T-link $T((r_1,s_1), (r_2,s_2), \dots, (r_{n-1},s_{n-1}+1)).$ So we assume that $s_n>1$.

\begin{lemma}\label{case2}
If $s_n\geq r_n$, then $K$ is equivalent to the V-link 
$$V((r_1,s_1), (r_2,s_2), \dots, (r_{n-1},s_{n-1}), (r_n,s_n)).$$
\end{lemma} 

\begin{proof}
This directly follows from the definitions of T-links and V-links.
\end{proof} 

\begin{lemma}\label{case1}
If $1<s_n< r_n$ and $s_n\geq r_{n-1}$, then $K$ is equivalent to the V-link
$$V((r_1,\overline{s_1}), \dots, (r_{n-1},\overline{s_{n-1}}), (s_n,r_n)).$$
\end{lemma} 

\begin{proof}
The T-link $K$ is given by the braid
$$(\sigma_{1}\dots\sigma_{r_{1} - 1})^{s_{1}} (\sigma_{1}\dots\sigma_{r_{2} - 1})^{s_{2}}\dots (\sigma_{1}\dots\sigma_{r_{n-1} - 1})^{s_{n-1}}(\sigma_{1}\dots\sigma_{r_{n} - 1})^{s_{n}}.$$

By Proposition~\ref{isopoty3'}, there is an isotopy that takes this braid to the following braid 
$$(\sigma_{1}\dots\sigma_{r_{1} - 1})^{s_{1}} (\sigma_{1}\dots\sigma_{r_{2} - 1})^{s_{2}}\dots (\sigma_{1}\dots\sigma_{r_{n-1} - 1})^{s_{n-1}}(\sigma_{s_n-1}\dots \sigma_{1})^{r_n},$$
which represents the V-link
$$V((r_1,\overline{s_1}), \dots, (r_{n-1},\overline{s_{n-1}}), (s_n,r_n))$$
by Lemma~\ref{otherbraid}.
\end{proof}

\begin{proposition}\label{propositionlastcases}
Suppose that $1<s_n < r_n, r_{n-1}$.
\begin{itemize}
\item Consider first there is $j\leq n-2$ such that $s_n + s_{n-1} + \dots + s_{n-j} \geq r_{n-j}$ or $s_n + s_{n-1} + \dots + s_{n-j} < r_{n-j}$ and $s_n + s_{n-1} + \dots + s_{n-j} \geq r_{n-(j+1)}$. Let $\overline{i}$ be the smallest number such that $s_n + s_{n-1} + \dots + s_{n-\overline{i}} \geq r_{n-\overline{i}}$ or $s_n + s_{n-1} + \dots + s_{n-\overline{i}} < r_{n-\overline{i}}$ and $s_n + s_{n-1} + \dots + s_{n-\overline{i}} \geq r_{n-(\overline{i}+1)}$. 
\begin{itemize}
\item If $s_n + s_{n-1} + \dots + s_{n-\overline{i}} \geq r_{n-\overline{i}}$, then $K$ is equivalent to the V-link
\begin{align*}
&V((s_n,\overline{r_n-r_{n-1}}), \dots, (s_n+\dots+s_{n-(\overline{i}-1)},\overline{r_{n-(\overline{i}-1)}-r_{n-\overline{i}}}), (r_1,s_1), \dots, (r_{n-(\overline{i}+1)},s_{n-(\overline{i}+1)}),\\ 
&(r_{n-\overline{i}},s_n+s_{n-1}+\dots+s_{n-\overline{i}})); 
\end{align*}

\item If $s_n + s_{n-1} + \dots + s_{n-\overline{i}} < r_{n-\overline{i}}$ and $s_n + s_{n-1} + \dots + s_{n-\overline{i}} \geq r_{n-(\overline{i}+1)}$, then $K$ is equivalent to the V-link
\begin{align*}
&V((r_1,\overline{s_1}), \dots, (r_{n-(\overline{i}+1)}, \overline{s_{n-(\overline{i}+1)}}), (s_n, r_n-r_{n-1}), \dots, (s_n+\dots+s_{n-(\overline{i}-1)}, r_{n-(\overline{i}-1)}-r_{n-\overline{i}}),\\ 
&(s_n+s_{n-1}+\dots+s_{n-\overline{i}}, r_{n-\overline{i}})).  
\end{align*} 
\end{itemize}

\item
Otherwise, if there is no $j\leq n-2$ such that $s_n + s_{n-1} + \dots + s_{n-j} \geq r_{n-j}$ or $s_n + s_{n-1} + \dots + s_{n-j} < r_{n-j}$ and $s_n + s_{n-1} + \dots + s_{n-j} \geq r_{n-(j+1)}$.
\begin{itemize}
\item
If $s_n+s_{n-1}+\dots+s_{1}\geq r_{1}$, $K$ is equivalent to the V-link
\begin{align*}
&V((s_n,\overline{r_n-r_{n-1}}), (s_n+s_{n-1},\overline{r_{n-1}-r_{n-2}}),\dots, (s_n+s_{n-1}+\dots+s_{2}, \overline{r_{2}-r_{1}}),\\ 
&(r_{1}, s_n+s_{n-1}+\dots+s_{1}));  
\end{align*}
\item
If $s_n+s_{n-1}+\dots+s_{1}< r_{1}$,  $K$ is equivalent to the V-link
\begin{align*}
&V((s_n,r_n-r_{n-1}), (s_n+s_{n-1},r_{n-1}-r_{n-2}),\dots, (s_n+s_{n-1}+\dots+s_{2}, r_{2}-r_{1}), \\ 
&(s_n+s_{n-1}+\dots+s_{1}, r_{1})). \qedhere 
\end{align*}
\end{itemize}
\end{itemize}
\end{proposition}

\begin{proof}
Consider first there is $j\leq n-2$ such that $s_n + s_{n-1} + \dots + s_{n-j} \geq r_{n-j}$ or $s_n + s_{n-1} + \dots + s_{n-j} < r_{n-j}$ and $s_n + s_{n-1} + \dots + s_{n-j} \geq r_{n-(j+1)}$. Let $\overline{i}$ be the smallest number such that $s_n + s_{n-1} + \dots + s_{n-\overline{i}} \geq r_{n-\overline{i}}$ or $s_n + s_{n-1} + \dots + s_{n-\overline{i}} < r_{n-\overline{i}}$ and $s_n + s_{n-1} + \dots + s_{n-\overline{i}} \geq r_{n-(\overline{i}+1)}$. 

By Proposition~\ref{isopoty1}, there is an isotopy that takes the standard braid of $K$ to the braid
\begin{align*}
&B_1 = (\sigma_{r_{n-1}-1}\dots \sigma_{r_{n-1}-s_n+1})^{r_n-r_{n-1}}(\sigma_1\dots \sigma_{r_1-1})^{s_1} \dots (\sigma_1\dots \sigma_{r_{n-2}-1})^{s_{n-2}}\\
&(\sigma_1\dots \sigma_{r_{n-1}-1})^{s_n+s_{n-1}}.  
\end{align*}
If $s_n+s_{n-1}\geq r_{n-1}$, then $\overline{i} = 1$ and $B_1$ represents the V-link $$V((s_n,\overline{r_n-r_{n-1}}), (r_1,s_1), \dots, (r_{n-2},s_{n-2}), (r_{n-1}, s_n+s_{n-1})).$$
 
Consider now that $s_n+s_{n-1}< r_{n-1}$. We push the sub-braid $$(\sigma_{r_{n-1}-1}\dots \sigma_{r_{n-1}-s_n+1})^{r_n-r_{n-1}}$$ with $s_n$ strands of  $B_1$ anticlockwise until it is between the sub-braids of $B_1$ provided by the parameters $(r_{n-2}, s_{n-2})$ and $(r_{n-1}, s_n+s_{n-1})$ (see Figure~\ref{Isopoty1} for the direction in which we close the braids), as illustrated in Figure~\ref{T2}, so that we obtain an equivalent braid to $B_1$ of the form 
$B_1'(\sigma_1\dots \sigma_{r_{n-1}-1})^{s_n+s_{n-1}}$ with $B_1'$ a positive braid with $s_n+s_{n-1}$, $r_{n-2}$ strands if $s_n+s_{n-1} \geq r_{n-2}$, $s_n+s_{n-1} < r_{n-2}$, respectively.

\begin{figure}
\includegraphics[scale=0.25]{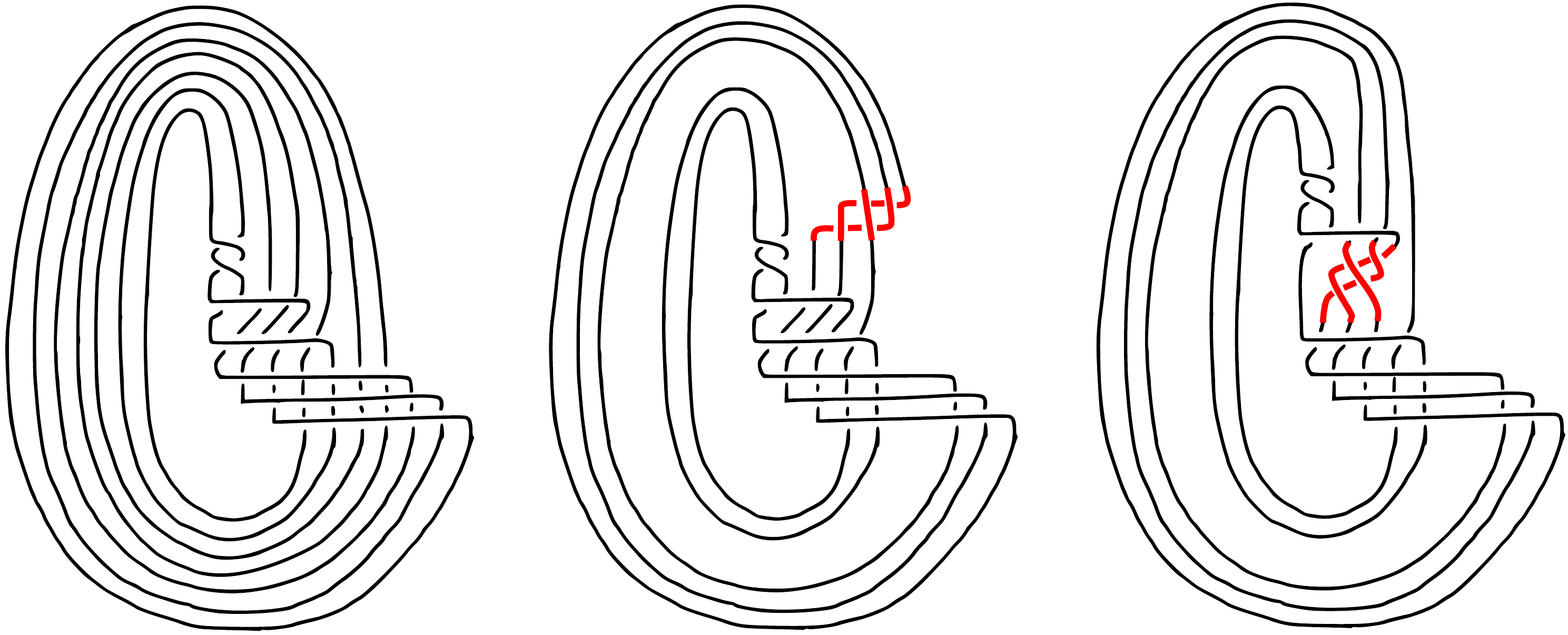} 
 \caption{We apply the isotopy of Proposition~\ref{isopoty1} to the first link to obtain the second link. Then we push the red crossings of the second link anticlockwise once around the braid closure to obtain the third link.}
\label{T2}
\end{figure}

If $s_n+s_{n-1} \geq r_{n-2}$, then $\overline{i} = 1$ and, by Proposition~\ref{isopoty3'}, there is an isotopy that takes the closure of the braid
$B_1'(\sigma_1\dots \sigma_{r_{n-1}-1})^{s_n+s_{n-1}}$
to the braid 
\begin{align*}
&(\sigma_1\dots \sigma_{r_1-1})^{s_1} \dots (\sigma_1\dots \sigma_{r_{n-2}-1})^{s_{n-2}}(\sigma_{s_n+s_{n-1}-1}\dots \sigma_{s_n+s_{n-1}-s_n+1})^{r_n-r_{n-1}} \\
&(\sigma_{s_n+s_{n-1}-1}\dots \sigma_{1})^{r_{n-1}}, 
\end{align*}
which represents the V-link
$$V((r_1,\overline{s_1}), \dots, (r_{n-2},\overline{s_{n-2}}), (s_n,r_n-r_{n-1}), (s_n+s_{n-1}, r_{n-1})).$$ 

If $s_n+s_{n-1} < r_{n-2}$, then we apply the isotopy of Proposition~\ref{isopoty1} to 
$B_1'(\sigma_1\dots \sigma_{r_{n-1}-1})^{s_n+s_{n-1}}$ and push the  sub-braid $(\sigma_{r_{n-1}-1}\dots \sigma_{r_{n-1}-s_n+1})^{r_n-r_{n-1}}$  back to the top. After that, we obtain the braid
\begin{align*}
&B_2 = (\sigma_{r_{n-2}-1}\dots \sigma_{r_{n-2}-s_n+1})^{r_n-r_{n-1}}
(\sigma_{r_{n-2}-1}\dots \sigma_{r_{n-2}-s_n-s_{n-1}+1})^{r_{n-1}-r_{n-2}} \\
&(\sigma_1\dots \sigma_{r_1-1})^{s_1} \dots (\sigma_1\dots \sigma_{r_{n-3}-1})^{s_{n-3}}(\sigma_1\dots \sigma_{r_{n-2}-1})^{s_n+s_{n-1}+s_{n-2}}.  
\end{align*}

If $s_n+s_{n-1}+s_{n-2} \geq r_{n-2}$, then $\overline{i} = 2$ and $B_2$ represents the V-link
$$V((s_n,\overline{r_n-r_{n-1}}), (s_n+s_{n-1},\overline{r_{n-1}-r_{n-2}}), (r_1,s_1), \dots, (r_{n-3},s_{n-3}), (r_{n-2}, s_n+s_{n-1}+s_{n-2})).$$

If not, then we push the sub-braid $$(\sigma_{r_{n-2}-1}\dots \sigma_{r_{n-2}-s_n+1})^{r_n-r_{n-1}}
(\sigma_{r_{n-2}-1}\dots \sigma_{r_{n-2}-s_n-s_{n-1}+1})^{r_{n-1}-r_{n-2}}$$ with $s_n+s_{n-1}$ strands anticlockwise until it is between the sub-braids provided by the parameters $(r_{n-3}, s_{n-3})$ and $(r_{n-2}, s_n+s_{n-1}+s_{n-2})$ so that we obtain an equivalent braid to $B_2$ of the form $B_2'(\sigma_1\dots \sigma_{r_{n-2}-1})^{s_n+s_{n-1}+s_{n-2}}$ with $B_2'$ a positive braid with $s_n+s_{n-1}+s_{n-2}$, $r_{n-3}$ strands if $s_n+s_{n-1}+s_{n-2}\geq r_{n-3}$, $s_n+s_{n-1}+s_{n-2}< r_{n-3}$, respectively. Then, we continue likewise until we obtain a positive braid with at least one positive full twist for $K$ of the form, if $s_n + s_{n-1} + \dots + s_{n-\overline{i}} \geq r_{n-\overline{i}}$, 
\begin{align*}
&(\sigma_{r_{n-\overline{i}}-1}\dots \sigma_{r_{n-\overline{i}}-s_n+1})^{r_n-r_{n-1}} \dots (\sigma_{r_{n-\overline{i}}-1}\dots \sigma_{r_{n-\overline{i}}-s_n-\dots-s_{n-(\overline{i}-1)}+1})^{r_{n-(\overline{i}-1)}-r_{n-\overline{i}}}\\ 
&(\sigma_1\dots \sigma_{r_1-1})^{s_1} \dots (\sigma_1\dots \sigma_{r_{n-(\overline{i}+1)}-1})^{s_{n-(\overline{i}+1)}}(\sigma_1\dots \sigma_{r_{n-\overline{i}}-1})^{s_n+s_{n-1}+\dots+s_{n-\overline{i}}}, 
\end{align*}
which represents the V-link
\begin{align*}
&V((s_n,\overline{r_n-r_{n-1}}), \dots, (s_n+\dots+s_{n-(\overline{i}-1)},\overline{r_{n-(\overline{i}-1)}-r_{n-\overline{i}}}), (r_1,s_1), \dots, (r_{n-(\overline{i}+1)},s_{n-(\overline{i}+1)})\\ 
&(r_{n-\overline{i}},s_n+s_{n-1}+\dots+s_{n-\overline{i}})), 
\end{align*}
or of the form, if $s_n + s_{n-1} + \dots + s_{n-\overline{i}} < r_{n-\overline{i}}$ and $s_n + s_{n-1} + \dots + s_{n-\overline{i}} \geq r_{n-(\overline{i}+1)}$,
\begin{align*}
&(\sigma_1\dots \sigma_{r_1-1})^{s_1} \dots (\sigma_1\dots \sigma_{r_{n-(\overline{i}+1)}-1})^{s_{n-(\overline{i}+1)}}\\ 
&(\sigma_{s_n+s_{n-1}+\dots+s_{n-\overline{i}}-1}\dots \sigma_{s_n+s_{n-1}+\dots+s_{n-\overline{i}}-s_n+1})^{r_n-r_{n-1}} \dots \\ 
&(\sigma_{s_n+s_{n-1}+\dots+s_{n-\overline{i}}-1}\dots \sigma_{s_n+s_{n-1}+\dots+s_{n-\overline{i}}-s_n-\dots-s_{n-(\overline{i}-1)}+1})^{r_{n-(\overline{i}-1)}-r_{n-\overline{i}}}\\
&(\sigma_{s_n+s_{n-1}+\dots+s_{n-\overline{i}}-1}\dots \sigma_{1})^{r_{n-\overline{i}}}, 
\end{align*}
which represents the V-link 
\begin{align*}
&V((r_1,\overline{s_1}), \dots, (r_{n-(\overline{i}+1)},\overline{s_{n-(\overline{i}+1)}}), (s_n,r_n-r_{n-1}), \dots, (s_n+\dots+s_{n-(\overline{i}-1)},r_{n-(\overline{i}-1)}-r_{n-\overline{i}}),\\ 
&(s_n+s_{n-1}+\dots+s_{n-\overline{i}}, r_{n-\overline{i}})).  
\end{align*}

Consider now that there is no $j\leq n-2$ such that $s_n + s_{n-1} + \dots + s_{n-j} \geq r_{n-j}$ or $s_n + s_{n-1} + \dots + s_{n-j} < r_{n-j}$ and $s_n + s_{n-1} + \dots + s_{n-j} \geq r_{n-(j+1)}$. Then, we continue with the above procedure until we obtain the braid
\begin{align*}
&B_{s_n+\dots+s_{1}} = (\sigma_{r_{1}-1}\dots \sigma_{r_{1}-s_n+1})^{r_n-r_{n-1}}
(\sigma_{r_{1}-1}\dots \sigma_{r_{1}-s_n-s_{n-1}+1})^{r_{n-1}-r_{n-2}}\dots\\ 
&(\sigma_{r_{1}-1}\dots \sigma_{r_{1}-s_n-s_{n-1}-\dots-s_{2}+1})^{r_{2}-r_{1}}(\sigma_1\dots \sigma_{r_{1}-1})^{s_n+s_{n-1}+\dots+s_{1}}.  
\end{align*}
If $s_n+s_{n-1}+\dots+s_{1}\geq r_{1}$, then this braid represents the V-link
\begin{align*}
&V((s_n,\overline{r_n-r_{n-1}}), (s_n+s_{n-1},\overline{r_{n-1}-r_{n-2}}),\dots, (s_n+s_{n-1}+\dots+s_{2}, \overline{r_{2}-r_{1}}),\\ 
&(r_{1}, s_n+s_{n-1}+\dots+s_{1})).  
\end{align*}

If $s_n+s_{n-1}+\dots+s_{1}< r_{1}$, then we push $r_{1}-(s_n+s_{n-1}+\dots+s_{2})$ meridional (horizontal) lines of the sub-braid $(\sigma_1\dots \sigma_{r_{1}-1})^{s_n+s_{n-1}+\dots+s_{1}}$ of $B_{s_n+\dots+s_{1}}$ anticlockwise around the braid closure to pass above the remaining sub-braid that is on the right and apply Proposition~\ref{isopoty3'} to obtain the braid
\begin{align*}
&B'_{s_n+\dots+s_{1}} = (\sigma_{k-1}\dots \sigma_{k-s_n+1})^{r_n-r_{n-1}}
(\sigma_{k-1}\dots \sigma_{k-s_n-s_{n-1}+1})^{r_{n-1}-r_{n-2}}\dots\\ 
&(\sigma_{k-1}\dots \sigma_{k-s_n-s_{n-1}-\dots-s_{2}+1})^{r_{2}-r_{1}}(\sigma_{k-1}\dots \sigma_{1})^{r_{1}} 
\end{align*}
with $k = s_n+\dots+s_{1}$. We see that this braid $B'_{s_n+\dots+s_{1}}$ represents the V-link
\begin{align*}
&V((s_n,r_n-r_{n-1}), (s_n+s_{n-1},r_{n-1}-r_{n-2}),\dots, (s_n+s_{n-1}+\dots+s_{2}, r_{2}-r_{1}),\\ 
&(s_n+s_{n-1}+\dots+s_{1}, r_{1})). \qedhere 
\end{align*}
\end{proof} 

\begin{theorem}\label{equivalence}
Let $r_1, \dots, r_n, s_1, \dots, s_n, u_1, \dots, u_m, v_1, \dots, v_m, p, q$ be positive integers such that $2\leq r_1< \dots < r_{n} < p$, $2\leq u_1< \dots < u_{m} \leq p\leq q$. Then, the V-link 
$$V((u_1,\overline{v_1}), \dots, (u_{m},\overline{v_{m}}), (r_1,s_1), \dots, (r_{n},s_{n}), (p, q))$$ 
is equivalent to the T-links
\begin{align*}
&T( (r_1,s_1), \dots, (r_{n},s_{n}), (p, q-u_m), (p+ v_m, u_m - u_{m-1}), (p+ v_m+v_{m-1}, u_{m-1} - u_{m-2}), \dots,\\
&(p+v_m+v_{m-1} + \dots + v_{2}, u_2-u_1), (p+ v_m+v_{m-1} + \dots + v_{1}, u_1))
\end{align*}
and
\begin{align*}
&T((u_1, v_1), \dots, (u_{m}, v_{m}), (q, p-r_n), (q+ s_n, r_n - r_{n-1}), (q+ s_n+s_{n-1}, r_{n-1} - r_{n-2}), \dots,\\
&(q+s_n+s_{n-1} + \dots + s_{2}, r_2-r_1), (q+ s_n+s_{n-1} + \dots + s_{1}, r_1)).
\end{align*}
\end{theorem}

\begin{proof} 
For the first T-link, if $m\geq 1$, then it follows by applying Proposition~\ref{propositionlastcases} to this T-link, otherwise it follows from Lemma~\ref{case2}.

For the second T-link, if $n\geq 1$, then it follows by applying Proposition~\ref{propositionlastcases} to this T-link, otherwise it follows from Lemma~\ref{case1}.
\end{proof}

\begin{corollary}\label{symmetry}
Let $r_1, \dots, r_n, s_1, \dots, s_n, u_1, \dots, u_m, v_1, \dots, v_m, p, q$ be positive integers such that $2\leq r_1< \dots < r_{n} < p$, $2\leq u_1< \dots < u_{m} \leq p\leq q$. Then, the T-links
\begin{align*}
&T( (r_1,s_1), \dots, (r_{n},s_{n}), (p, q-u_m), (p+ v_m, u_m - u_{m-1}), (p+ v_m+v_{m-1}, u_{m-1} - u_{m-2}), \dots,\\
&(p+v_m+v_{m-1} + \dots + v_{2}, u_2-u_1), (p+ v_m+v_{m-1} + \dots + v_{1}, u_1))
\end{align*}
and
\begin{align*}
&T((u_1, v_1), \dots, (u_{m}, v_{m}), (q, p-r_n), (q+ s_n, r_n - r_{n-1}), (q+ s_n+s_{n-1}, r_{n-1} - r_{n-2}), \dots,\\
&(q+s_n+s_{n-1} + \dots + s_{2}, r_2-r_1), (q+ s_n+s_{n-1} + \dots + s_{1}, r_1)).
\end{align*}
are equivalent.
\end{corollary}

\begin{proof} 
It follows from Theorem~\ref{equivalence}.
\end{proof}

\begin{theorem}\label{equivalenceTV-links}
Minimal braids of T-links are the same as V-link braids. 
In other words, a minimal braid of a T-link represents a V-link and a V-link represents a minimal braid of a T-link.
Furthermore, they are related as follows: the V-link
$$V((u_1,\overline{v_1}), \dots, (u_{m},\overline{v_{m}}), (r_1,s_1), \dots, (r_{n},s_{n}), (p, q))$$ 
is equivalent to the T-links
\begin{align*}
&T( (r_1,s_1), \dots, (r_{n},s_{n}), (p, q-u_m), (p+ v_m, u_m - u_{m-1}), \dots, (p+v_m+\dots + v_{2}, u_2-u_1),\\
&(p+ v_m+ \dots + v_{1}, u_1))\text{ and }T((u_1, v_1), \dots, (u_{m}, v_{m}), (q, p-r_n), (q+ s_n, r_n - r_{n-1}), \dots,\\
&(q+s_n+\dots + s_{2}, r_2-r_1), (q+ s_n+ \dots + s_{1}, r_1)).
\end{align*}
\end{theorem}

\begin{proof} 
A minimal braid of a T-link is a V-link by Lemmas~\ref{minimal}, \ref{case2}, \ref{case1}, and Proposition~\ref{propositionlastcases}, and 
a V-link is a minimal braid of a T-link by Lemma~\ref{minimal} and Theorem~\ref{equivalence}.
Furthermore, they are related as stated by Theorem~\ref{equivalence}.
\end{proof} 

\section{On Distinguishing T-Links}\label{section4}

In this section, we use the equivalence between T-links and V-links to obtain
an obstruction to isotopy among T-links.  The idea is that the V-link
representative keeps track of the number of parameter pairs in a controlled
way, once the obvious degeneracies have been removed.

Before applying Theorem~\ref{equivalenceTV-links} for parameter counting, we
exclude one harmless degeneracy.  Suppose that
\[
q=u_m.
\]
Since a V-link satisfies
\[
u_m\le p\le q,
\]
we must have
\[
q=u_m=p.
\]
In this case, the first T-link presentation associated to
\[
V((u_1,\overline{v_1}),\dots,(u_m,\overline{v_m}),
(r_1,s_1),\dots,(r_n,s_n),(p,q))
\]
in Theorem~\ref{equivalenceTV-links} contains the zero-exponent block
\[
(p,q-u_m)=(p,0),
\]
which is omitted under the standard simplification.  Moreover, by
Theorem~\ref{thm:rotate-V-link}, the same link is represented by
\[
V((r_1,\overline{s_1}),\dots,(r_n,\overline{s_n}),
(u_1,v_1),\dots,(u_m,v_m),(p,p)).
\]
Since \(u_m=p\), this simplifies to
\[
V((r_1,\overline{s_1}),\dots,(r_n,\overline{s_n}),
(u_1,v_1),\dots,(u_{m-1},v_{m-1}),(p,p+v_m)).
\]
Thus the case \(q=u_m\) is already absorbed by a shorter V-link presentation.
For this reason, throughout this section we assume
\[
q\neq u_m.
\]
Equivalently, since \(u_m\le p\le q\), we assume
\[
q>u_m.
\]

Under this convention, Theorem~\ref{equivalenceTV-links} gives useful
parameter-counting information.  Every T-link admits a braid-index-realizing
representative which is a V-link braid.  Moreover, if a V-link is written in
the form
\[
V((u_1,\overline{v_1}),\dots,(u_j,\overline{v_j}),
(w_1,z_1),\dots,(w_i,z_i)),
\]
where the last non-overlined pair \((w_i,z_i)\) is the final block \((p,q)\),
then the corresponding T-link has exactly \(j+i\) parameter pairs.

\begin{theorem}\label{thm:parameter-counting-obstruction}
Let
\[
2\le r_1<\cdots<r_n
\qquad\text{and}\qquad
2\le r'_1<\cdots<r'_m,
\]
and let \(s_1,\dots,s_n,s'_1,\dots,s'_m>0\), with \(s_n,s'_m>1\).
Consider the T-link presentations
\[
L=T((r_1,s_1),\dots,(r_n,s_n))
\qquad\text{and}\qquad
L'=T((r'_1,s'_1),\dots,(r'_m,s'_m)).
\]
If
\[
n>2\,br(L')-2,
\]
then \(L\) and \(L'\) are not isotopic in \(S^3\).
\end{theorem}

\begin{proof}
Assume, for contradiction, that \(L\) and \(L'\) are isotopic.

The assumptions
\[
2\le r_1<\cdots<r_n
\qquad\text{and}\qquad
s_1,\dots,s_n>0
\]
ensure that the displayed T-link presentation of \(L\) has no zero-exponent
blocks and no consecutive blocks with the same strand number.  Thus the
displayed presentation of \(L\) is in the simplified form considered in this
section.

Apply Theorem~\ref{equivalenceTV-links} to this displayed T-link presentation
of \(L\). Then \(L\) admits a braid-index-realizing representative which is a
V-link braid. We choose this V-link representative in the reduced form fixed at
the beginning of this section, so that the degeneracy \(q=u_m\) does not occur
and no zero-exponent block is produced in the corresponding T-link
presentation. Write it as
\[
V=
V((u_1,\overline{v_1}),\dots,(u_j,\overline{v_j}),
(w_1,z_1),\dots,(w_i,z_i)),
\]
where the last pair \((w_i,z_i)\) is the final V-link block. In particular,
\[
w_i=br(L).
\]

By the explicit correspondence in Theorem~\ref{equivalenceTV-links}, this
V-link, which has \(j\) overlined pairs and \(i\) non-overlined pairs including
the final block, corresponds to the displayed T-link presentation of \(L\).
Since this presentation has \(n\) parameter pairs, we have
\[
j+i=n.
\]

On both sides of the V-link, the strand parameters are strictly increasing and
lie in
\[
\{2,\dots,w_i\}.
\]
Therefore each side contains at most \(w_i-1\) pairs, and hence
\[
j\le w_i-1
\qquad\text{and}\qquad
i\le w_i-1.
\]
Thus
\[
n=j+i\le 2w_i-2.
\]

Since \(L\) and \(L'\) are isotopic, they have the same braid index. Thus
\[
w_i=br(L)=br(L').
\]
Consequently,
\[
n\le 2br(L')-2,
\]
contradicting the hypothesis. Therefore \(L\) and \(L'\) are not isotopic in
\(S^3\).
\end{proof}

\begin{corollary}\label{cor:parameter-counting-obstruction-rm}
Let
\[
2\le r_1<\cdots<r_n
\qquad\text{and}\qquad
2\le r'_1<\cdots<r'_m,
\]
and let \(s_1,\dots,s_n,s'_1,\dots,s'_m>0\), with \(s_n,s'_m>1\).
Consider the T-link presentations
\[
L=T((r_1,s_1),\dots,(r_n,s_n))
\qquad\text{and}\qquad
L'=T((r'_1,s'_1),\dots,(r'_m,s'_m)).
\]
If
\[
n>2r'_m-2,
\]
then \(L\) and \(L'\) are not isotopic in \(S^3\).
\end{corollary}

\begin{proof}
The displayed presentation of \(L'\) is the closure of a braid on \(r'_m\)
strands. Hence
\[
br(L')\le r'_m.
\]
Therefore
\[
2br(L')-2\le 2r'_m-2.
\]
By hypothesis,
\[
n>2r'_m-2,
\]
and hence
\[
n>2br(L')-2.
\]
The result follows from Theorem~\ref{thm:parameter-counting-obstruction}.
\end{proof}

The parameter-counting obstruction above distinguishes many T-links by comparing
the number of parameter pairs with the braid index.  We now give a complementary
obstruction for V-links with the same braid index.  When two V-link
presentations have the same final strand number \(p\), their associated
V-braids are positive braids on the same number of strands.  Hence, if they
represent the same link, Lemma~\ref{lem:crossing-number-positive-same-strands}
implies that their crossing numbers must agree.  Thus a difference in the
crossing numbers of the two V-braids obstructs isotopy.

\begin{corollary}\label{cor:crossing-number-obstruction-V-links}
Let
\[
L=
V((u_1,\overline{v_1}),\dots,(u_m,\overline{v_m}),
(r_1,s_1),\dots,(r_n,s_n),(p,q))
\]
and
\[
L'=
V((u'_1,\overline{v'_1}),\dots,(u'_{m'},\overline{v'_{m'}}),
(r'_1,s'_1),\dots,(r'_{n'},s'_{n'}),(p,q'))
\]
be two V-links with the same final strand number \(p\). If
\[
\sum_{i=1}^m v_i(u_i-1)
+
\sum_{j=1}^n s_j(r_j-1)
+
q(p-1)
\neq
\sum_{i=1}^{m'} v'_i(u'_i-1)
+
\sum_{j=1}^{n'} s'_j(r'_j-1)
+
q'(p-1),
\]
then \(L\) and \(L'\) are not isotopic in \(S^3\).
\end{corollary}

\begin{proof}
Let \(B_L\) and \(B_{L'}\) be the V-braids associated to \(L\) and \(L'\),
respectively. Both are positive braids on \(p\) strands. The crossing number of
\(B_L\) is
\[
c(B_L)=
\sum_{i=1}^m v_i(u_i-1)
+
\sum_{j=1}^n s_j(r_j-1)
+
q(p-1).
\]
Indeed, an overlined block \((u_i,\overline{v_i})\) contributes
\(v_i(u_i-1)\) crossings, a non-overlined block \((r_j,s_j)\) contributes
\(s_j(r_j-1)\) crossings, and the final block \((p,q)\) contributes
\(q(p-1)\) crossings. Similarly,
\[
c(B_{L'})=
\sum_{i=1}^{m'} v'_i(u'_i-1)
+
\sum_{j=1}^{n'} s'_j(r'_j-1)
+
q'(p-1).
\]

Suppose, for contradiction, that \(L\) and \(L'\) are isotopic. Then \(B_L\)
and \(B_{L'}\) are positive braids on the same number \(p\) of strands
representing the same link. By
Lemma~\ref{lem:crossing-number-positive-same-strands}, they must have the same
number of crossings. Hence
\[
c(B_L)=c(B_{L'}),
\]
contradicting the assumed inequality. Therefore \(L\) and \(L'\) are not
isotopic in \(S^3\).
\end{proof}

We conclude this section by emphasizing the role of V-link representatives in
the comparison of T-link presentations.  The usefulness of the V-link
description is that it places T-links in a more rigid braid-theoretic position.
From an arbitrary T-link presentation, it can be difficult to decide whether
two different parameter lists represent the same link.  In the V-link form,
however, the link is represented by a positive braid with an explicit full
twist, and this braid realizes the braid index.  Thus the braid index is visible
directly from the final strand number.

This gives strong restrictions on possible equivalences.  If two V-link
presentations represent the same link, then their final strand numbers must
agree.  Hence the comparison of T-link presentations can be reduced to the
comparison of V-link presentations with the same number of strands.  Once this
number is fixed, the V-link description still gives additional rigidity.
Indeed, V-link braids are positive braid representatives realizing the braid
index; therefore, two V-link presentations of the same link with the same final
strand number must also have the same number of crossings.  Thus crossing-number
comparisons provide a second obstruction to isotopy, complementary to braid
index considerations.

This is one of the main advantages of the V-link language: it places T-links in
a minimal braid-theoretic position where both braid index and crossing number
can be used effectively to compare presentations.

\section{Equivalent V-link and T-link Presentations}\label{section5}

In this section, we use the explicit correspondence between V-links and
T-links to study when different parameter presentations represent the same
link.  We begin by recovering the Birman--Kofman parameter equivalence from the
V-link/T-link correspondence: the two T-link presentations associated to a
single V-link are exchanged by the Birman--Kofman transformation.  We then
analyze when these two presentations are genuinely different after the standard
simplifications.

Next, we use the full-twist structure of V-links to generalize this phenomenon.
By redistributing the full-twist part of the final V-link block, we obtain
additional V-link presentations of the same link.  Combining this with the
V-link/T-link correspondence produces, under simple non-degeneracy conditions,
four distinct T-link presentations of the same link.

We also show that the opposite phenomenon occurs.  There are infinite families
of Lorenz links for which the T-link presentation is unique: once the
destabilization case \(s_n=1\) is excluded, these links admit only one T-link
presentation.

Finally, we prove that non-uniqueness is not caused merely by increasing the
braid index.  Even when the braid index is fixed, the number of distinct V-link
and T-link presentations of the same link can be made arbitrarily large.

As a first consequence of Theorem~\ref{equivalenceTV-links}, we obtain two
T-link presentations of the same link from each V-link.

\begin{corollary}\label{cor:two-T-links-same}
The two T-links
\begin{align*}
&T( (r_1,s_1), \dots, (r_{n},s_{n}), (p, q-u_m), (p+ v_m, u_m - u_{m-1}), \dots, (p+v_m+\dots + v_{2}, u_2-u_1),\\
&(p+ v_m+ \dots + v_{1}, u_1)).
\end{align*}
and
\begin{align*}
&T((u_1, v_1), \dots, (u_{m}, v_{m}), (q, p-r_n), (q+ s_n, r_n - r_{n-1}), \dots, (q+s_n+\dots + s_{2}, r_2-r_1),\\
&(q+ s_n+ \dots + s_{1}, r_1))
\end{align*}
represent the same link in \(S^3\). 
\end{corollary}

\begin{proof}
By Theorem~\ref{equivalenceTV-links}, both T-links are equivalent to the same
V-link
\[
V\big((u_1,\overline{v_1}),\dots,(u_m,\overline{v_m}),
(r_1,s_1),\dots,(r_n,s_n),(p,q)\big).
\]
Therefore they represent the same link in \(S^3\).
\end{proof}

The previous corollary gives a V-link interpretation of a parameter equivalence
due to Birman--Kofman \cite[Corollary~4]{newtwis}. Indeed, applying the
Birman--Kofman transformation to the first T-link presentation in
Corollary~\ref{cor:two-T-links-same} gives exactly the second one. Hence the
Birman--Kofman equivalence is recovered as the T-link-level expression of a
single V-link presentation. We record this consequence below.

\begin{corollary}[Birman--Kofman]\label{cor:BK}
Let
\[
T((R_1,S_1),\dots,(R_N,S_N))
\]
be a T-link. Define
\[
R'_1=S_N,\quad R'_2=S_N+S_{N-1},\quad \dots,\quad
R'_N=S_N+S_{N-1}+\cdots+S_1
\]
and
\[
S'_1=R_N-R_{N-1},\quad S'_2=R_{N-1}-R_{N-2},\quad \dots,\quad
S'_{N-1}=R_2-R_1,\quad S'_N=R_1.
\]
Then the T-links
\[
T((R_1,S_1),\dots,(R_N,S_N))
\]
and
\[
T((R'_1,S'_1),\dots,(R'_N,S'_N))
\]
represent the same link in \(S^3\).
\end{corollary}

\begin{proof}
By Theorem~\ref{equivalenceTV-links}, every T-link presentation arises, after
the standard simplifications, as one of the two T-link presentations associated
to a V-link in Corollary~\ref{cor:two-T-links-same}. Therefore it is enough to
check that the Birman--Kofman transformation sends the first presentation in
Corollary~\ref{cor:two-T-links-same} to the second one. If the given T-link
appears as the second presentation, the conclusion follows by applying the same
argument in reverse, since the Birman--Kofman transformation is an involution.

We write the calculation for the case in which both the overlined and
non-overlined blocks are non-empty; if one of the blocks is empty, the
corresponding terms are omitted.

Write the first presentation in Corollary~\ref{cor:two-T-links-same} as
\[
T((R_1,S_1),\dots,(R_N,S_N)),
\]
where \(N=m+n+1\). Thus
\[
(R_i,S_i)=(r_i,s_i)
\qquad\text{for } i=1,\dots,n,
\]
\[
(R_{n+1},S_{n+1})=(p,q-u_m),
\]
and, for \(a=1,\dots,m\),
\[
(R_{n+1+a},S_{n+1+a})
=
(p+v_m+\cdots+v_{m-a+1},\,u_{m-a+1}-u_{m-a}),
\]
with the convention \(u_0=0\). Therefore the first presentation is
\begin{align*}
T((R_1,S_1),\dots,(R_N,S_N))
={}&T((r_1,s_1),\dots,(r_n,s_n),(p,q-u_m), (p+v_m,u_m-u_{m-1}), \dots,\\
&\quad (p+v_m+\cdots+v_1,u_1)).
\end{align*}

Now apply the Birman--Kofman transformation. The transformed strand parameters
are
\[
R'_j=S_N+S_{N-1}+\cdots+S_{N-j+1},
\]
and the transformed exponents are
\[
S'_j=R_{N-j+1}-R_{N-j},
\]
where we use the convention \(R_0=0\).

For \(j=1,\dots,m\), the last \(j\) second coordinates of the first
presentation are
\[
u_1,\quad u_2-u_1,\quad \dots,\quad u_j-u_{j-1}.
\]
Hence
\[
R'_j=u_j.
\]
Moreover, the corresponding differences of the first coordinates are
\[
S'_j=v_j.
\]
Thus the first \(m\) transformed pairs are
\[
(R'_j,S'_j)=(u_j,v_j),
\qquad j=1,\dots,m.
\]

The next transformed strand parameter is obtained by also adding the second
coordinate \(q-u_m\). Hence
\[
R'_{m+1}=u_m+(q-u_m)=q.
\]
The corresponding exponent is
\[
S'_{m+1}=p-r_n.
\]
Therefore
\[
(R'_{m+1},S'_{m+1})=(q,p-r_n).
\]

For the remaining pairs, for \(j=1,\dots,n\), we obtain
\[
R'_{m+1+j}
=
q+s_n+s_{n-1}+\cdots+s_{n-j+1},
\]
and
\[
S'_{m+1+j}
=
r_{n-j+1}-r_{n-j},
\]
with the convention \(r_0=0\). Hence these pairs are
\[
(q+s_n,r_n-r_{n-1}),\dots,
(q+s_n+\cdots+s_2,r_2-r_1),
(q+s_n+\cdots+s_1,r_1).
\]

Therefore the Birman--Kofman transform of the first presentation is exactly
\begin{align*}
&T((u_1, v_1), \dots, (u_{m}, v_{m}), (q, p-r_n), (q+ s_n, r_n - r_{n-1}), \dots, (q+s_n+\dots + s_{2}, r_2-r_1),\\
&(q+ s_n+ \dots + s_{1}, r_1)).
\end{align*}
This is precisely the second T-link presentation in
Corollary~\ref{cor:two-T-links-same}. Since the two presentations in
Corollary~\ref{cor:two-T-links-same} represent the same link, the two
T-links in the statement represent the same link in \(S^3\).
\end{proof}

Thus the Birman--Kofman equivalence is recovered from the V-link/T-link
correspondence.  At the level of T-link parameters, the Birman--Kofman
transformation sends one of the two presentations associated to a V-link to the
other.  In this sense, Theorem~\ref{equivalenceTV-links} gives a minimal-braid
interpretation of the classical parameter equivalence for T-links.

We say that a T-link presentation is simplified if it contains no parameter
pair with second coordinate equal to zero and no two consecutive parameter
pairs with the same first coordinate.  Given any parameter list, its standard
simplification is obtained by omitting all pairs of the form \((r,0)\) and by
replacing consecutive pairs \((r,a),(r,b)\) with \((r,a+b)\).  These operations
preserve the represented link, and hence only change the presentation, not the
link type.

\begin{proposition}\label{prop:when-two-T-presentations-coincide}
Let
\[
L=
V((u_1,\overline{v_1}),\dots,(u_m,\overline{v_m}),
(r_1,s_1),\dots,(r_n,s_n),(p,q))
\]
be a V-link. Let \(T_1\) and \(T_2\) be the two T-link presentations obtained
from Corollary~\ref{cor:two-T-links-same}, after the standard simplifications.
If \(T_1\) and \(T_2\) coincide as ordered lists of parameters, then one of the
following holds:
\begin{enumerate}
\item \(m=n\), $u_m<p$,
\[
(u_i,v_i)=(r_i,s_i)\quad\text{for all }i=1,\dots,n,
\]
and
\[
q=p;
\]

\item \(m=n+1\),
\[
(u_i,v_i)=(r_i,s_i)\quad\text{for all }i=1,\dots,n,
\]
and
\[
(u_m,v_m)=(p,q-p).
\]
\end{enumerate}

Conversely, in each of these two cases, the two simplified T-link presentations
coincide.
\end{proposition}

\begin{proof}
Let \(T_1\) and \(T_2\) be the two T-link presentations obtained from
Corollary~\ref{cor:two-T-links-same}, after the standard simplifications.

First suppose that \(m=n=0\). Then
\[
T_1=T((p,q))
\qquad\text{and}\qquad
T_2=T((q,p)).
\]
Hence \(T_1=T_2\) if and only if \(q=p\), which is case (1).

Next suppose that \(m=0\) and \(n\ge1\). Then
\[
T_1=T((r_1,s_1),\dots,(r_n,s_n),(p,q)),
\]
whereas
\[
T_2=
T((q,p-r_n),(q+s_n,r_n-r_{n-1}),\dots,
(q+s_n+\cdots+s_1,r_1)).
\]
The first pair of \(T_1\) has first coordinate \(r_1\), while the first pair of
\(T_2\) has first coordinate \(q\). Since
\[
r_1<p\le q,
\]
these first coordinates are different. Thus \(T_1\neq T_2\), so no coincidence
occurs in this case.

Now suppose that \(n=0\) and \(m\ge1\). Before simplification, the two
presentations are
\[
T_1=
T((p,q-u_m),
(p+v_m,u_m-u_{m-1}),\dots,
(p+v_m+\cdots+v_1,u_1))
\]
and
\[
T_2=
T((u_1,v_1),\dots,(u_m,v_m),(q,p)).
\]

If \(m>1\), then \(u_1<p\). If \(q>u_m\), then \(T_1\) begins with a pair whose
first coordinate is \(p\), while \(T_2\) begins with \((u_1,v_1)\), whose first
coordinate is \(u_1<p\). Hence \(T_1\neq T_2\).

If \(q=u_m\), then, since \(p\le q\le u_m\le p\), we have \(q=u_m=p\). The
pair \((p,q-u_m)=(p,0)\) is omitted from \(T_1\). The first remaining pair of
\(T_1\) has first coordinate
\[
p+v_m>p,
\]
whereas \(T_2\) begins with \((u_1,v_1)\), with \(u_1<p\). Hence again
\(T_1\neq T_2\). Thus no coincidence occurs when \(n=0\) and \(m>1\).

It remains to consider \(n=0\) and \(m=1\). Then
\[
T_1=T((p,q-u_1),(p+v_1,u_1))
\]
and
\[
T_2=T((u_1,v_1),(q,p)).
\]
If \(q=u_1\), then \(q=u_1=p\), so after simplification
\[
T_1=T((p+v_1,p)),
\]
whereas \(T_2\) simplifies to
\[
T((p,v_1+p)).
\]
These are distinct because \(p+v_1\neq p\). Therefore equality can occur only
when \(q>u_1\). In that case no simplification occurs in the first pair of
\(T_1\), and comparing the first pairs gives
\[
u_1=p
\qquad\text{and}\qquad
v_1=q-p.
\]
Conversely, if
\[
(u_1,v_1)=(p,q-p),
\]
then
\[
T_1=T((p,q-p),(q,p))=T_2.
\]
Thus, when \(n=0\), the only coincidence is case (2).

Finally, suppose that \(m,n\ge1\). Before simplification, the two presentations
are
\begin{align*}
T_1
={}&T((r_1,s_1),\dots,(r_n,s_n),
(p,q-u_m),
(p+v_m,u_m-u_{m-1}),\dots,
(p+v_m+\cdots+v_1,u_1))
\end{align*}
and
\begin{align*}
T_2
={}&T((u_1,v_1),\dots,(u_m,v_m),
(q,p-r_n),
(q+s_n,r_n-r_{n-1}),\dots,
(q+s_n+\cdots+s_1,r_1)).
\end{align*}

First assume that \(q=u_m\). Then, since \(p\le q\le u_m\le p\), we have
\[
q=u_m=p.
\]
Hence the pair \((p,q-u_m)=(p,0)\) is omitted from \(T_1\). The first coordinate
in \(T_1\) becomes at least \(p\) for the first time at the pair
\[
(p+v_m,p-u_{m-1}),
\]
whose first coordinate is \(p+v_m>p\). In \(T_2\), the last initial block is
\[
(u_m,v_m)=(p,v_m),
\]
followed by
\[
(q,p-r_n)=(p,p-r_n).
\]
After simplification these combine into a pair with first coordinate \(p\).
Thus the first coordinate at which \(T_2\) reaches the \(p\)-level is \(p\),
whereas for \(T_1\) it is \(p+v_m\). Therefore \(T_1\neq T_2\). Hence equality
cannot occur when \(q=u_m\).

Thus, if \(T_1=T_2\), we must have \(q>u_m\), so no zero-exponent block is
omitted from \(T_1\). Now compare the first position at which the first
coordinate is at least \(p\).

In \(T_1\), this occurs at position \(n+1\), at the pair
\[
(p,q-u_m).
\]
In \(T_2\), there are two cases.

If \(u_m<p\), then this occurs at position \(m+1\), at the pair
\[
(q,p-r_n).
\]
Thus equality forces
\[
m=n.
\]
Comparing the first \(n\) pairs gives
\[
(u_i,v_i)=(r_i,s_i)
\quad\text{for all }i=1,\dots,n.
\]
Comparing the next pair gives
\[
(q,p-r_n)=(p,q-u_m).
\]
Since \(u_m=r_n\), this forces
\[
q=p.
\]
Thus case (1) holds.

If \(u_m=p\), then the first coordinate in \(T_2\) reaches \(p\) at position
\(m\), at the pair
\[
(u_m,v_m)=(p,v_m).
\]
Since \(T_1\) reaches \(p\) at position \(n+1\), equality forces
\[
m=n+1.
\]
Comparing the first \(n\) pairs gives
\[
(u_i,v_i)=(r_i,s_i)
\quad\text{for all }i=1,\dots,n.
\]
Comparing the next pair gives
\[
(p,v_m)=(p,q-u_m).
\]
Since \(u_m=p\), we get
\[
v_m=q-p.
\]
Thus case (2) holds.

Conversely, suppose case (1) holds. If \(n=0\), then \(m=0\) and \(q=p\), so
\[
T_1=T((p,p))=T_2.
\]
Now suppose \(n\ge1\). Then \(m=n\), the two initial blocks agree, and
\(q=p\). In particular,
\[
u_m=r_n<p.
\]
Hence no simplification affects the relevant transition pair, and substitution
into the formulas shows that the two simplified ordered lists are equal.

Suppose case (2) holds. Then \(m=n+1\), the first \(n\) overlined blocks agree
with the \(n\) non-overlined blocks, and
\[
(u_m,v_m)=(p,q-p).
\]
Since \(v_m>0\), we have \(q>p\), so no zero-exponent block appears.

If \(n=0\), then \(m=1\), and the two presentations are
\[
T_1=T((p,q-u_1),(p+v_1,u_1))
\]
and
\[
T_2=T((u_1,v_1),(q,p)).
\]
Since \((u_1,v_1)=(p,q-p)\), we obtain
\[
T_1=T((p,q-p),(q,p))=T_2.
\]

Now assume \(n\ge1\). Substituting the identities from case (2) into the
formulas gives
\[
(p,q-u_m)=(p,q-p)=(u_m,v_m),
\]
and the remaining pairs agree term by term:
\[
(p+v_m,u_m-u_{m-1})=(q,p-r_n),
\]
then
\[
(p+v_m+v_{m-1},u_{m-1}-u_{m-2})
=
(q+s_n,r_n-r_{n-1}),
\]
and so on. Hence \(T_1=T_2\).

This proves the proposition.
\end{proof}

\begin{lemma}\label{lem:crossing-number-positive-same-strands}
Let \(\beta\) and \(\beta'\) be positive braid representatives of the same link.
If \(\beta\) and \(\beta'\) have the same number of strands, then they have the
same number of crossings.
\end{lemma}

\begin{proof}
Let \(L\) be the link represented by both braids, and let \(\mu\) be the number
of components of \(L\). For a positive braid representative of \(L\) with
\(p\) strands and \(c\) crossings, the genus of \(L\) is given by
\[
g(L)=\frac{c-p-\mu+2}{2},
\]
by Stallings' theorem \cite{Genus}. Since the genus is a link invariant, and since \(\beta\) and \(\beta'\) have the
same number of strands, it follows that they have the same number of crossings.
\end{proof}

We first make a simple convention.  If the last exponent in a T-link
presentation is equal to \(1\), then the last block can be removed by a
destabilization.  Indeed,
\[
T((r_1,s_1),\dots,(r_{n-1},s_{n-1}),(r_n,1))
\]
is equivalent to
\[
T((r_1,s_1),\dots,(r_{n-1},s_{n-1}+1)).
\]
Thus, we may assume that the last exponent
satisfies
\[
s_n>1.
\]

\begin{proposition}\label{prop:unique-T-presentation-torus}
For each \(p\ge 2\), the link
\[
L_p=V((p,p))
\]
has a unique simplified T-link presentation, namely
\[
T((p,p)).
\]
\end{proposition}

\begin{proof}
The V-link
\[
V((p,p))
\]
is represented by the positive braid
\[
(\sigma_1\sigma_2\cdots\sigma_{p-1})^p,
\]
which is the positive full twist on \(p\) strands. This braid has \(p\) strands
and
\[
p(p-1)
\]
crossings.

Suppose that the same link admits another V-link presentation
\[
V((u_1,\overline{v_1}),\dots,(u_m,\overline{v_m}),
(r_1,s_1),\dots,(r_n,s_n),(p',q')).
\]
By Theorem~\ref{equivalenceTV-links}, V-link braids are minimal braid
representatives of the corresponding T-links. Hence this new V-link braid is
also a minimal braid representative of the same link. Therefore it has the same
braid index as \(V((p,p))\), and so
\[
p'=p.
\]

Moreover, since both braids are positive with the same number of strands representing the same
link, Lemma~\ref{lem:crossing-number-positive-same-strands} implies that they have
the same number of crossings. Hence the new V-link presentation also has exactly
\[
p(p-1)
\]
crossings.

By the definition of a V-link, the final block \((p,q')\) satisfies
\[
q'\ge p.
\]
Thus the final block alone contributes at least
\[
p(p-1)
\]
crossings. Since the total number of crossings is exactly \(p(p-1)\), no
crossings can occur in any of the preceding blocks. Therefore the overlined and
non-overlined blocks must be empty:
\[
m=n=0.
\]
Also, the final block must contribute exactly \(p(p-1)\) crossings, and hence
\[
q'=p.
\]
Therefore the only possible V-link presentation is
\[
V((p,p)).
\]

Now let \(T\) be any simplified T-link presentation of \(L_p\). By
Theorem~\ref{equivalenceTV-links}, \(T\) determines a V-link presentation of
\(L_p\). Since the V-link presentation is unique, this V-link must be
\[
V((p,p)).
\]
Applying Corollary~\ref{cor:two-T-links-same} to \(V((p,p))\), the two
associated T-link presentations are both
\[
T((p,p)).
\]
Therefore
\[
T=T((p,p)).
\]
Thus \(L_p\) has a unique simplified T-link presentation.
\end{proof}

\begin{proposition}\label{prop:unique-V-presentation-general-r}
Let \(r\ge 2\) and \(p\ge 2r\). Let
\[
L_{r,p}=V((r,\overline{1}),(r,1),(p,p)).
\]
Then \(L_{r,p}\) has a unique V-link presentation. Consequently, the two
T-link presentations obtained from Corollary~\ref{cor:two-T-links-same}
coincide, and the corresponding simplified T-link presentation is
\[
T((r,1),(p,p-r),(p+1,r)).
\]
\end{proposition}

\begin{proof}
The V-link
\[
L_{r,p}=V((r,\overline{1}),(r,1),(p,p))
\]
is represented by the positive braid
\[
(\sigma_{p-1}\sigma_{p-2}\cdots\sigma_{p-r+1})
(\sigma_1\sigma_2\cdots\sigma_{r-1})
(\sigma_1\sigma_2\cdots\sigma_{p-1})^p .
\]
This braid has \(p\) strands. The final factor is the positive full twist on
\(p\) strands and contributes
\[
p(p-1)
\]
crossings. The two remaining factors contribute
\[
(r-1)+(r-1)=2(r-1)
\]
crossings. Hence this braid has
\[
p(p-1)+2(r-1)
\]
crossings.

Suppose that \(L_{r,p}\) admits another V-link presentation
\[
V((u_1,\overline{v_1}),\dots,(u_m,\overline{v_m}),
(r_1,s_1),\dots,(r_n,s_n),(p',q')).
\]
By Theorem~\ref{equivalenceTV-links}, V-link braids realize the braid index of
the corresponding T-links. Therefore this new V-link braid has the same number
of strands as the braid above, and hence
\[
p'=p.
\]

Moreover, since both braids are positive with the same number of strands and
represent the same link, Lemma~\ref{lem:crossing-number-positive-same-strands}
implies that they have the same number of crossings. Thus this new V-link
presentation also has
\[
p(p-1)+2(r-1)
\]
crossings.

The final block \((p,q')\) contributes \(q'(p-1)\) crossings. Since, by the
definition of V-links,
\[
q'\ge p,
\]
this contribution is at least
\[
p(p-1).
\]
If \(q'>p\), then \(q'\ge p+1\), and the final block alone would contribute at
least
\[
(p+1)(p-1)=p(p-1)+(p-1)
\]
crossings. Since \(p\ge 2r\), we have
\[
p-1>2(r-1),
\]
so this would exceed the total number
\[
p(p-1)+2(r-1).
\]
Therefore
\[
q'=p.
\]
It follows that the blocks before the final full twist contribute exactly
\[
2(r-1)
\]
crossings.

We now analyze the components before the full twist. In the original
presentation, the pre-full-twist braid is
\[
(\sigma_{p-1}\cdots\sigma_{p-r+1})
(\sigma_1\cdots\sigma_{r-1}).
\]
Since \(p\ge 2r\), these two factors act on disjoint sets of strands. They
create two components supported on \(r\) strands each: one on the initial
\(r\) strands and one on the terminal \(r\) strands. All remaining components
are supported on single strands. After the final full twist is added, each of
these two \(r\)-strand components is represented by a positive \(r\)-strand
braid containing a full twist. Hence, by \cite[Corollary~2.4]{Franks}, each of
these two components has braid index \(r\).

Now consider the same two components in the alternative V-link presentation.
Suppose that they are supported on \(a\) and \(b\) strands, respectively. 
Again, by \cite[Corollary~2.4]{Franks}, the number of
strands supporting each component realizes its braid index. Since both
components have braid index \(r\), we obtain
\[
a=r
\qquad\text{and}\qquad
b=r.
\]

We also use their linking number. In the original presentation, these two
components do not interact before the final full twist, because their supports
are disjoint. Thus their linking number comes only from the final full twist.
The full twist makes every strand of one component wind once positively around
every strand of the other component.  Equivalently, for each pair consisting of
one strand from the first component and one strand from the second, the full
twist contributes two positive crossings.  Since the linking number is one half
of the signed number of crossings between the two components, each such pair of
strands contributes \(1\) to the linking number.  There are \(r^2\) such pairs
of strands.  Hence the full twist contributes
\[
r^2
\]
to the linking number between the two components.

In the alternative V-link presentation, the same two components are again
supported on \(r\) strands each. The final full twist already contributes
\[
r^2
\]
to their linking number. Since their linking number is an invariant of the
oriented link and is equal to \(r^2\), there can be no crossings between these
two components before the final full twist. Indeed, all crossings in a V-link
braid are positive, so any such crossing would increase their linking number.

Thus the crossings before the final full twist split into two disjoint groups:
one group creates one of the \(r\)-strand components, and the other group
creates the other \(r\)-strand component. Since the total number of crossings
before the final full twist is exactly
\[
2(r-1),
\]
and a component supported on \(r\) strands requires at least \(r-1\) crossings
to be created, each group contains exactly \(r-1\) crossings. Hence each
\(r\)-strand component is created minimally.

In a V-link presentation, the only way to create an \(r\)-strand component with
exactly \(r-1\) crossings is by a single block with second coordinate \(1\).
Thus, after possibly relabelling the two \(r\)-strand components, one of them
is created by the block
\[
(r,1),
\]
and the other is created by the block
\[
(r,\overline{1}).
\]
The two components cannot both be created on the non-overlined side, because
non-overlined blocks are supported on nested initial intervals of strands.
Similarly, they cannot both be created on the overlined side, because overlined
blocks are supported on nested terminal intervals of strands. Since the two
components are disjoint and \(p\ge 2r\), one must be created on the
non-overlined side and the other on the overlined side. Therefore the only
possible V-link presentation is
\[
V((r,\overline{1}),(r,1),(p,p)).
\]
This proves uniqueness of the V-link presentation.

Finally, applying Corollary~\ref{cor:two-T-links-same} to this unique V-link
presentation gives
\[
T((r,1),(p,p-r),(p+1,r))
\]
from both sides of the correspondence. Indeed, the two sides have the same
block \((r,1)\), and the final block is \((p,p)\). Hence the two T-link
presentations coincide, and the corresponding simplified T-link presentation is
\[
T((r,1),(p,p-r),(p+1,r)). \qedhere
\] 
\end{proof}

\begin{theorem}\label{thm:unique-T-link-presentations}
The following families give links with unique T-link presentations.

\begin{enumerate}
\item For each \(p\ge 2\), the T-link
\[
T((p,p))
\]
has a unique T-link presentation.

\item Let \(r\ge 2\) and \(p\ge 2r\). The T-link
\[
T((r,1),(p,p-r),(p+1,r))
\]
has a unique T-link presentation.
\end{enumerate}
\end{theorem}

\begin{proof}
The first statement follows from
Proposition~\ref{prop:unique-T-presentation-torus} and the second statement follows Proposition~\ref{prop:unique-V-presentation-general-r}.
\end{proof}

Every V-link has a final block \((p,Q)\) with \(Q\ge p\).  Writing
\[
Q=kp+q,\qquad k\ge1,\qquad 0\le q<p,
\]
we obtain the following equivalence.

\begin{theorem}\label{thm:rotate-V-link}
Let \(k\ge1\) and \(0\le q<p\). Then the two V-links
\[
V((u_1,\overline{v_1}), \dots, (u_m,\overline{v_m}),
(r_1,s_1), \dots, (r_n,s_n), (p,kp+q))
\]
and
\[
V((r_1,\overline{s_1}), \dots, (r_n,\overline{s_n}),
(p,\overline q),
(u_1,v_1), \dots, (u_m,v_m), (p,kp))
\]
represent the same link in \(S^3\), where the pair \((p,\overline q)\) is
omitted if \(q=0\).
\end{theorem}

\begin{proof}
The V-link
\[
V((u_1,\overline{v_1}), \dots, (u_m,\overline{v_m}), (r_1,s_1), \dots, (r_n,s_n), (p,kp+q))
\]
is given by the braid
\begin{align*}
B={}&(\sigma_{p-1}\sigma_{p-2}\dots\sigma_{{p-u_1+1}})^{v_1}
\dots
(\sigma_{p-1}\sigma_{p-2}\dots\sigma_{{p-u_m+1}})^{v_m}
(\sigma_1\sigma_2\dots\sigma_{r_1-1})^{s_1}
\dots\\
&(\sigma_1\sigma_2\dots\sigma_{r_n-1})^{s_n}(\sigma_1\sigma_2\dots\sigma_{{p-1}})^{kp+q}.
\end{align*}
By conjugating the braid, its closure is equivalent to the closure of the braid
\begin{align*}
B_1={}&(\sigma_1\sigma_2\dots\sigma_{r_1-1})^{s_1}
\dots
(\sigma_1\sigma_2\dots\sigma_{r_n-1})^{s_n}
(\sigma_1\sigma_2\dots\sigma_{{p-1}})^{kp+q}
(\sigma_{p-1}\sigma_{p-2}\dots\sigma_{{p-u_1+1}})^{v_1}
\dots\\
&(\sigma_{p-1}\sigma_{p-2}\dots\sigma_{{p-u_m+1}})^{v_m}.
\end{align*}
The braid \(B_1\) is a braid with \(p\) strands, and
\[
(\sigma_1\sigma_2\dots\sigma_{{p-1}})^{kp}
\]
is the \(k\)-th power of the positive full twist on \(p\) strands. Since the full twist is central in the braid group \(B_p\), the closure of \(B_1\) is equivalent to the closure of the braid
\begin{align*}
B_2={}&(\sigma_1\sigma_2\dots\sigma_{r_1-1})^{s_1}
\dots
(\sigma_1\sigma_2\dots\sigma_{r_n-1})^{s_n}
(\sigma_1\sigma_2\dots\sigma_{{p-1}})^q
(\sigma_{p-1}\sigma_{p-2}\dots\sigma_{{p-u_1+1}})^{v_1}
\dots\\
&(\sigma_{p-1}\sigma_{p-2}\dots\sigma_{{p-u_m+1}})^{v_m}
(\sigma_1\sigma_2\dots\sigma_{{p-1}})^{kp}.
\end{align*}
Moreover,
\[
(\sigma_1\sigma_2\dots\sigma_{{p-1}})^{kp}
=
(\sigma_{p-1}\sigma_{p-2}\dots\sigma_1)^{kp},
\]
since both are the \(k\)-th power of the positive full twist on \(p\) strands. Hence the closure of \(B_2\) is equivalent to the closure of
\begin{align*}
B_3={}&(\sigma_1\sigma_2\dots\sigma_{r_1-1})^{s_1}
\dots
(\sigma_1\sigma_2\dots\sigma_{r_n-1})^{s_n}
(\sigma_1\sigma_2\dots\sigma_{{p-1}})^q
(\sigma_{p-1}\sigma_{p-2}\dots\sigma_{{p-u_1+1}})^{v_1}
\dots\\
&(\sigma_{p-1}\sigma_{p-2}\dots\sigma_{{p-u_m+1}})^{v_m}
(\sigma_{p-1}\sigma_{p-2}\dots\sigma_1)^{kp}.
\end{align*}
By Lemma~\ref{otherbraid}, this braid represents the V-link
\[
V((r_1,\overline{s_1}), \dots, (r_n,\overline{s_n}), (p,\overline{q}), (u_1,v_1), \dots, (u_m,v_m), (p,kp)).
\]
Therefore, the two V-links represent the same link in \(S^3\).
\end{proof}

\begin{corollary}\label{cor:four-T-link-presentations}
Let \(0\le q<p\) and \(k>0\). Consider the V-link
\[
L=
V((u_1,\overline{v_1}),\dots,(u_m,\overline{v_m}),
(r_1,s_1),\dots,(r_n,s_n),(p,kp+q)),
\]
where the overlined and non-overlined blocks are allowed to be empty. Then the
four T-link presentations obtained from \(L\) and from its rotated V-link
represent the same link in \(S^3\).

When \(m,n\ge 1\), these four presentations are
\begin{align*}
T_1
={}&T\big((r_1,s_1),\dots,(r_n,s_n),
(p,kp+q-u_m),                                      \\
&\hspace{1.4cm}
(p+v_m,u_m-u_{m-1}),\dots,
(p+v_m+\cdots+v_2,u_2-u_1),                       \\
&\hspace{1.4cm}
(p+v_m+\cdots+v_1,u_1)\big),                      \\[0.5em]
T_2
={}&T\big((u_1,v_1),\dots,(u_m,v_m),
(kp+q,p-r_n),                                      \\
&\hspace{1.4cm}
(kp+q+s_n,r_n-r_{n-1}),\dots,
(kp+q+s_n+\cdots+s_2,r_2-r_1),                    \\
&\hspace{1.4cm}
(kp+q+s_n+\cdots+s_1,r_1)\big),                   \\[0.5em]
T_3
={}&T\big((u_1,v_1),\dots,(u_m,v_m),
(p,(k-1)p),
(p+q,p-r_n),                                      \\
&\hspace{1.4cm}
(p+q+s_n,r_n-r_{n-1}),\dots,
(p+q+s_n+\cdots+s_2,r_2-r_1),                     \\
&\hspace{1.4cm}
(p+q+s_n+\cdots+s_1,r_1)\big),                    \\[0.5em]
T_4
={}&T\big((r_1,s_1),\dots,(r_n,s_n),
(p,q),
(kp,p-u_m),                                       \\
&\hspace{1.4cm}
(kp+v_m,u_m-u_{m-1}),\dots,
(kp+v_m+\cdots+v_2,u_2-u_1),                      \\
&\hspace{1.4cm}
(kp+v_m+\cdots+v_1,u_1)\big).
\end{align*}

If \(m=0\) or \(n=0\), the formulas are interpreted by applying
Corollary~\ref{cor:two-T-links-same} directly to the relevant V-link and
omitting the terms corresponding to the empty block. In all cases, the standard
simplifications are applied: pairs with second coordinate zero are omitted, and
consecutive pairs with the same first coordinate are combined by adding their
second coordinates.
\end{corollary}

\begin{proof}
By Theorem~\ref{thm:rotate-V-link}, the V-link
\[
V((u_1,\overline{v_1}),\dots,(u_m,\overline{v_m}),
(r_1,s_1),\dots,(r_n,s_n),(p,kp+q))
\]
is equivalent to
\[
V((r_1,\overline{s_1}),\dots,(r_n,\overline{s_n}),
(p,\overline q),(u_1,v_1),\dots,(u_m,v_m),(p,kp)),
\]
where the pair \((p,\overline q)\) is omitted when \(q=0\). Hence these two
V-links represent the same link \(L\) in \(S^3\).

Applying Corollary~\ref{cor:two-T-links-same} to the original V-link gives the
two T-link presentations \(T_1\) and \(T_2\). Applying the same corollary to the
rotated V-link gives the two T-link presentations \(T_3\) and \(T_4\). Therefore
all four presentations represent the same link \(L\).

When both sides are non-empty, the explicit formulas above are obtained by
substituting the parameters of the original and rotated V-links into
Corollary~\ref{cor:two-T-links-same}. If one side is empty, the same substitution
is made after omitting the corresponding empty block. Finally, omitting
zero-exponent pairs and combining consecutive pairs with the same first
coordinate are the standard simplifications of T-link presentations and do not
change the represented link.
\end{proof}

The next result gives simple sufficient conditions under which the four
T-link presentations in Corollary~\ref{cor:four-T-link-presentations} are
genuinely distinct after the standard simplifications.

\begin{theorem}\label{thm:four-distinct-T-link-presentations}
Let
\[
L=
V((u_1,\overline{v_1}),\dots,(u_m,\overline{v_m}),
(r_1,s_1),\dots,(r_n,s_n),(p,kp+q)),
\]
where the overlined and non-overlined blocks are allowed to be empty,
\(0\le q<p\), and \(k>1\). Then the four T-link presentations in
Corollary~\ref{cor:four-T-link-presentations} are distinct under any of the
following conditions:

\begin{enumerate}
\item \(m,n\ge 1\), \(0<q<p\), \(u_m<p\), and the ordered blocks
\[
((r_1,s_1),\dots,(r_n,s_n))
\quad\text{and}\quad
((u_1,v_1),\dots,(u_m,v_m))
\]
are distinct;

\item \(m=0\), \(n\ge 1\);

\item \(n=0\), \(m\ge 1\), \(0<q<p\), and \(u_m<p\).
\end{enumerate}
\end{theorem}

\begin{proof}
By Corollary~\ref{cor:four-T-link-presentations}, the four T-link
presentations \(T_1,T_2,T_3,T_4\) represent the same link \(L\), after the
standard simplifications. We prove that they are pairwise distinct in each of
the three cases.

Assume first that \(m,n\ge1\). Under the hypotheses
\[
0<q<p,\qquad k>1,\qquad u_m<p,
\]
the pairs
\[
(p,q),\qquad (p,(k-1)p),\qquad (kp,p-u_m)
\]
all have positive second coordinate and are not removed. Moreover, \(u_m<p\)
prevents the pair \((p,(k-1)p)\) in \(T_3\) from being combined with the
preceding pair \((u_m,v_m)\), and \(k>1\) prevents the pairs \((p,q)\) and
\((kp,p-u_m)\) in \(T_4\) from being combined. Hence \(T_1,T_2\) have
\(m+n+1\) parameter pairs, whereas \(T_3,T_4\) have \(m+n+2\) parameter pairs.
Thus no presentation among \(T_1,T_2\) can coincide with one among
\(T_3,T_4\).

It remains to distinguish \(T_1\) from \(T_2\), and \(T_3\) from \(T_4\).
The presentations \(T_1\) and \(T_2\) begin respectively with the ordered
blocks
\[
((r_1,s_1),\dots,(r_n,s_n))
\quad\text{and}\quad
((u_1,v_1),\dots,(u_m,v_m)).
\]
These blocks are distinct by hypothesis, so \(T_1\neq T_2\). 
Similarly, \(T_3\) begins with the block
\[
((u_1,v_1),\dots,(u_m,v_m)),
\]
whereas \(T_4\) begins with
\[
((r_1,s_1),\dots,(r_n,s_n)).
\]
Since these ordered blocks are distinct, \(T_3\neq T_4\).
Thus the four presentations are distinct in
case (1).

Now assume that \(m=0\) and \(n\ge1\). Then the four presentations reduce to
\[
T_1=
T((r_1,s_1),\dots,(r_n,s_n),(p,kp+q)),
\]
\[
T_2=
T((kp+q,p-r_n),(kp+q+s_n,r_n-r_{n-1}),\dots,
(kp+q+s_n+\cdots+s_1,r_1)),
\]
\[
T_3=
T((p,(k-1)p),(p+q,p-r_n),
(p+q+s_n,r_n-r_{n-1}),\dots,
(p+q+s_n+\cdots+s_1,r_1)),
\]
and
\[
T_4=
T((r_1,s_1),\dots,(r_n,s_n),(p,q),(kp,p)),
\]
with the usual simplifications when \(q=0\). Since \(k>1\), we have
\(kp+q\ge 2p\). Thus \(T_2\) begins with first coordinate \(kp+q>p\), \(T_3\)
begins with first coordinate \(p\), and \(T_1\) begins with first coordinate
\(r_1<p\). Also \(T_4\) begins with the same initial block as \(T_1\), but
after this block \(T_1\) has the pair \((p,kp+q)\), whereas \(T_4\) has either
\((p,q),(kp,p)\) if \(q>0\), or \((kp,p)\) if \(q=0\). In either case this is
different from \((p,kp+q)\). Hence the four presentations are pairwise
distinct in case (2).

Finally assume that \(n=0\), \(m\ge1\), \(0<q<p\), and \(u_m<p\). Then the four
presentations reduce to
\[
T_1=
T((p,kp+q-u_m),
(p+v_m,u_m-u_{m-1}),\dots,
(p+v_m+\cdots+v_1,u_1)),
\]
\[
T_2=
T((u_1,v_1),\dots,(u_m,v_m),(kp+q,p)),
\]
\[
T_3=
T((u_1,v_1),\dots,(u_m,v_m),(p,(k-1)p),(p+q,p)),
\]
and
\[
T_4=
T((p,q),(kp,p-u_m),
(kp+v_m,u_m-u_{m-1}),\dots,
(kp+v_m+\cdots+v_1,u_1)).
\]
The assumptions \(0<q<p\), \(k>1\), and \(u_m<p\) ensure that the pairs
\[
(p,q),\qquad (p,(k-1)p),\qquad (kp,p-u_m)
\]
are not omitted and are not combined in the relevant places. Therefore
\(T_1,T_2\) have \(m+1\) parameter pairs, while \(T_3,T_4\) have \(m+2\)
parameter pairs. Hence no presentation among \(T_1,T_2\) can coincide with one
among \(T_3,T_4\).

Moreover, \(T_1\) begins with first coordinate \(p\), while \(T_2\) begins with
first coordinate \(u_1<p\), since \(u_m<p\). Thus \(T_1\neq T_2\). Similarly,
\(T_3\) begins with first coordinate \(u_1<p\), while \(T_4\) begins with first
coordinate \(p\). Thus \(T_3\neq T_4\). Hence the four presentations are
pairwise distinct in case (3).
\end{proof}

\begin{remark}
The hypotheses in Theorem~\ref{thm:four-distinct-T-link-presentations}
are sufficient but not necessary. They are imposed only to avoid the obvious
degeneracies in Corollary~\ref{cor:four-T-link-presentations}. In special
cases, one may still obtain three or four distinct T-link presentations even if
some of these hypotheses fail.
\end{remark}

\begin{remark}
The exceptional cases in
Proposition~\ref{prop:when-two-T-presentations-coincide} only describe when the
two T-link presentations associated to a single V-link coincide. They do not
give a complete description of the possible coincidences among the four
presentations in Corollary~\ref{cor:four-T-link-presentations}. In fact, even
when \(T_1=T_2\), the presentations coming from the rotated V-link may still be
new. For example, if \(Q=kp+q\), with \(k>1\) and \(0<q<p\), then
\[
L=V((p,\overline{Q-p}),(p,Q))
\]
satisfies
\[
T_1=T_2=T((p,Q-p),(Q,p)),
\]
whereas the rotated presentation gives
\[
T_3=T((p,2(k-1)p+q),(p+q,p)),
\]
after simplification. Hence \(T_3\neq T_1\).
\end{remark}

\begin{theorem}\label{thm:many-V-presentations-general}
Let \(k>1\) be an integer and let \(0\le q<p\). Consider the V-link
\[
L=
V((u_1,\overline{v_1}), \dots, (u_m,\overline{v_m}),
(r_1,s_1), \dots, (r_n,s_n), (p,kp+q)).
\]
Then \(L\) admits at least \(k-1\) distinct V-link presentations. More
precisely, for each \(a=1,\dots,k-1\), \(L\) is equivalent to the V-link
\[
V((r_1,\overline{s_1}),\dots,(r_n,\overline{s_n}),
(p,\overline{ap+q}),
(u_1,v_1),\dots,(u_m,v_m),(p,(k-a)p)),
\]
after the standard simplifications of combining consecutive blocks on the same
side with the same strand number.
\end{theorem}

\begin{proof}
The proof is the same full-twist redistribution argument used in
Theorem~\ref{thm:rotate-V-link}. For each \(a=1,\dots,k-1\), write
\[
kp+q=(ap+q)+(k-a)p.
\]
The factor
\[
(\sigma_1\cdots\sigma_{p-1})^{(k-a)p}
\]
is a power of the positive full twist on \(p\) strands, and hence is central in
\(B_p\). Therefore it can be moved through the braid without changing the
closure.

Using the same conjugation and centrality argument as in
Theorem~\ref{thm:rotate-V-link}, the part \(ap+q\) is transferred to the
overlined side, while the remaining \((k-a)p\) remains as the final full-twist
block. This gives the V-link presentation
\[
V((r_1,\overline{s_1}),\dots,(r_n,\overline{s_n}),
(p,\overline{ap+q}),
(u_1,v_1),\dots,(u_m,v_m),(p,(k-a)p)),
\]
after the standard simplifications.

For different values of \(a\), the overlined block with strand number \(p\) has
different exponent \(ap+q\). Since \(1\le a\le k-1\), these exponents are
positive and distinct. Moreover, this overlined block is not removed by any
standard simplification. Hence the resulting V-link presentations are distinct.
Therefore \(L\) has at least \(k-1\) distinct V-link presentations.
\end{proof}

\begin{remark}
Corollary~\ref{cor:two-T-links-same} applied to the V-link presentations
produced by Theorem~\ref{thm:many-V-presentations-general}, also gives many
T-link presentations of the same link.  The exact number of distinct T-link
presentations depends on the parameters, since some of the resulting
presentations may coincide or may simplify to the same parameter list.
Nevertheless, the number of presentations produced by this construction grows
with the number of full twists in the final V-link block.

This illustrates one advantage of the V-link language.  In a V-link
presentation, the source of this non-uniqueness in Theorem~\ref{thm:many-V-presentations-general} is visible: it comes from
redistributing the full-twist part of the final block \((p,kp+q)\).  In the
corresponding T-link presentations, the same non-uniqueness is encoded more
indirectly in the parameters and is therefore harder to detect.
\end{remark}

\begin{theorem}\label{thm:fixed-braid-index-many-presentations}
Fix an integer \(p\ge 3\). For every \(k\ge 1\), there exists a V-link \(L_k\)
with braid index \(p\) such that \(L_k\) has at least \(k\) distinct V-link
presentations. Moreover, \(L_k\) has at least \(2k\) distinct T-link
presentations.
\end{theorem}

\begin{proof}
Fix \(p\ge 3\) and \(k\ge 1\), and consider
\[
L_k=V((2,1),(p,kp+1)).
\]
This link is represented by the braid-index-realizing positive braid
\[
\sigma_1(\sigma_1\sigma_2\cdots\sigma_{p-1})^{kp+1}.
\]
Hence
\[
\operatorname{br}(L_k)=p.
\]

For each \(a=0,\dots,k-1\), write
\[
kp+1=(ap+1)+(k-a)p.
\]
For \(a=0\), Theorem~\ref{thm:rotate-V-link} gives the V-link presentation
\[
V_0=
V((2,\overline{1}),(p,\overline{1}),(p,kp)).
\]
For \(a=1,\dots,k-1\), Theorem~\ref{thm:many-V-presentations-general} gives
the V-link presentation
\[
V_a=
V((2,\overline{1}),(p,\overline{ap+1}),(p,(k-a)p)).
\]
Thus, for every \(a=0,\dots,k-1\), the link \(L_k\) admits the V-link
presentation
\[
V_a=
V((2,\overline{1}),(p,\overline{ap+1}),(p,(k-a)p)).
\]
These \(k\) V-link presentations are distinct, because the overlined parameter
\[
(p,\overline{ap+1})
\]
depends on \(a\). Therefore \(L_k\) has at least \(k\) distinct V-link
presentations.

Now apply Corollary~\ref{cor:two-T-links-same} to each \(V_a\). Since
\[
V_a=
V((2,\overline{1}),(p,\overline{ap+1}),(p,(k-a)p)),
\]
the overlined side has parameters
\[
(u_1,v_1)=(2,1),
\qquad
(u_2,v_2)=(p,ap+1),
\]
and the non-overlined side is empty. Hence the first T-link presentation
associated to \(V_a\) is
\[
T_a^{(1)}
=
T((p,(k-a-1)p),(p+ap+1,p-2),(p+ap+2,2)),
\]
where the pair \((p,(k-a-1)p)\) is omitted when \(a=k-1\).

The second associated T-link presentation is
\[
T_a^{(2)}
=
T((2,1),(p,ap+1),((k-a)p,p)).
\]
Standard simplifications are applied throughout. In particular, when
\(a=k-1\), the last pair in \(T_a^{(2)}\) is \((p,p)\), so it is combined with
the preceding pair \((p,(k-1)p+1)\).

We now show that the resulting \(2k\) T-link presentations are pairwise
distinct. First, the presentations \(T_a^{(1)}\) are distinct as \(a\) varies:
if \(a<k-1\), the second pair has first coordinate
\[
p+ap+1,
\]
and if \(a=k-1\), the first displayed pair is omitted but the resulting first
pair still has first coordinate
\[
p+(k-1)p+1.
\]
Thus the strand number \(p+ap+1\) distinguishes the presentations
\(T_a^{(1)}\).

Next, the presentations \(T_a^{(2)}\) are distinct as \(a\) varies. For
\(a<k-1\), no simplification combines the pair \((p,ap+1)\) with the last pair,
because
\[
(k-a)p>p.
\]
Thus the second pair is \((p,ap+1)\). For \(a=k-1\), the last pair is
\[
((k-a)p,p)=(p,p),
\]
and the standard simplification gives
\[
T_{k-1}^{(2)}=T((2,1),(p,kp+1)).
\]
Hence the second pair has exponent \(ap+1\) for \(a<k-1\), and exponent
\(kp+1\) for \(a=k-1\). These exponents are all distinct, so the presentations
\(T_a^{(2)}\) are pairwise distinct.

Finally, no presentation \(T_a^{(1)}\) is equal to a presentation
\(T_b^{(2)}\). Indeed, after the standard simplifications, every
\(T_b^{(2)}\) begins with the pair
\[
(2,1),
\]
whereas \(T_a^{(1)}\) begins either with a pair whose first coordinate is \(p\),
if \(a<k-1\), or with the pair
\[
(p+ap+1,p-2),
\]
if \(a=k-1\). In both cases, since \(p\ge 3\), the first coordinate of the first
pair of \(T_a^{(1)}\) is strictly larger than \(2\). Therefore
\[
T_a^{(1)}\neq T_b^{(2)}
\]
for all \(a,b\).

Thus the \(2k\) T-link presentations
\[
T_a^{(1)},\ T_a^{(2)}
\qquad
(a=0,\dots,k-1)
\]
are pairwise distinct. Hence \(L_k\) has at least \(2k\) distinct T-link
presentations.
\end{proof}

The V-link description makes the non-uniqueness of T-link presentations
more transparent.  From the point of view of a T-link presentation, it is not
always clear how many different parameter sequences can represent the same
link.  In contrast, in a V-link presentation with final block \((p,kp+q)\), the
integer \(k\) records the number of full twists on \(p\) strands.  Since full
twists are central in the braid group, they can be separated in different ways,
and each separation gives a new V-link presentation of the same link.  Applying
the correspondence between V-links and T-links then gives new T-link
presentations of the same link.

\section{Satellite V-links and T-links}\label{Satellite}\label{section6}

In this section we give conditions on the parameters of V-links to yield satellite links. We then apply these conditions to find new families of satellite T-links. 

\begin{theorem}\label{satellitecaseone}
Let $r_1, \dots, r_n, s_1, \dots, s_n, u_1, \dots, u_m, v_1, \dots, v_m, p, q$ be positive integers such that $2\leq r_1< \dots < r_{n} < p$, $2\leq u_1< \dots < u_{m} \leq p\leq q$.
Suppose that there are positive numbers $i, j,$ and $d>1$ such that 
$r_1, \dots, r_i, u_1, \dots, u_j$ are less than or equal to $d$ and 
$r_{i+1}, \dots, r_n, s_{i+1}, \dots, s_n, u_{j+1}, \dots, u_m, v_{j+1}, \dots, v_m, p, q$ are multiples of and greater than $d$. Then, the V-link  
\begin{align*}
&V_1 = V((u_1,\overline{v_1}), \dots, (u_{j},\overline{v_{j}}), (u_{j+1},\overline{v_{j+1}}), \dots, (u_{m},\overline{v_{m}}), (r_1,s_1), \dots, (r_{i},s_{i}), (r_{i+1},s_{i+1}), \dots,\\
&(r_{n},s_{n}), (p, q))
\end{align*}
is the satellite link with companion 
$$C = V((u_{j+1}/d, \overline{v_{j+1}}/d), \dots, (u_{m}/d,\overline{v_{m}}/d), (r_{i+1}/d, s_{i+1}/d), \dots, (r_{n}/d, s_{n}/d), (p/d, q/d))$$
such that if $C$ is a knot then the pattern is the V-link given by the braid
$$(\sigma_{d-1}\dots\sigma_{{d-u_1+1}})^{v_1}\dots(\sigma_{d-1}\dots\sigma_{{d-u_j+1}})^{v_j}(\sigma_1\dots\sigma_{r_1-1})^{s_1}\dots(\sigma_1\dots\sigma_{r_i-1})^{s_i}(\sigma_{1}\dots\sigma_{{d-1}})^{d\eta}$$
where $\eta = \sum_{a=i+1}^{n}r_{a}s_{a}/d^2 + \sum_{b=j+1}^{m}u_{a}v_{b}/d^2 + pq/d^2.$
\end{theorem}

\begin{proof}
We start with the braid $B$ of $V_1$ with $p$ strands 
\begin{align*}
&(\sigma_{p-1}\dots\sigma_{{p-u_1+1}})^{v_1}\dots (\sigma_{p-1}\dots\sigma_{{p-u_j+1}})^{v_j}(\sigma_{p-1}\dots\sigma_{{p-u_{j+1}+1}})^{v_{j+1}}\dots (\sigma_{p-1}\dots\sigma_{{p-u_m+1}})^{v_m}\\
&(\sigma_1\dots\sigma_{r_1-1})^{s_1}\dots(\sigma_{1}\dots\sigma_{r_i-1})^{s_i}(\sigma_{1}\dots\sigma_{r_{i+1}-1})^{s_{i+1}}\dots(\sigma_{1}\dots\sigma_{r_n-1})^{s_n}(\sigma_{1}\dots\sigma_{{p-1}})^{q}.
\end{align*}
We consider a set of $p/d$ circles on top of $B$ encircling $d$ adjacent parallel strands, as illustrated in Figure~\ref{Satellite}. We fix these circles on top of the braid and push copies of them down. This creates $p/d$ tubes. As $u_1, \dots, u_j, r_1, \dots, r_i \leq d$, the sub-braids provided by the parameters $(u_1,\overline{v_1}), \dots, (u_{j}, \overline{v_j})$ and  $(r_1,s_1), \dots, (r_{i},s_{i})$ are completely inside these tubes and do not change the shape of these tubes.
However, as we keep pushing these tubes down, they start to change following the strands of $B$. 
But since $u_{j+1}, \dots, u_m, v_{j+1}, \dots, v_m, r_{i+1}, \dots, r_n, s_{i+1}, \dots, s_n, p, q$ are multiples of $d$, every time these tubes pass through a sub-braid provided by the parameters$$(u_{j+1},\overline{v_{j+1}}), \dots, (u_{m},\overline{v_{m}}), (r_{i+1},s_{i+1}), \dots, (r_{n},s_{n}), \text{ or }(p, q),$$they still encircle $d$ adjacent parallel strands with the shape as at the top but they  possibly encircle different strands. 
Therefore, when they arrive at the bottom of $B$, they are in the same shape as at the top but encircling possibly different  adjacent parallel strands. This guarantees that if we push these tubes around the braid closure, they close to give the boundary of a tubular neighbourhood of a link $C$. 
More precisely, when these tubes pass through a sub-braid 
$$(u_{j+1},\overline{v_{j+1}}), \dots, (u_{m},\overline{v_{m}}), (r_{i+1},s_{i+1}), \dots, (r_{n},s_{n}), \text{ or }(p, q))$$their cores become the braid
$$(u_{j+1}/d,\overline{v_{j+1}/d}), \dots, (u_{m}/d,\overline{v_{m}/d}), (r_{i+1}/d,s_{i+1}/d), \dots, (r_{n}/d,s_{n}/d), \text{ or }(p/d, q/d)),$$respectively, as illustrated in Figure~\ref{Satellite}.
Hence, the companion of $V_1$ is the V-link 
$$C = V((u_{j+1}/d, \overline{v_{j+1}}/d), \dots, (u_{m}/d,\overline{v_{m}}/d), (r_{i+1}/d, s_{i+1}/d), \dots, (r_{n}/d, s_{n}/d), (p/d, q/d)).$$

\begin{figure}
\includegraphics[scale=0.5]{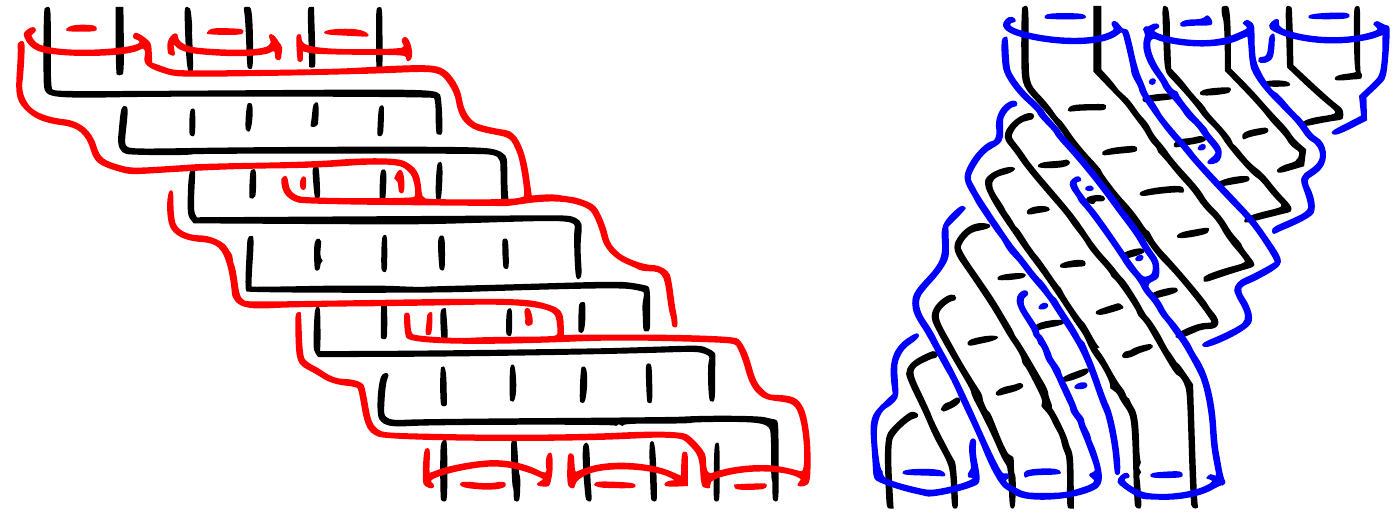} 
\caption{In the first drawing we have the braid $(6, 6)=(\sigma_1\sigma_2\sigma_3\sigma_4\sigma_5)^6$ and the red tubes are encircling 2 strands with cores given by the braid $(3, 3)=(\sigma_1\sigma_2)^3$. In the second drawing we have the braid 
$(6, \overline{6}) =(\sigma_5\sigma_4\sigma_3\sigma_2\sigma_1)^6$ and the blue tubes are encircling 2 strands with cores given by the braid $(3, \overline{3}) =(\sigma_2\sigma_1)^3$.}
\label{Satellite}
\end{figure}

Now we find the pattern considering that $C$ is knot. The pattern is the link inside the companion. So it contains the sub-braids
$$(u_1,\overline{v_1}), \dots, (u_{j},\overline{v_{j}}), (r_1,s_1), \dots, (r_{i},s_{i}).$$
In addition, it has a full twist for every horizontal line of the braid of $C$. 
There are $$b = \sum_{x=j+1}^{m}v_{x}/d + \sum_{z=i+1}^{n}s_{z}/d + q/d$$ horizontal lines in the braid diagram of $C$.
Furthermore, since we preserve the standard longitudes in our definition of satellite links, the pattern must have more 
$c$ full twists where $c$ is the number of crossings of the braid diagram of $C$. We have that
$$c = \sum_{x=j+1}^{m}(u_{x}/d-1)v_{x}/d + \sum_{z=i+1}^{n}(r_{z}/d-1)s_{z}/d + (p/d-1)q/d.$$
Since
$$b+c = \sum_{x=j+1}^{m}u_{x}v_{x}/d^2 + \sum_{z=i+1}^{n}r_{z}s_{z}/d^2 + pq/d^2.$$ 

Therefore, we conclude that the pattern $P$ of $V_1$ is the V-link given by the braid
$$(\sigma_{d-1}\dots\sigma_{{d-u_1+1}})^{v_1}\dots(\sigma_{d-1}\dots\sigma_{{d-u_j+1}})^{v_j}(\sigma_1\dots\sigma_{r_1-1})^{s_1}\dots(\sigma_1\dots\sigma_{r_i-1})^{s_i}(\sigma_{1}\dots\sigma_{{d-1}})^{d(b+c)}.$$
Hence $P$ is the V-link
$$V((u_1,\overline{v_1}), \dots, (u_{j},\overline{v_{j}}), (r_1,s_1), \dots, (r_{i},s_{i}), (d, d(b+c)))$$
if $r_{i}<d$ or
$$V((u_1,\overline{v_1}), \dots, (u_{j},\overline{v_{j}}), (r_1,s_1), \dots, (r_{i-1},s_{i-1}), (d, d(b+c) + s_{i}))$$
if $r_{i}=d$.
\end{proof}

\begin{corollary}\label{satellitecase1}
Let $r_1, \dots, r_n, s_1, \dots, s_n, u_1, \dots, u_m, v_1, \dots, v_m, p, q$ be positive integers such that $2\leq r_1< \dots < r_{n} < p$, $2\leq u_1< \dots < u_{m} \leq p\leq q$.
Suppose that there are positive numbers $i, j,$ and $d>1$ such that 
$r_1, \dots, r_i, u_1, \dots, u_j$ are less than or equal to $d$ and 
$r_{i+1}, \dots, r_n, s_{i+1}, \dots, s_n, u_{j+1}, \dots, u_m, v_{j+1}, \dots, v_m, p, q$ are multiples of and greater than $d$. 
Then, the T-links 
\begin{align*}
&T((r_1,s_1), \dots, (r_{i},s_{i}), (r_{i+1},s_{i+1}), \dots, (r_{n},s_{n}), (p, q-u_m), (p+ v_m, u_m - u_{m-1}), \dots,\\
&(p+ v_m+v_{m-1} + \dots + v_{j+1}, u_{j+1} - u_{j}), (p+ v_m+\dots + v_{j}, u_{j}-u_{j-1}), \dots,\\
&(p+ v_m+\dots + v_{2}, u_2-u_1), (p+ v_m + \dots + v_{1}, u_1)).
\end{align*}
and
\begin{align*}
&T((u_1, v_1), \dots, (u_{j}, v_{j}), (u_{j+1}, v_{j+1}), \dots, (u_{m}, v_{m}), (q, p-r_n), (q+ s_n, r_n - r_{n-1}), \dots,\\
&(q+s_n+s_{n-1} + \dots + s_{i+1}, r_{i+1}-r_{i}), (q+ s_n+s_{n-1} + \dots + s_{i}, r_{i}-r_{i-1}), \dots,\\
&(q+s_n+s_{n-1} + \dots + s_{2}, r_2-r_1), (q+ s_n+s_{n-1} + \dots + s_{1}, r_1))
\end{align*}
are satellite links with companions and patterns described in Theorem~\ref{satellitecaseone}.
\end{corollary}

\begin{proof}
It follows from Theorems~\ref{satellitecaseone} and \ref{equivalence}. 
\end{proof}

In particular, Corollary~\ref{satellitecase1} generalizes \cite[Theorem 5.4]{dePaivaPurcell2024}.

\begin{theorem}\label{satellitecase2}
Let $r_1, \dots, r_n, s_1, \dots, s_n, u_1, \dots, u_m, v_1, \dots, v_m, a, b, c, d, q$ be positive integers such that $2\leq r_1< \dots < r_{n} < qa$, $2\leq u_1< \dots < u_{m} \leq c$, $c+d = q$, $qb + c\geq aq$, and $q, a>1$.
Suppose that there is a natural number $i$ such that $r_1, \dots, r_i$ are less than or equal to $d$, $qa-c\leq r_{i+1}$, and $s_{i+1}+\dots +s_n = d$. Then, the V-link  
$$V_1 = V((u_1,\overline{v_1}), \dots, (u_{m},\overline{v_{m}}), (r_1,s_1), \dots, (r_{i},s_{i}), (r_{i+1},s_{i+1}), \dots, (r_{n},s_{n}), (qa, qb + c))$$
is the satellite link with companion the V-link $C = V(a, b+1)$. If $C$ is a knot, then the pattern is the V-link given by the braid
\begin{align*}
&(\sigma_{q-1}\dots\sigma_{{q-u_1+1}})^{v_1}\dots(\sigma_{q-1}\dots\sigma_{{q-u_m+1}})^{v_m}(\sigma_1\dots\sigma_{r_1-1})^{s_1}\dots(\sigma_{1}\dots\sigma_{r_i-1})^{s_i}\\
&(\sigma_{1}\dots\sigma_{d-1})^{d}(\sigma_{1}\dots\sigma_{{q-1}})^{q[b+ (a-1)(b+1)] +c}.
\end{align*}
\end{theorem}

\begin{proof}
The V-link $V_1$ is given by the following braid with $qa$
\begin{align*}
&B = (\sigma_{qa-1}\dots\sigma_{{qa-u_1+1}})^{v_1}\dots(\sigma_{qa-1}\dots\sigma_{{qa-u_m+1}})^{v_m}(\sigma_1\dots\sigma_{r_1-1})^{s_1}\dots(\sigma_{1}\dots\sigma_{r_i-1})^{s_i}\\
&(\sigma_{1}\dots\sigma_{r_{i+1}-1})^{s_{i+1}}\dots(\sigma_{1}\dots\sigma_{r_n-1})^{s_n}(\sigma_{1}\dots\sigma_{{qa-1}})^{qb + c}.
\end{align*}

To better visualize the construct of the essential torus, we push the sub-braid $(\sigma_{1}\dots\sigma_{{qa-1}})^{qb + c}$ of $B$ around the braid closure to place it on top of $B$ so that we obtain the braid
\begin{align*}
&B' = (\sigma_{1}\dots\sigma_{{qa-1}})^{qb + c}(\sigma_{qa-1}\dots\sigma_{{qa-u_1+1}})^{v_1}\dots(\sigma_{qa-1}\dots\sigma_{{qa-u_m+1}})^{v_m}(\sigma_1\dots\sigma_{r_1-1})^{s_1}\dots\\
&(\sigma_{1}\dots\sigma_{r_i-1})^{s_i}(\sigma_{1}\dots\sigma_{r_{i+1}-1})^{s_{i+1}}\dots(\sigma_{1}\dots\sigma_{r_n-1})^{s_n}.
\end{align*}

As before, we start with a set of $a$ circles on top of $B'$ encircling $q$ adjacent parallel strands as illustrated in Figure~\ref{Satellite2}.
\begin{figure}
\includegraphics[scale=0.3]{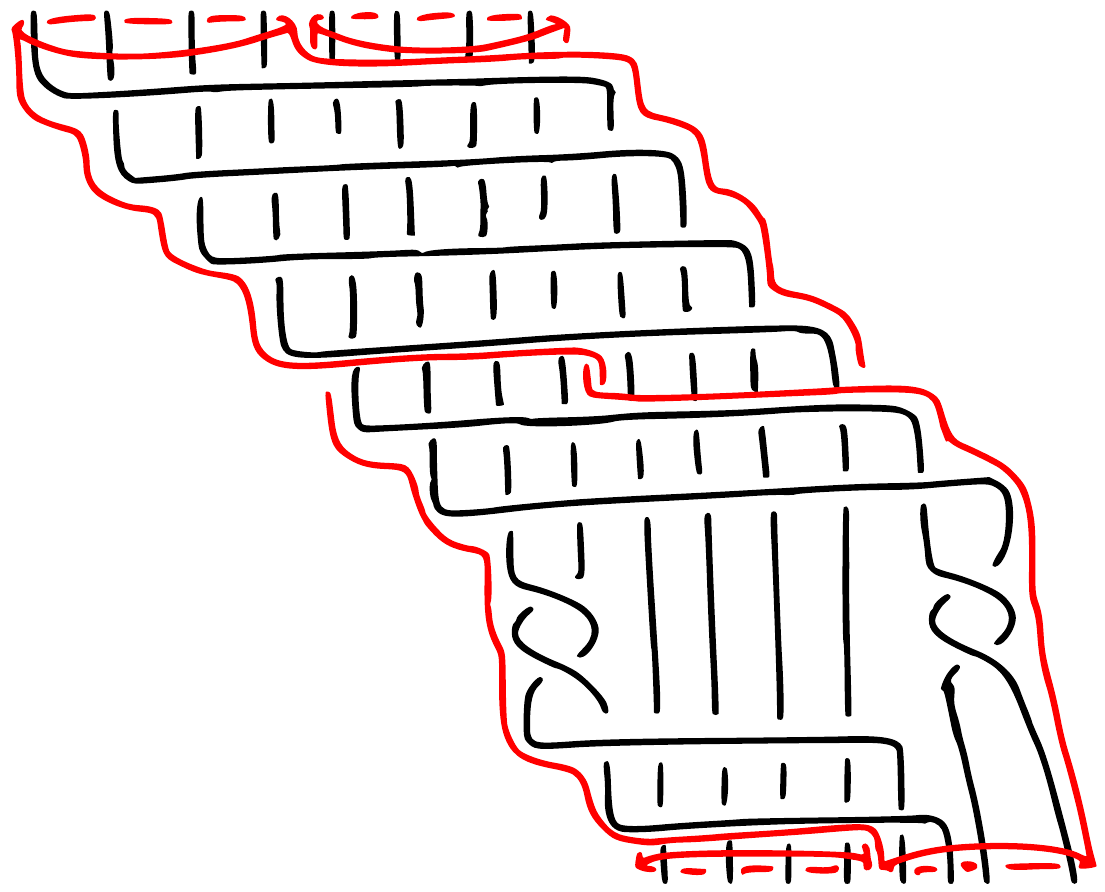}  
\caption{This example illustrates part of the satellite V-link with parameters $V((2, \overline{2}), (2, 2), (6, 2), (8, 10))$. Its companion is the V-link $V(2, 3)$ and the pattern is the V-link
$V((2, \overline{2}), (2, 4), (4, 2+ 4\times 5))$.}
\label{Satellite2}
\end{figure}
Then, we push copies of these circles down to be below the sub-braid $(\sigma_{1}\dots\sigma_{{qa-1}})^{qb}$. As a set, they are in the same positions as at the top but in different orders. Their cores are given by the braid $(\sigma_{1}\dots\sigma_{{a-1}})^{b}$. We consider the first circle $C_1$ in this position. As we keep pushing it down, it goes to the right to enclose the last $c$ horizontal strands of $(\sigma_{1}\dots\sigma_{{qa-1}})^{qb + c}$ and down to enclose the first $q-c$ strands at the bottom. 
As $2\leq u_1< \dots < u_{m} \leq c$, then after passing through the last $c$ horizontal strands of $(\sigma_{1}\dots\sigma_{{qa-1}})^{qb + c}$, this part of $C_1$ encloses the sub-braid provided by the parameters $(\sigma_{qa-1}\dots\sigma_{{qa-u_1+1}})^{v_1}\dots(\sigma_{qa-1}\dots\sigma_{{qa-u_m+1}})^{v_m}$ for this part of $C_1$ to end up encircling the last $c$ strands. 
As $r_1, \dots, r_i$ are less than or equal to $d$, $qa-c\leq r_{i+1}$, and $s_{i+1}+\dots +s_n = d$, then
the part of $C_1$ that was enclosing the first $q-c = d$ strands  keep going down to enclose the sub-braid provided by the parameters $(\sigma_1\dots\sigma_{r_1-1})^{s_1}\dots(\sigma_{1}\dots\sigma_{r_i-1})^{s_i}$ until it starts going to the right following the horizontal strands of the sub-braid $(\sigma_{1}\dots\sigma_{r_{i+1}-1})^{s_{i+1}}\dots(\sigma_{1}\dots\sigma_{r_n-1})^{s_n}$ to end up enclosing from the $(qa-q+1)$-th strand to the $(qa-c)$-th strand of $B'$. Since $c+d = q$, then $C_1$ becomes a circle encircling the last $q$ strands at the bottom of $B'$. 

At the bottom of $(\sigma_{1}\dots\sigma_{{qa-1}})^{qb}$, the other circles go down without any distortion to encircling $q$ adjacent parallel strands. So, a horizontal line is added to the braid that gives the core of the tubes formed by pushing the circles down. Therefore, their cores become the braid $(\sigma_{1}\dots\sigma_{{a-1}})^{b+1}$.
Hence after pushing the $a$ circles from the top to the bottom of $B'$, they end up in the same positions but in different orders. This ensures that if we push these circles around the braid closure, they close to give the boundary of the tubular neighbourhood of the torus link $C = T(a, b+1)$ or the V-link $V(a, b+1)$.  

Now we find the pattern considering that the companion $C$ is a knot. The pattern is the link within the companion. Then, the pattern is given by a braid with $q$ strands that has the sub-braid provided by the parameters 
$$(u_1,\overline{v_1}) \dots (u_{m},\overline{v_{m}})(r_1,s_1) \dots (r_{i},s_{i})(d, d)(q, qb + c).$$
Furthermore, since we preserve the standard longitudes, we need to add the number of crossings of $C$ in full twists to the pattern. Hence the pattern is the V-link given by the braid
\begin{align*}
&(\sigma_{q-1}\dots\sigma_{{q-u_1+1}})^{v_1}\dots(\sigma_{q-1}\dots\sigma_{{q-u_m+1}})^{v_m}(\sigma_1\dots\sigma_{r_1-1})^{s_1}\dots(\sigma_{1}\dots\sigma_{r_i-1})^{s_i}\\
&(\sigma_{1}\dots\sigma_{d-1})^{d}(\sigma_{1}\dots\sigma_{{q-1}})^{q[b+ (a-1)(b+1)] +c}.
\end{align*}
\end{proof}


\begin{corollary}\label{corollarysatellitecase2}
Let $r_1, \dots, r_n, s_1, \dots, s_n, u_1, \dots, u_m, v_1, \dots, v_m, a, b, c, d, q$ be positive integers such that $2\leq r_1< \dots < r_{n} < qa$, $2\leq u_1< \dots < u_{m} \leq c$, $c+d = q$, $qb + c\geq aq$, and $q, a>1$.
Suppose that there is a natural number $i$ such that $r_1, \dots, r_i$ are less than or equal to $d$, $qa-c\leq r_{i+1}$, and $s_{i+1}+\dots +s_n = d$. Then, the T-links 
\begin{align*}
&T((r_1,s_1), \dots, (r_{i},s_{i}), (r_{i+1},s_{i+1}), \dots, (r_{n},s_{n}), (p, q-u_m), (p+ v_m, u_m - u_{m-1}), \dots,\\
&(p+ v_m+\dots + v_{2}, u_2-u_1), (p+ v_m + \dots + v_{1}, u_1))
\end{align*}
and
\begin{align*}
&T((u_1, v_1), \dots, (u_{m}, v_{m}), (q, p-r_n), (q+ s_n, r_n - r_{n-1}), \dots, (q+s_n+ \dots + s_{i+1}, r_{i+1}-r_{i}), \\
&(q+ s_n+\dots + s_{i}, r_{i}-r_{i-1}), \dots, (q+s_n+ \dots + s_{2}, r_2-r_1), (q+ s_n+ \dots + s_{1}, r_1))
\end{align*}
are satellite links with companions and patterns described in Theorem~\ref{satellitecase2}.
\end{corollary}

\begin{proof}
It follows from Theorems~\ref{satellitecase2} and \ref{equivalence}. 
\end{proof}

In particular, Corollary~\ref{corollarysatellitecase2} generalizes \cite[Theorem 6.1]{de2021satellites}.

\section{Hyperbolic knots and positive braids with a full twist}\label{Hyperbolic}

In this section we find some mild new conditions on positive braids with at least one full twist to produce hyperbolic knots.

\subsection{Non-satellite knots given by braids}
In this subsection we prove that, under these mild conditions, knots given by positive braids with at least one full twist are atoroidal, meaning that they do not have any essential torus, which is equivalent to saying that they are not satellite knots.

\begin{definition}
Consider an integer $q>1$. A generalized $q$-cabling of a link $L$ is a link $L'$ contained in the
interior of a tubular neighbourhood $L\times D^2$ of $L$ such that
\begin{enumerate}
\item each fiber $D^2$ intersects $L'$ transversely in $q$ points;  and

\item all strands of $L'$ are oriented in the same direction as $L$ itself.
\end{enumerate} 
\end{definition}

The next theorem was proved by Williams in \cite{Williams}.
\begin{theorem}[Williams, \cite{Williams}]\label{Williams}
If $L$ is a link with each component a non-trivial knot and $L'$ is a generalized $q$-cabling of $L$,
then $$\beta(L') = q\beta(L),$$ where $\beta(*)$ is the braid index of $*$.
\end{theorem}

The case where $L$ is a knot given by a positive braid with a full twist, which is the case we are interested in this paper, also follows from Birman and Menasco \cite[Corollary 3]{MR1286930} and Ito \cite[Theorem 5.5, Lemma 3.5]{ito2024satellite}.

de Paiva considered the case where $L$ is a trivial knot.

\begin{lemma}[de Paiva, \cite{MR4494619}]\label{lemma2}
Let $L'$ be a generalized $q$-cabling of the unknot $L$, with $L$ given by a positive braid with $n$ strands, where $n>1$. Also, assume the knot inside $L$ is given by a positive braid. Then, $L'$ has braid index equal to $q$.
\end{lemma}

\begin{lemma}\label{lemma1}
Let $p, r, k$ be positive integers with $p>r>1$. Consider $\beta$ a braid with $r$ strands. Then, the knot $K$ given by the closure of the braid $\beta(\sigma_1\dots \sigma_{p-1})^{pk}$ is a link with more than one link component.
\end{lemma}

\begin{proof} 
The braid $(\sigma_1\dots \sigma_{p-1})^{pk}$ is obtained by applying $k$ full twists on $p$ unlink components. Since full twists are homeomorphisms, the closure of $(\sigma_1\dots \sigma_{p-1})^{pk}$ also has $p$ link components. So $(\sigma_1\dots \sigma_{p-1})^{pk}$ is a braid with $p$ strands and $p$ link components. Since $\beta$ has less strands than $p$, this implies that $\beta(\sigma_1\dots \sigma_{p-1})^{pk}$ has at least two link components. More precisely, the last strand of the braid $\beta(\sigma_1\dots \sigma_{p-1})^{pk}$ forms a single link component and the first $p-1$ strands form at least one link component.
\end{proof}
 
\begin{lemma}\label{lemmaTori1}
Let $p, q, r, k$ be positive integers such that $p>q>r> gcd(p, q)> 1$ and $r$ has no common divisor with $p$ and $q$. Consider 
$\beta$ a positive braid with $r$ strands and at least one positive full twist on $r$ strands.
Suppose that the knot $K$ given by the closure of the positive braid $\beta(\sigma_1\dots \sigma_{p-1})^{q + pk}$ is a generalized $d$-cabling of a knotted torus $T$ that doesn't intersect the braid axis of $K$. Then, $d$ divides $p, q$ and there are positive integers $p', q'$ multiples of $d$ and a torus $T'$ with $p'>r>q'\geq d$ such that the closure $K'$ of one of the positive braids $\beta(\sigma_1\dots \sigma_{p'-1})^{q'}$ or $\beta(\sigma_{p'-1}\dots \sigma_{1})^{q'}$ is a generalized $d$-cabling of the core of $T'$. Furthermore, $T'$ doesn't intersect the braid axis of $K'$.
\end{lemma}

\begin{proof} 
After doing $(-1/k)$-Dehn surgery along the braid axis $C$ of $\beta(\sigma_1\dots \sigma_{p-1})^{q + pk}$, the knot $K$ becomes the knot $K_1$ given by the closure of the braid
$$B_1 = \beta(\sigma_1\dots \sigma_{p-1})^{q}$$ and $T$ becomes a new torus $T_1$ in the complement of $K_1$.

By Proposition~\ref{isopoty3}, $B_1$ is equivalent to the braid 
$$B_1' = (\sigma_{q-1}\dots \sigma_{1})^{p-q}\beta(\sigma_{1}\dots \sigma_{q-1})^{q}.$$
We have that $q$ can't divide $p$. Because suppose  $q$ divides $p$. 
By hypotheses, $r$ is less than $q$, but the sub-braids $(\sigma_{q-1}\dots \sigma_{1})^{p-q}$ and $(\sigma_{1}\dots \sigma_{q-1})^{q}$ are full twists on $q$, then the closure of $B_1'$ is a link with more than one component by Lemma~\ref{lemma1}, a contradiction. Hence $q$ doesn't divide $p$.

The braid $B_1$ is a positive braid with $q$ strands that has at least one positive full twist on $q$ strands. Therefore, its braid index is equal to $q$ by Franks and Williams \cite[Corollary 2.4]{Franks}.

If the torus $T_1$ is unknotted, then, by Lemma~\ref{lemma2}, $d$ is equal to $q$. But, this would imply that $q$ divides $p$, but we have already excluded this possibility. So, $T_1$ is knotted. Therefore, $T_1$ doesn't intersect the braid axis $C_1$ of $B_1'$ by Ito \cite[Theorem 5.5, Lemma 3.5]{ito2024satellite}. Thus, $d$ also divides $q$.

Consider that $p-q= k_2q + a_2$ with $0<a_2<q$. We see that $a_2$ is also a multiple of $d$. So, $a_2\neq r$. Then, after doing $(-1/(k_2+1))$-Dehn surgery along $C_1$, then $T_1$ becomes a new torus $T_2$ and the braid $B_1'$ becomes the braid  
$$B_2 = (\sigma_{q-1}\dots \sigma_{1})^{a_2}\beta.$$

If $a_2< r$, then we set $p' = q$, $q' = a_2$, and $T' = T_2$. Consider next that $a_2> r$.
We push the sub-braid $(\sigma_{q-1}\dots \sigma_{1})^{a_2}$ around the braid closure so that $B_2$ becomes the braid 
$\beta(\sigma_{q-1}\dots \sigma_{1})^{a_2}$. Then, by Proposition~\ref{isopoty4}, $B_2$ is equivalent to the braid
$$B_2' = \beta(\sigma_{1}\dots \sigma_{a_2-1})^{q-a_2}(\sigma_{a_2-1}\dots \sigma_{1})^{a_2}.$$

We have that $a_2$ can't divide $q$. Because suppose $a_2$ divides $q$. As $r$ is less than $a_2$ and the sub-braids 
$(\sigma_{1}\dots \sigma_{a_2-1})^{q-a_2}$ and $(\sigma_{a_2-1}\dots \sigma_{1})^{a_2}$ are full twists on $a_2$, then the closure of $B_2'$ is a link with more than one component by Lemma~\ref{lemma1}, a contradiction. Hence $a_2$ doesn't divide $q$.

The braid $B_1$ has braid index equal to $a_2$ \cite{Franks}. If the torus $T_2$ is unknotted, then, by Lemma~\ref{lemma2}, $d$ is equal to $a_2$. But, this would imply that $a_2$ divides $q$, but we have already excluded this possibility. So, $T_2$ is knotted. Therefore, $T_2$ doesn't intersect the braid axis $C_2$ of $B_2'$ \cite{ito2024satellite}. Thus, $d$ also divides $a_2$.

Consider that $q-a_2= k_3a_2 + a_3$ with $0<a_3<a_2$. We see that $a_3$ is also a multiple of $d$. So, $a_3\neq r$. Then, after doing $(-1/(k_3+1))$-Dehn surgery along $C_2$, then $T_2$ becomes a new torus $T_3$ and the braid $B_2'$ becomes the braid  
$$B_3 = \beta(\sigma_{1}\dots \sigma_{a_2-1})^{a_3}.$$

If $a_3< r$, then we set $p' = a_2$, $q' = a_3$, and $T' = T_3$. If not, then we consider that $a_3> r$. Since eventually this process finishes, we always find positive numbers $p', q'$ multiples of $d$ and a torus $T'$ with $p'>r>q'\geq d$ such that the closure $K'$ of one of the positive braids $\beta(\sigma_1\dots \sigma_{p'-1})^{q'}$ or $\beta(\sigma_{p'-1}\dots \sigma_{1})^{q'}$ is a generalized $d$-cabling of the core of $T'$. Furthermore, $T'$ doesn't intersect the braid axis of $K'$.
\end{proof}

If we consider $p, q, r, k$ to be positive integers such that $p>q \geq gcd(p, q)\geq r> 1$ and $\beta$ to be a positive braid with $r$ strands. Then, using a similar idea of the proof of Theorem~\ref{satellitecaseone}, then we can prove that 
the link given by the closure of the positive braid $\beta(\sigma_1\dots \sigma_{p-1})^{q + pk}$ is a generalized $gcd(p, q)$-cabling of a link. However, we want to change these conditions a bit to obtain non-satellite knots.

\begin{proposition}\label{propositionTori2}
Let $p, q, r, k$ be positive integers such that $p>q>r> gcd(p, q)> 1$ and $r$ has no common divisor with $p$ and $q$. Consider $\beta$ a positive braid with $r$ strands and at least one positive full twist on $r$ strands. Then, the knot $K$ given by the closure of the positive braid $\beta(\sigma_1\dots \sigma_{p-1})^{q + pk}$ has no essential torus.
\end{proposition}

\begin{proof} 
Denote by $K$ the closure of the braid $\beta(\sigma_1\dots \sigma_{p-1})^{q + pk}$. Suppose that $K$ has an essential torus $T$.

By Ito \cite[Theorem 5.5, Lemma 3.5]{ito2024satellite}, $T$ doesn't intersect the braid axis $C$ of $\beta(\sigma_1\dots \sigma_{p-1})^{q + pk}$, the core of $T$ is given by a positive braid with at least one full twist, and the knot inside is given by a positive braid. In particular, $K$ is a generalized $d$-cabling of the core of $T$. So, $d$ divides $p$.

By Lemma~\ref{lemmaTori1}, $d$ divides $p, q$ and there are positive numbers $p', q'$ multiples of $d$  with $p'>r>q'\geq d$ and a torus $T'$ such that the closure $K'$ of one, we denote it by $B'$, of the positive braids $\beta(\sigma_1\dots \sigma_{p'-1})^{q'}$ or $\beta(\sigma_{p'-1}\dots \sigma_{1})^{q'}$ is a generalized $d$-cabling of the core of $T'$. Furthermore, $T'$ doesn't intersect the braid axis of $K'$.
 
By Propositions~\ref{isopoty1} or ~\ref{isopoty1}, depending on whether $B'$ is $\beta(\sigma_1\dots \sigma_{p'-1})^{q'}$ or $\beta(\sigma_{p'-1}\dots \sigma_{1})^{q'}$, $B'$ is equivalent to a positive braid $B''$ with $r$ strands and at least one positive full twist on $r$ strands as $\beta$ is a positive braid with $r$ strands and at least one positive full twist on $r$. Thus, by \cite[Corollary 2.4]{Franks}, $B''$ has braid index equal to $r$. 

If the torus $T'$ is unknotted, then, by Lemma~\ref{lemma2}, $d$ is equal to $r$. But, this would imply that $r$ divides $q$ and $p$, which is not possible by hypotheses. So, $T'$ is knotted. Hence $T'$ doesn't intersect the braid axis of $B''$ by \cite{ito2024satellite}. Thus, $d$ also divides $r$. But this contradicts the hypothesis that $r$ has no common divisor with $p$ and $q$. Therefore, $K$ has no essential torus, as we want.
\end{proof}

\begin{proposition}\label{Atoroidalcase2} 
Let $p, q, r, k$ be positive integers such that $p>r > q$, $p-r\geq q$, and $r$ doesn't divide $p$. Furthermore, suppose that $gcd(r, q) = 1$, $gcd(p, q) = 1$, or $gcd(p, r) = 1$.  Let $\beta$ be a positive braid with $r$ strands and at least one full twist on $r$ strands.
Then, the knot given by the closure of the positive braid $$\beta(\sigma_{1}\dots \sigma_{p-1})^{kp+q}$$ is atoroidal.
\end{proposition}

\begin{proof} 
Denote by $K$ the closure of $B = \beta(\sigma_{1}\dots \sigma_{p-1})^{kp+q}$,

Suppose that $K$ has an essential torus. By Ito \cite[Theorem 5.5, Lemma 3.5]{ito2024satellite}, $T$ doesn't intersect the braid axis $C$ of $B$. Therefore, $K$ is a generalized $d$-cabling of the core of $T$. So, $d$ divides $p$.

After doing $(-1/k)$-Dehn surgery along $C$, the braid $B$ becomes the braid
$$B' = \beta(\sigma_{1}\dots \sigma_{p-1})^{q}$$ and $T$ becomes a new torus $T'$ in the exterior of the closure $K'$ of $B'$.

By Proposition~\ref{isopoty1}, there is an isotopy, which happens in the complement of the braid axis of $\beta$, that takes the closure of the braid $B'$ to the braid with $r$ strands
$$B''=(\sigma_{r-1}\dots \sigma_{r-q+1})^{p-r}\beta(\sigma_{1}\dots \sigma_{r-1})^{q}$$
if $q>1$ or
$B''=\beta(\sigma_{1}\dots \sigma_{r-1})$
if $q=1$.

Since $\beta$ is a positive braid with $r$ strands and a full twist on $r$ strands, it follows that $B''$ has braid index equal to $r$ by Franks and Williams \cite[Corollary 2.4]{Franks}.

If $T'$ is trivial, then, by Lemma~\ref{lemma2}, $K'$ has braid index equal to $d$. It implies that $d = r$. However, this is not possible as $r$ doesn't divide $p$. Hence $T'$ is knotted.

The knotted torus $T'$ doesn't intersect the braid axis of $B''$ by \cite{ito2024satellite}. So, $d$ divides $r$. 
Furthermore, $T'$ doesn't intersect the circle $C_{r}$ encircling the strands of the sub-braid $\beta$ of $B'$ since $C_{r}$ is isotopic to the braid axis of $B''$. Then, we push the circle $C_{r}$ anticlockwise once around the braid closure of $B'$ so that it ends up between the 
sub-braids $\beta$ and $(\sigma_{1}\dots \sigma_{p-1})^{q}$ of $B'$ encircling from the $(q+1)$-th strand to the $(q+r)$-th strand as $p-r\geq q $. 
As $T'$ intersects in circles the circle $C_{r}$ in its initial position, which encircles from the $1$-th strand to the $r$-th strand, $T'$ also intersects in circles the circle $C_{q}$ encircling from the $1$-th strand to the $q$-th strand. So, $d$ divides $q$. This implies that $q>1$ otherwise $d$ would be equal to $1$, which is not possible. Hence $d$ divides $p, q$, and $r$.
But this contradicts the hypotheses. Therefore, $K$ is atoroidal.
\end{proof}

\subsection{Non-torus knots given by braids}
In this subsection we prove that under mild conditions knots given by positive braids with at least one full twist are not torus knots.

\begin{proposition}\label{propositionAnnuli}
Let $p, q, r, k$ be positive integers such that $p>q>r> gcd(p, q)> 1$ and $r$ has no common divisor with $p$ and $q$. Consider $\beta$ a positive braid with $r$ strands and at least one positive full twist on $r$ strands. Then, the knot $K$ given by the closure of the positive braid $\beta(\sigma_1\dots \sigma_{p-1})^{q + pk}$ is not torus knot.
\end{proposition}

\begin{proof}Consider that $K$ is a torus knot. Since the braid $B = \beta(\sigma_1\dots \sigma_{p-1})^{q + pk}$ contains at least one full twist on
$p$ strands, it follows from \cite[Corollary 2.4]{Franks} that the knot $K$ has braid index equal to $p$. 
So, $K$ is the $(p, pk + d)$-torus knot with $d>0$.
By Los \cite[Corollary 1.2]{Los}, there is an isotopy in the complement of the braid axis $C$ of $B$ that takes $B$ to the standard braid of the torus knot $T(p, pk + d)$, which is $(\sigma_1\sigma_2\dots\sigma_{p-1})^{pk + d}$. 

After $(-1/k)$-Dehn surgery along the braid axis $C$ of $B$, the knot $K$ becomes the knot $K'$ given by the closure of the braid $B' = \beta(\sigma_1\dots \sigma_{p-1})^{q}$ and the torus knot $T(p, pk + d)$  becomes the torus knot $T(p, d)$. 
So, $K'$ is the $(p, d)$-torus knot. 

By Proposition~\ref{isopoty3}, the braid $B'$ is equivalent to the braid
$$B''=(\sigma_{q-1}\dots \sigma_{1})^{p-q}\beta(\sigma_{1}\dots \sigma_{q-1})^{q}.$$
The braid $B''$ is a positive braid with $q$ strands and at least one positive full twist on $q$ strands. 
Therefore, its braid index is equal to $q$ by Franks and Williams \cite[Corollary 2.4]{Franks}. 

As $B'$ has braid index equal to $q$, $K'$ is also the $(q, c)-$torus knot with $c>0$. So, $q$ is equal to $p$ or $d$. By hypothesis, $p\neq q$. So $q =  d$. But this implies that $K'$ is the torus link with $gcd(p, q)$ link components, which is a contradiction. 
\end{proof}
 
\begin{proposition}\label{Annulicase2} 
Let $p, q, r, k$ be positive integers with $p>r > q$ and $\beta$ be a positive braid with $r$ strands and at least one full twist on $r$ strands.
Denote by $\mathcal{C}$ the number of crossings of the braid $\beta$.
If $\mathcal{C}$ is different from $(p-1)(r-q)$, then the knot $K$ given by the closure of the positive braid $$\beta(\sigma_{1}\dots \sigma_{p-1})^{kp+q}$$ is not a torus knot.
\end{proposition}

\begin{proof}Suppose that $K$ is a torus knot. The braid $B = \beta(\sigma_1\dots \sigma_{p-1})^{q + pk}$ contains at least one full twist on $p$ strands, this implies that the knot $K$ has braid index equal to $p$ \cite{Franks}. 
So, $K$ is the $(p, d)$-torus knot with $d>pk$.

By Los \cite[Corollary 1.2]{Los}, there is an isotopy in the complement of the braid axis $C$ of $B$ that takes $B$ to the standard braid of the torus knot $T(p, d)$, which is $(\sigma_1\sigma_2\dots\sigma_{p-1})^{d}$. 
After $(-1/k)$-Dehn surgery along the braid axis $C$ of $B$, the knot $K$ becomes the knot $K'$ given by the closure of the braid $B' = \beta(\sigma_1\dots \sigma_{p-1})^{q}$. 
So, the knot $K'$ is the $(p, d-pk)$-torus knot. 

By Proposition~\ref{isopoty1}, there is an isotopy that takes the closure of the braid $B'$ to the braid with $r$ strands
$$B''=(\sigma_{r-1}\dots \sigma_{r-q+1})^{p-r}\beta(\sigma_{1}\dots \sigma_{r-1})^{q}.$$
if $q>1$ and
$\beta(\sigma_{1}\dots \sigma_{r-1})$
if $q=1$.

Since $\beta$ is a positive braid with $r$ strands and a full twist on $r$ strands, it follows that $B''$ has braid index equal to $r$ by  \cite[Corollary 2.4]{Franks}.
Hence $K'$ is also the $(r, c)$-torus knot with $c>0$. So, $r$ is equal to $p$ or $d-pk$. By hypotheses, $p\neq r$. So, $r =  d-pk$. But this implies that $K'$ is the torus knot $T(p, r)$.
Furthermore, there is an isotopy in the complement of the braid axis $C$ of $B'$ that takes this braid to the standard braid of the torus knot $T(p, r)$. So the number of crossings of $B'$ should be equal to $(p-1)r$. As the number of crossings of $B'$ is equal to $(p-1)q + \mathcal{C}$ , where $\mathcal{C}$ is the number of crossings of the braid $\beta$, we have that  $(p-1)q + \mathcal{C} = (p-1)r$, which implies that $\mathcal{C} = (p-1)(r-q)$, a contradiction.
\end{proof}

\subsection{Hyperbolic knots given by braids}
In this subsection we find some new mild conditions on knots given by positive braids with at least one full twist to guarantee they are hyperbolic.

\begin{theorem}\label{Hyperbolicity1}
Let $p, q, r, k$ be positive integers with $p>q\geq r >1$ and $gcd(p, q)=1$. Consider $\beta$ a non-trivial positive braid with $r$ strands. Then, the knot $K$ given by the closure of the positive braid $$\beta(\sigma_1\dots \sigma_{p-1})^{q + pk}$$ is hyperbolic.
\end{theorem}

\begin{proof} 
The case $k\geq 2$ follows from de Paiva \cite[Theorem 1.2]{MR4494619}. The case $k\ge 1$ follows from the same idea of the proof of \cite[Theorem 1.2]{MR4494619} if we consider \cite[Theorem 5.5, Lemma 3.5]{ito2024satellite}.
\end{proof}

\begin{theorem}\label{Hyperbolicity2}
Let $p, q, r, k$ be positive integers such that $p>q>r> gcd(p, q)> 1$ and $r$ has no common divisor with $p$ and $q$. Consider $\beta$ a positive braid with $r$ strands and at least one positive full twist on $r$ strands. Then, the knot $K$ given by the closure of the positive braid $$\beta(\sigma_1\dots \sigma_{p-1})^{q + pk}$$ is hyperbolic.
\end{theorem}

\begin{proof}
The knot $K$ is given by a positive braid with $p$ strands and at least one full twist on $p$ strands. By Franks and Williams \cite[Corollary~2.4]{Franks}, its braid index is equal to $p>1$. Hence, it can't be a trivial knot.   
Furthermore, $K$ is not a satellite knot by Proposition~\ref{propositionTori2}. By Proposition~\ref{propositionAnnuli}, $K$ is not a torus knot either. 
Therefore, $K$ is a hyperbolic knot.
\end{proof}

\begin{theorem}\label{Hyperbolicity3}
Let $p, q, r, k$ be positive integers such that $p>r > q$, $p-r\geq q$, and $r$ doesn't divide $p$. Furthermore, suppose that $gcd(r, q) = 1$, $gcd(p, q) = 1$, or $gcd(p, r) = 1$.  
Let $\beta$ be a positive braid with $r$ strands and at least one full twist on $r$ strands with crossing number different from $(p-1)(r-q)$. Then, the knot $K$ given by the closure of the positive braid $$\beta(\sigma_{1}\dots \sigma_{p-1})^{kp+q}$$ is hyperbolic.
\end{theorem}

\begin{proof} 
The knot $K$ has braid index equal to $p$ \cite{Franks}. As $p>1$, $K$ is a non-trivial knot. Furthermore, 
it follows from Proposition~\ref{Atoroidalcase2} that $K$ is not a satellite knot and from Proposition~\ref{Annulicase2} that $K$ is not a torus knot. Therefore, $K$ is hyperbolic.
\end{proof}

\section{Hyperbolic V-links and T-links}\label{section7}

In this section we apply the conditions found in section 5 on positive braids with a full twist to ensure their closure are hyperbolic to prove that some V-links and T-links are hyperbolic.

\begin{theorem}\label{Hyperbolicityfirstcase}
Let $r_1, \dots, r_n, s_1, \dots, s_n, u_1, \dots, u_m, v_1, \dots, v_m, p, q, k$ be positive integers such that $2\leq r_1< \dots < r_{n} \leq q$, $2\leq u_1< \dots < u_{m} \leq q < p$, and $gcd(p, q)=1$. Then, the V-knot
$$V_1 = V((u_1,\overline{v_1}), \dots, (u_{m},\overline{v_{m}}), (r_1,s_1), \dots, (r_{n},s_{n}), (p, kp+q))$$ 
is hyperbolic.
\end{theorem}

\begin{proof} 
The V-knot $V_1$ is given by the braid
\begin{align*}
&(\sigma_{p-1}\sigma_{p-2}\dots\sigma_{{p-u_1+1}})^{v_1}\dots(\sigma_{p-1}\sigma_{p-2}\dots\sigma_{{p-u_m+1}})^{v_m}(\sigma_1\sigma_2\dots\sigma_{r_1-1})^{s_1}\dots(\sigma_{1}\sigma_{2}\dots\sigma_{r_n-1})^{s_n}\\
&(\sigma_{1}\sigma_{2}\dots\sigma_{{p-1}})^{kp+q}.
\end{align*}
We push the sub-braid provided by the parameters $(u_1,\overline{v_1}), \dots, (u_{m},\overline{v_{m}})$ anticlockwise once around the braid closure to obtain a braid of the form 
$\beta(\sigma_1\dots \sigma_{p-1})^{q + pk}$ with $\beta$ a positive braid with at most $q$ strands satisfying the hypotheses of Theorem~\ref{Hyperbolicity1}. Hence $V_1$ is a hyperbolic knot.
\end{proof}

\begin{corollary}\label{C1} 
Let $r_1, \dots, r_n, s_1, \dots, s_n, u_1, \dots, u_m, v_1, \dots, v_m, p, q, k$ be positive integers such that $2\leq r_1< \dots < r_{n} \leq q$, $2\leq u_1< \dots < u_{m} \leq q < p$, and $gcd(p, q)=1$. Then, the T-knots
\begin{align*}
&T( (r_1,s_1), \dots, (r_{n},s_{n}), (p, kp+q-u_m), (p+ v_m, u_m - u_{m-1}), \dots, (p+v_m+ \dots + v_{2}, u_2-u_1),\\
&(p+ v_m+\dots + v_{1}, u_1)) \text{ and } T((u_1, v_1), \dots, (u_{m}, v_{m}), (kp+q, p-r_n), (kp+q+ s_n, r_n - r_{n-1}),\\
&\dots, (kp+q+s_n+ \dots + s_{2}, r_2-r_1), (kp+q+ s_n+ \dots + s_{1}, r_1))
\end{align*}
are hyperbolic.
\end{corollary}

\begin{proof}
It follows from Theorems~\ref{Hyperbolicityfirstcase} and \ref{equivalence}.
\end{proof} 

In particular, Corollary~\ref{C1} generalizes \cite[Corollary 1.3]{MR4494619}.

\begin{theorem}\label{Hyperbolicitysecondcase}
Let $r_1, \dots, r_n, s_1, \dots, s_n, u_1, \dots, u_m, v_1, \dots, v_m, p, q, k$ be positive integers such that $2\leq r_1< \dots < r_{n} <q$, $2\leq u_1< \dots < u_{m} <q < p$, and $gcd(p, q)>1$.
\begin{itemize}
\item If  $s_{n}\geq r_{n}>gcd(p, q)$ and $gcd(p, q, r_n)=1$, then the V-knot
$$V_1 = V((r_1,s_1), \dots, (r_{n},s_{n}), (p, kp+q))$$
is hyperbolic.
\item If $v_{m}\geq u_{m}>gcd(p, q)$ and $gcd(p, q, u_{m})=1$ then the V-knot
$$V_2 = V((u_1,\overline{v_1}), \dots, (u_{m},\overline{v_{m}}), (p, kp+q))$$ 
is hyperbolic.
\end{itemize}
\end{theorem}

\begin{proof}
It directly follows from Theorem~\ref{Hyperbolicity2}.
\end{proof}

We can also obtain hyperbolic V-links $V((u_1,\overline{v_1}), \dots, (u_{m},\overline{v_{m}}), (r_1,s_1), \dots, (r_{n},s_{n}), (p, kp+q))$ from Theorem~\ref{Hyperbolicity2} for the parameters that it is possible to push the sub-braid provided by the parameters $(u_1,\overline{v_1}), \dots, (u_{m},\overline{v_{m}})$ clockwise around the braid closure so that it becomes above the sub-braid provided by the parameters $(r_1,s_1), \dots, (r_{n},s_{n})$ so that we obtain a braid of the form $\beta(\sigma_{1}\dots \sigma_{p-1})^{kp+q}$ with $\beta$ a positive braid with $r_{n}$ strands that satisfies the hypotheses of Theorem~\ref{Hyperbolicity2}. Similarly, in some cases it is possible to push the sub-braid provided by the parameters 
$(r_1,s_1), \dots, (r_{n},s_{n})$ anticlockwise around the braid closure so that it becomes below the sub-braid provided by the parameters $(u_1,\overline{v_1}), \dots, (u_{m},\overline{v_{m}})$ so that we obtain a braid of the form $\beta(\sigma_{1}\dots \sigma_{p-1})^{kp+q}$ with $\beta$ a positive braid with $u_{m}$ strands that satisfies the hypotheses of Theorem~\ref{Hyperbolicity2}.

\begin{corollary}\label{C2} 
Let $r_1, \dots, r_n, s_1, \dots, s_n, u_1, \dots, u_m, v_1, \dots, v_m, p, q, k$ be positive integers such that $2\leq r_1< \dots < r_{n} <q$, $2\leq u_1< \dots < u_{m} <q < p$, and $gcd(p, q)>1$.
\begin{itemize}
\item If  $s_{n}\geq r_{n}>gcd(p, q)$ and $gcd(p, q, r_n)=1$, then the T-knots
\begin{align*}
&T((r_1,s_1), \dots, (r_{n},s_{n}), (p, kp+q)) \text{ and } T((kp+q, p-r_n), (kp+q+ s_n, r_n - r_{n-1}), \dots, \\
&(kp+q+s_n+ \dots + s_{2}, r_2-r_1), (kp+q+ s_n+ \dots + s_{1}, r_1))
\end{align*} 
are hyperbolic.
\item If $v_{m}\geq u_{m}>gcd(p, q)$ and $gcd(p, q, u_{m})=1$, then the T-knots
\begin{align*}
&T((u_1, v_1), \dots, (u_{m}, v_{m}), (kp+q, p)) \text{ and }T((p, kp+q-u_m), (p+ v_m, u_m - u_{m-1}), \dots,\\
&(p+v_m+ \dots + v_{2}, u_2-u_1),(p+ v_m+\dots + v_{1}, u_1))
\end{align*}
are hyperbolic.
\end{itemize}
\end{corollary}

\begin{proof}
It follows from Theorems~\ref{Hyperbolicitysecondcase} and \ref{equivalence}.
\end{proof} 

\begin{theorem}\label{Hyperbolicitythirdcase}
Let $r_1, \dots, r_n, s_1, \dots, s_n, u_1, \dots, u_m, v_1, \dots, v_m, p, q, k$ be positive integers such that $2\leq r_1< \dots < r_{n} < p$, $2\leq u_1< \dots < u_{m} < p$, $q<p$ with $gcd(p, q) = 1$. Denote $\mathcal{C} = (r_1-1)s_1 + \dots + (r_n-1)s_n+(u_1-1)v_1 + \dots + (u_m-1)v_m$.
\begin{itemize}
\item Suppose $s_n\geq r_{n}> q\geq u_{m}$, $p-r_{n}\geq q$, $r_{n}$ doesn't divide $p$, $gcd(r_{n}, q) = 1$ or $gcd(p, r_{n}) = 1$, and $\mathcal{C}$ is different from
$(p-1)(r_n-q)$, or
\item Suppose $v_{m}\geq u_{m}>q\geq r_{n}$, $p-u_{m}\geq q$, $u_{m}$ doesn't divide $p$, $gcd(u_{m}, q) = 1$ or $gcd(p, u_{m}) = 1$, and $\mathcal{C}$ is different from
$(p-1)(u_{m}-q)$.
\end{itemize}
Then, the V-knot 
$$V = V((u_1,\overline{v_1}), \dots, (u_{m},\overline{v_{m}}), (r_1,s_1), \dots, (r_{n},s_{n}), (p, kp+q))$$ 
is hyperbolic.
\end{theorem}
\begin{proof}
The V-knot $V$ is given by the braid
\begin{align*}
&(\sigma_{p-1}\sigma_{p-2}\dots\sigma_{{p-u_1+1}})^{v_1}\dots(\sigma_{p-1}\sigma_{p-2}\dots\sigma_{{p-u_m+1}})^{v_m}(\sigma_1\sigma_2\dots\sigma_{r_1-1})^{s_1}\dots(\sigma_{1}\sigma_{2}\dots\sigma_{r_n-1})^{s_n}\\
&(\sigma_{1}\sigma_{2}\dots\sigma_{{p-1}})^{kp+q}.
\end{align*}

Consider first that the first set of conditions is satisfied.
First we push the sub-braid provided by the parameters $(u_1,\overline{v_1}), \dots, (u_{m},\overline{v_{m}})$ anticlockwise once around the braid closure so that it becomes between the sub-braids provided by the parameters $(r_{n},s_{n})$ and $(p, kp+q)$. After that, we obtain a braid of the form
$\beta(\sigma_{1}\dots \sigma_{p-1})^{kp+q}$, with $\beta$ a positive braid with $r_{n}$ strands and a full twist on $r_{n}$ strands, that satisfies the hypotheses of Theorem~\ref{Hyperbolicity3}. Hence $K$ is hyperbolic.

Next suppose that the second set of conditions is satisfied. 
We push the sub-braid provided by the parameters $(r_1,s_1), \dots, (r_{n},s_{n})$ clockwise once around the braid closure so that it becomes above the sub-braid provided by the parameters $(u_1,\overline{v_1}), \dots, (u_{m},\overline{v_{m}})$. This yields a braid of the form
$\beta'(\sigma_{1}\dots \sigma_{p-1})^{kp+q}$ with $\beta'$ a positive braid on the last $u_{m}$ strands of $(\sigma_{1}\dots \sigma_{p-1})^{kp+q}$ with a full twist on $u_{m}$ strands. Then, we push $p - u_{m}$ horizontal strands of $(\sigma_{1}\dots \sigma_{p-1})^{kp+q}$ above $\beta'$ and around the braid closure to place $\beta'$ in the first strands of $(\sigma_{1}\dots \sigma_{p-1})^{kp+q}$. This produces a braid of the form $\beta(\sigma_{1}\dots \sigma_{p-1})^{kp+q}$ with $\beta$ a positive braid with $u_{m}$ strands and a full twist on $u_{m}$ strands that satisfies the hypotheses of Theorem~\ref{Hyperbolicity3}. This implies that $K$ is hyperbolic in this case as well.
\end{proof} 

\begin{corollary}\label{hyperbolicT-links3}
Let $r_1, \dots, r_n, s_1, \dots, s_n, u_1, \dots, u_m, v_1, \dots, v_m, p, q, k$ be positive integers such that $2\leq r_1< \dots < r_{n} < p$, $2\leq u_1< \dots < u_{m} < p$, $q<p$ with $gcd(p, q) = 1$. Denote $\mathcal{C} = (r_1-1)s_1 + \dots + (r_n-1)s_n+(u_1-1)v_1 + \dots + (u_m-1)v_m$.
\begin{itemize}
\item Suppose $s_n\geq r_{n}> q\geq u_{m}$, $p-r_{n}\geq q$, $r_{n}$ doesn't divide $p$, $gcd(r_{n}, q) = 1$ or $gcd(p, r_{n}) = 1$, and $\mathcal{C}$ is different from
$(p-1)(r_n-q)$, or
\item Suppose $v_m\geq u_{m}>q\geq r_{n}$, $p-u_{m}\geq q$, $u_{m}$ doesn't divide $p$, $gcd(u_{m}, q) = 1$ or $gcd(p, u_{m}) = 1$, and $\mathcal{C}$ is different from
$(p-1)(u_{m}-q)$.
\end{itemize}
Then, the T-knots
\begin{align*}
&T( (r_1,s_1), \dots, (r_{n},s_{n}), (p, kp+q-u_m), (p+ v_m, u_m - u_{m-1}), \dots, (p+v_m+ \dots + v_{2}, u_2-u_1),\\
&(p+ v_m+\dots + v_{1}, u_1)) \text{ and } T((u_1, v_1), \dots, (u_{m}, v_{m}), (kp+q, p-r_n), (kp+q+ s_n, r_n - r_{n-1}),\\
&\dots, (kp+q+s_n+ \dots + s_{2}, r_2-r_1), (kp+q+ s_n+ \dots + s_{1}, r_1))
\end{align*}
are hyperbolic.
\end{corollary}

\begin{proof}
It follows from Theorems~\ref{Hyperbolicitythirdcase} and \ref{equivalence}.
\end{proof}

\bibliographystyle{amsplain}  

\bibliography{referencess}

\providecommand{\bysame}{\leavevmode\hbox to3em{\hrulefill}\thinspace}
\providecommand{\MR}{\relax\ifhmode\unskip\space\fi MR }
\providecommand{\MRhref}[2]{%
  \href{http://www.ams.org/mathscinet-getitem?mr=#1}{#2}
}
\providecommand{\href}[2]{#2}
\begin{thebibliography}{10}

\bibitem{newtwis}
Joan Birman and Ilya Kofman, \emph{A new twist on {L}orenz links}, J. Topol.
  \textbf{2} (2009), no.~2, 227--248. \MR{2529294}

\bibitem{MR1286930}
Joan~S. Birman and William~W. Menasco, \emph{Special positions for essential
  tori in link complements}, Topology \textbf{33} (1994), no.~3, 525--556.
  \MR{1286930}

\bibitem{MR682059}
Joan~S. Birman and R.~F. Williams, \emph{Knotted periodic orbits in dynamical
  systems. {I}. {L}orenz's equations}, Topology \textbf{22} (1983), no.~1,
  47--82. \MR{682059}

\bibitem{MR4494619}
Thiago de~Paiva, \emph{Hyperbolic knots given by positive braids with at least
  two full twists}, Proc. Amer. Math. Soc. \textbf{150} (2022), no.~12,
  5449--5458. \MR{4494619}

\bibitem{unexpected}
\bysame, \emph{Unexpected essential surfaces among exteriors of twisted torus
  knots}, Algebr. Geom. Topol. \textbf{22} (2022), no.~8, 3965--3982.
  \MR{4562562}

\bibitem{de2022hyperbolic}
\bysame, \emph{Hyperbolic twisted torus links}, Geom. Dedicata \textbf{217}
  (2023), no.~2, Paper No. 42, 16. \MR{4551666}

\bibitem{de2022torus}
\bysame, \emph{Torus {L}orenz links obtained by full twists along torus links},
  Proc. Amer. Math. Soc. \textbf{151} (2023), no.~6, 2671--2677. \MR{4576328}

\bibitem{Lorenzknots}
\bysame, \emph{Satellite knots that cannot be represented by positive braids
  with full twists}, New York J. Math. \textbf{31} (2025), 1690--1701.

\bibitem{de2021satellites}
Thiago de~Paiva and Jessica~S. Purcell, \emph{Satellites and {L}orenz knots},
  Int. Math. Res. Not. IMRN (2023), no.~19, 16540--16573. \MR{4651895}

\bibitem{dePaivaPurcell2024}
\bysame, \emph{Hyperbolic and satellite {L}orenz links obtained by twisting},
  Michigan Math. J. \textbf{73} (2024), no.~2, 379--404.

\bibitem{Necessary}
E.~A. El-Rifai, \emph{Necessary and sufficient condition for {L}orenz knots to
  be closed under satellite construction}, Chaos Solitons Fractals \textbf{10}
  (1999), no.~1, 137--146. \MR{1682295}

\bibitem{Franks}
John Franks and R.~F. Williams, \emph{Braids and the {J}ones polynomial},
  Trans. Amer. Math. Soc. \textbf{303} (1987), no.~1, 97--108. \MR{896009}

\bibitem{GomesFrancoSilva:Partial}
Paulo Gomes, Nuno Franco, and Lu\'{\i}s Silva, \emph{Partial classification of
  {L}orenz knots: syllable permutations of torus knots words}, Phys. D
  \textbf{306} (2015), 16--24. \MR{3367570}

\bibitem{GomesFrancoSilva:Farey}
\bysame, \emph{Farey neighbors and hyperbolic {L}orenz knots}, J. Knot Theory
  Ramifications \textbf{26} (2017), no.~9, 1743004, 14. \MR{3687479}

\bibitem{GuckenheimerWilliams}
John Guckenheimer and R.~F. Williams, \emph{Structural stability of {L}orenz
  attractors}, Inst. Hautes \'{E}tudes Sci. Publ. Math. (1979), no.~50, 59--72.
  \MR{556582}

\bibitem{ito2024satellite}
Tetsuya Ito, \emph{Satellite fully positive braid links are braided satellites
  of fully positive braid links}, Journal of Topology and Analysis (2025).

\bibitem{lorenz1963deterministic}
Edward~N Lorenz, \emph{Deterministic nonperiodic flow}, Journal of atmospheric
  sciences \textbf{20} (1963), no.~2, 130--141.

\bibitem{Los}
J\'{e}r\^{o}me~E. Los, \emph{Knots, braid index and dynamical type}, Topology
  \textbf{33} (1994), no.~2, 257--270. \MR{1273785}

\bibitem{Genus}
John~R. Stallings, \emph{Constructions of fibred knots and links}, Algebraic
  and geometric topology ({P}roc. {S}ympos. {P}ure {M}ath., {S}tanford {U}niv.,
  {S}tanford, {C}alif., 1976), {P}art 2, Proc. Sympos. Pure Math., XXXII, Amer.
  Math. Soc., Providence, R.I., 1978, pp.~55--60. \MR{520522}

\bibitem{Tucker}
Warwick Tucker, \emph{A rigorous {ODE} solver and {S}male's 14th problem},
  Found. Comput. Math. \textbf{2} (2002), no.~1, 53--117. \MR{1870856}

\bibitem{MR758900}
R.~F. Williams, \emph{Lorenz knots are prime}, Ergodic Theory Dynam. Systems
  \textbf{4} (1984), no.~1, 147--163. \MR{758900}

\bibitem{Williams}
\bysame, \emph{The braid index of generalized cables}, Pacific J. Math.
  \textbf{155} (1992), no.~2, 369--375. \MR{1178031}

\end{thebibliography}

\end{document}